\documentclass[11pt]{amsart}
\usepackage[margin=3cm]{geometry}

\usepackage[utf8]{inputenc}
\usepackage{mathtools,amsfonts,amsthm,mathrsfs,amssymb}
    \usepackage{stmaryrd}
\usepackage{textcomp}
\mathtoolsset{showonlyrefs}
\usepackage{bookmark}
\usepackage{microtype}
\usepackage{todonotes}
\usepackage{enumerate}

\makeatletter
\newtheorem*{rep@theorem}{\rep@title}
\newcommand{\newreptheorem}[2]{
\newenvironment{rep#1}[1]{
 \def\rep@title{#2 \ref{##1}}
 \begin{rep@theorem}}
 {\end{rep@theorem}}}
\makeatother

\newtheorem{thm}{Theorem}[section]
\newreptheorem{thm}{Theorem}
\newtheorem{lemma}[thm]{Lemma}
\newtheorem{clm}{Claim}
\newtheorem{proposition}[thm]{Proposition}
\newtheorem{corollary}[thm]{Corollary}

\newtheorem{obs}[thm]{Observation}
\theoremstyle{definition}
\newtheorem*{definition*}{Definition}

\newtheorem{qn}{Question}

\makeatletter
\newenvironment{claimproof}[1][\proofname]{\par\pushQED{\hfill$\lozenge$}\normalfont \topsep6\p@\@plus6\p@\relax \trivlist \item\relax{\itshape#1\@addpunct{.}}\hspace\labelsep\ignorespaces}{\popQED\endtrivlist\@endpefalse}
\makeatother

\newtheorem*{acknowledgement}{Acknowledgments}

\newcommand{\G}{\mathcal{G}}

\newcommand{\Z}{\mathbb{Z}}

\newcommand{\N}{\mathbb{N}}

\newcommand{\W}{W} 

\renewcommand{\le}{\leqslant}
\renewcommand{\leq}{\leqslant}
\renewcommand{\ge}{\geqslant}
\renewcommand{\geq}{\geqslant}

\def\F {{\mathcal F}}

\def\P {{\mathcal P}}
\def\X {{\mathcal X}}
\def\W {{\mathcal W}}
\def\L {{\mathcal L}}
\def\U {{\mathcal U}}
\def\cC {{\mathcal C}}

\def\ad {{\rm asdim}}

\title[Asymptotic Dimension of Minor-Closed Families]{Asymptotic Dimension of Minor-Closed Families and Assouad-Nagata Dimension of Surfaces}

\author[M. Bonamy]{Marthe Bonamy} \address{LaBRI, CNRS,
  Universit\'e de Bordeaux, Bordeaux, France}
\email{marthe.bonamy@u-bordeaux.fr}

\author[N. Bousquet]{Nicolas Bousquet}
\address{LIRIS, CNRS, Universit\'e Claude Bernard Lyon 1, Lyon, France.}
\email{nicolas.bousquet@univ-lyon1.fr}

\author[L. Esperet]{Louis Esperet}
\address{Laboratoire G-SCOP, CNRS, Univ. Grenoble Alpes, Grenoble, France.}
\email{louis.esperet@grenoble-inp.fr}

\author[C. Groenland]{Carla Groenland}
\address{Mathematical  Institute,  University  of  Oxford,  Oxford  OX2  6GG,  United  Kingdom.}
\email{carla.groenland@maths.ox.ac.uk}

\author[C.-H. Liu]{Chun-Hung Liu}
\address{Department of Mathematics,
  Texas A\&M University, College Station, TX 77843-3368, USA}
\email{chliu@math.tamu.edu}

\author[F.\ Pirot]{Fran\c{c}ois Pirot}
\address{Laboratoire G-SCOP, CNRS, Univ. Grenoble Alpes, Grenoble, France.}
\email{francois.pirot@grenoble-inp.fr}

\author[A. Scott]{Alex Scott}
\address{Mathematical  Institute,  University  of  Oxford,  Oxford  OX2  6GG,  United  Kingdom.}
\email{scott@maths.ox.ac.uk}

\thanks{This paper is a combination of parts of two manuscripts \cite{BBEGPS,Liu20} in the arXiv repository, where \cite{Liu20} has appeared as an extended abstract in a conference, and a number of results in the unpublished manuscript~\cite{BBEGPS} are strengthened to Assouad-Nagata dimension in this paper.
M. \ Bonamy and N.\ Bousquet are supported by ANR Projects DISTANCIA (\textsc{ANR-17-CE40-0015}) and GrR (\textsc{ANR-18-CE40-0032}). 
L.\ Esperet and F. Pirot are supported by ANR Projects GATO
(\textsc{ANR-16-CE40-0009-01}) and GrR
(\textsc{ANR-18-CE40-0032}). C.-H. Liu is partially supported by NSF under Grant No. DMS-1929851 and DMS-1954054.}

\begin{document}
\begin{abstract}
The asymptotic dimension is an invariant of metric spaces introduced by Gromov in the context of geometric group theory. 
In this paper, we study the asymptotic dimension of metric spaces generated by graphs and their shortest path metric and show their applications to some continuous spaces.
The asymptotic dimension of such graph metrics can be seen as a large scale generalisation of weak diameter network decomposition which has been extensively studied in computer science.

We prove that every proper minor-closed family of graphs has asymptotic dimension at most 2, which gives optimal answers to a question of Fujiwara and Papasoglu and (in a strong form) to a problem raised by Ostrovskii and Rosenthal on minor excluded groups. 
For some special minor-closed families, such as the class of graphs embeddable in a surface of bounded Euler genus, we prove a stronger result and apply this to show that complete Riemannian surfaces have Assouad-Nagata dimension 
at most 2.
Furthermore, our techniques allow us to determine the asymptotic dimension of graphs of bounded layered treewidth and graphs  with any fixed growth rate, which are graph classes that are defined by purely combinatorial notions and properly contain graph classes with some natural topological and geometric flavours.
\end{abstract}
\maketitle

\section{Introduction}\label{sec:intro}

\subsection{Asymptotic dimension}

Asymptotic dimension of metric spaces was introduced by Gromov~\cite{Gr93} in the context of geometric group theory.
Every metric space can be realized by a graph whose edges are weighted.
In this paper, we address connections between structural graph theory and the theory of asymptotic dimension.
In particular, we solve open problems in coarse geometry and geometric group theory via tools from structural graph theory.

Let $(X,d)$ be a pseudometric space, and let $\mathcal{U}$ be a family of subsets of
$X$. We say that $\mathcal{U}$ is \emph{$D$-bounded} if each set $U\in \mathcal{U}$ has diameter at
most $D$. We say that $\mathcal{U}$ is \emph{$r$-disjoint} if for any
$a,b$ belonging to different elements of $\mathcal{U}$ we have
$d(a,b)> r$.

We say that $D_X:\mathbb{R}^+\to
\mathbb{R}^+$ is an
\emph{$n$-dimensional control function} for $(X,d)$ if for any $r>0$, $(X,d)$ has a cover
$\mathcal{U}=\bigcup_{i=1}^{n+1}\mathcal{U}_i$, such that each
$\mathcal{U}_i$ is $r$-disjoint and each element of $\mathcal{U}$ is $D_X(r)$-bounded. 
The \emph{asymptotic dimension} of $(X,d)$, denoted by $\mathrm{asdim}\, (X,d)$, is the
least integer $n$ such that $(X,d)$ has an $n$-dimensional control function. If no such integer $n$ exists, then the asymptotic
dimension is infinite. 
 
The reader is referred to~\cite{BD08} for a
survey on asymptotic dimension and its group theoretic applications,
and to the lecture notes of Roe~\cite{Roe03} on coarse geometry for more detailed proofs of some results of~\cite{Gr93}.

As the asymptotic dimension of a bounded space is 0, we are more interested in the asymptotic dimension of infinite pseudometric spaces or (infinite) families of (infinite or finite)  pseudometric spaces.
We define the \emph{asymptotic dimension of a family} $\mathcal{X}$ of pseudometric spaces as the least $n$ for which there exists a function $D_{\mathcal{X}}:\mathbb{R}^+\to \mathbb{R}^+$ which is an $n$-dimensional control function for  all $X\in \mathcal{X}$.

\subsection{Graphs as (pseudo)metric spaces}

A \emph{weighted graph} $(G,\phi)$ consists of a graph $G$ and a function $\phi: E(G) \rightarrow {\mathbb R}^+$.
We call $\phi(e)$ the \emph{weight} of $e$ for each $e \in E(G)$. 
The \emph{length in $(G,\phi)$} of a path $P$ in $G$ is the sum of the weights of the edges of $P$.
Given two vertices $x,y \in V(G)$, we define $d_{(G,\phi)}(x,y)$ to be the infimum of the length in $(G,\phi)$ of a path between $x$ and $y$; we define $d_{(G,\phi)}(x,y)=\infty$ if there exists no path between $x$ and $y$.
Notice that $(V(G),d_{(G,\phi)})$ is a pseudometric space. 

In this paper, we do not distinguish a weighted graph $(G,\phi)$ from the pseudometric space $(V(G),d_{(G,\phi)})$.
Therefore, the \emph{asymptotic dimension} of a weighted graph is the asymptotic dimension of this pseudometric space.

Given an (unweighted) graph $G$, it can be viewed as a weighted graph in which each edge has weight 1.
In particular, the asymptotic dimension of a graph is defined.

The main goal of our present work is to determine the asymptotic dimension of various (classes of  weighted or unweighted) graphs.
As observed above, a finite graph has asymptotic dimension
0, so this 
is only interesting for infinite graphs, or for infinite classes of (finite or infinite) graphs.
On the other hand, every metric space can be realised by a weighted complete graph, so it will be more interesting to study graphs with restricted structure.

Indeed, the asymptotic dimension of (unweighted) graphs  with structure restriction has attracted wide attention.
For example, Gromov considered the asymptotic dimension of Cayley graphs of groups (see Section~\ref{sec:cayley} for more details).
In addition, Gromov~\cite{Gr93} observed 
that $d$-dimensional Euclidean spaces have asymptotic
dimension $d$, and it can easily be deduced from this that for any $d\ge 1$, the class
of $d$-dimensional grids (with or without diagonals) has asymptotic dimension $d$. 
On the other hand, it is implicit in the work of Gromov~\cite{Gr03} (see also~\cite[Proposition 11.29]{Roe03}) that  any infinite family of bounded degree
expanders (in particular cubic expanders) has unbounded asymptotic
dimension.
A different proof, using vertex expansion instead of spectral expansion, is given in~\cite{Hum17}. 
Another example of a class of graphs with bounded degree
and infinite asymptotic dimension is the class of lamplighter
graphs of binary trees~\cite{BMSZ20} (these graphs have maximum degree 4). This implies that bounding the degree is not enough to bound the asymptotic dimension.

The asymptotic dimension of graphs is studied in several research areas. 
It is a large-scale generalisation of weak diameter network decomposition which has been studied in distributed computing (see Section~\ref{sec:clustered} for more details); a more refined notion of asymptotic dimension is called Assouad-Nagata dimension and its algorithmic form is related to weak sparse partition schemes, which are studied in theoretical computer science (see Sections~\ref{subsec:AN_dim} and \ref{sec:sparsecover} for more details).

A simple compactness argument (see Theorem~\ref{compact_graphs}) shows that the asymptotic dimension of an infinite weighted graph is at most the asymptotic dimension of the class of its finite induced weighted subgraphs\footnote{For a graph $G$ and a subset $S$ of $V(G)$, the \emph{subgraph of $G$ induced by $S$} is the graph, denoted by $G[S]$, whose vertex-set is $S$ and whose edge-set consists of the edges of $G$ with both ends in $S$. For a weighted graph $(G,\phi)$ and a subset $S$ of $V(G)$, the \emph{(weighted) subgraph of $(G,\phi)$ induced by $S$} is the weighted graph $(G[S],\phi|_{E(G[S])})$.}.
Hence in this paper, we only consider the asymptotic dimension of classes of finite weighted or unweighted graphs.

From now on, graphs, and more generally weighted graphs, are finite, unless stated otherwise. Note that the pseudometric space generated by a finite weighted or unweighted graph is a metric space.

\subsection{Minor-closed classes of graphs}
A (finite) graph $H$ is a \emph{minor} of a  finite or infinite graph $G$ if it can be obtained from a subgraph of $G$ by contracting edges. 
We say that a class $\mathcal{G}$ of graphs is \emph{minor-closed} if any minor of a graph from $\mathcal{G}$ is also in $\mathcal{G}$. A minor-closed class $\mathcal{G}$ is \emph{proper} if it does not contain all graphs.
In particular, for every proper minor-closed family ${\mathcal G}$, there exists a graph $H$ such that  $H \not \in {\mathcal G}$, so ${\mathcal G}$ is a subclass of the class of $H$-minor free graphs. 

Minor-closed classes are a far-reaching generalisation of classes of graphs with some geometric or topological properties.
For example, any minor of a graph embeddable in a surface\footnote{A \emph{surface} is a non-null connected 2-dimensional manifold without boundary.} $\Sigma$ is also embeddable in $\Sigma$, and thus classes of graphs embeddable in a fixed surface form a natural example of a proper minor-closed class. 
Other well-known examples of minor-closed classes include the class of linkless embeddable graphs and the class of knotless embeddable graphs.
Moreover, a deep result in graph theory (the Graph Minor Theorem~\cite{RS04}) states that for every proper minor-closed class $\mathcal G$, there exist finitely many graphs $H_1,H_2,...,H_n$ such that $\mathcal G$ is the intersection of the classes of $H_i$-minor free graphs.
In order to prove this result, Robertson and Seymour obtained a structural description of any proper minor-closed class, which will be instrumental in the proof of Theorem~\ref{thm:minorch} below.

Ostrovskii and Rosenthal~\cite{OR15} proved that for every integer $t$, the class of $K_t$-minor free graphs has asymptotic dimension at most $4^t$.
Recently, Fujiwara and Papasoglu~\cite{FP20} proved that the class of planar graphs, which is one of the most extensively studied minor-closed classes, has asymptotic dimension at most 3 and asked whether there exists a uniform upper bound on the asymptotic dimension of proper minor-closed families, as follows. 

\begin{qn}[Question 5.2 in~\cite{FP20}]\label{qn:fp}
Is there a constant $k$ such that for any graph $H$, the class of $H$-minor free graphs has asymptotic dimension at most $k$? 
Can we take $k=2$?
\end{qn}

Question~\ref{qn:fp} is a possible common strengthening of their result and the aforementioned result in \cite{OR15}. 
One of the main results of this paper is a complete solution of Question~\ref{qn:fp}.

\begin{thm}\label{thm:minorch}
For any graph $H$, the class of $H$-minor free graphs has asymptotic dimension at most 2.
In particular, every proper minor-closed class has asymptotic dimension at most 2.
\end{thm}

The bound in Theorem \ref{thm:minorch} is optimal if $H$ is a non-planar graph, since the class of 2-dimensional grids has asymptotic dimension 2~\cite{Gr93} and is a subclass of planar graphs.
When $H$ is planar, we prove that the  bound for asymptotic dimension can be further reduced to 1.

\begin{thm} \label{tw_ad_intro_1}
For every planar graph $H$, the asymptotic dimension of the class of $H$-minor free graphs is at most 1.
In particular, every proper minor-closed class that does not contain all planar graphs has asymptotic dimension at most 1.
\end{thm}

The bound in Theorem \ref{tw_ad_intro_1} is optimal.
Bell and Dranishnikov \cite{BD08} and Fujiwara and Papasoglu \cite{FP20}, respectively, showed that the class of trees and the class of cacti, respectively, have asymptotic dimension 1.
Note that trees are $K_3$-minor free and cacti are $K_{2,3}$-minor free.

By a simple compactness argument (Theorem~\ref{compact_graphs}), we obtain the following immediate corollary of Theorems~\ref{thm:minorch} and \ref{tw_ad_intro_1}.

\begin{corollary} \label{minor_infinite_intro}
Let $\F$ be a class of finite or infinite graphs. 
    \begin{enumerate}
        \item For any (finite) graph $H$, if no member of $\F$ contains $H$ as a minor, then $\ad(\F) \leq 2$.
        \item For any (finite) planar graph $H$, if no member of $\F$ contains $H$ as a minor, then $\ad(\F) \leq 1$.
        \item If $\F$ is minor-closed and does not contain all (finite) graphs, then $\ad(\F) \leq 2$.
        \item If $\F$ is minor-closed and does not contain all (finite) planar graphs, then $\ad(\F) \leq 1$.
    \end{enumerate}
\end{corollary}

As we will elaborate in Section~\ref{subsec:AN_dim}, we are able to strengthen the statement in Theorem~\ref{thm:minorch}  for some graphs $H$ so that it holds for weighted graphs and for a more refined notion of dimension which leads to results about Assouad-Nagata dimension for Riemannian surfaces.

\subsection{Treewidth and layered treewidth}

A {\it tree-decomposition} of a graph $G$ is a pair $(T,\X)$ such that $T$ is a tree and $\X$ is a collection $(X_t: t \in V(T))$ of subsets of $V(G)$, called the {\it bags}, such that
	\begin{itemize}
		\item $\bigcup_{t \in V(T)}X_t = V(G)$,
		\item for every $e \in E(G)$, there exists $t \in V(T)$ such that $X_t$ contains the ends of $e$, and
		\item for every $v \in V(G)$, the set $\{t \in V(T): v \in X_t\}$ induces a connected subgraph of $T$.
	\end{itemize}
For a tree-decomposition $(T,\X)$, the {\it adhesion} of $(T,\X)$ is $\max_{tt' \in E(T)}\lvert X_t \cap X_{t'} \rvert$, and the {\it width} of $(T,\X)$ is $\max_{t \in V(T)}\lvert X_t \rvert-1$.
The {\it treewidth} of $G$ is the minimum width of a tree-decomposition of $G$.

By the Grid Minor Theorem \cite{RS86}, excluding any planar graph as a minor is equivalent to having bounded treewidth.
Hence we have Theorem \ref{tw_ad_intro_2} which is an equivalent form of Theorem \ref{tw_ad_intro_1}. 

\begin{thm} \label{tw_ad_intro_2}
For every positive integer $w$, the asymptotic dimension of the class of graphs of treewidth at most $w$ is at most 1.
\end{thm}
The bound in Theorem \ref{tw_ad_intro_2} is optimal since  it is easy to see that every path has treewidth 1 and the class of paths has asymptotic dimension at least 1. 
Note that the special case of Theorem~\ref{tw_ad_intro_2} with the  additional   assumption of bounded maximum degree, which was proved in~\cite{BBEGPS}, can also be deduced from the work of Benjamini, Schramm and Tim\'ar~\cite{BST10}.

A {\it layering} of a graph $G$ is an ordered partition $(V_i)_{i\in \N}$ or $(V_i)_{i\in \Z}$  of $V(G)$ into (possibly empty) subsets $V_i$ such that for every edge $e$ of $G$, there exists $i_e$ such that $V_{i_e} \cup V_{i_e+1}$ contains both ends of $e$.
We call each set $V_i$ a {\it layer}.
The {\it layered treewidth} of a graph $G$ is the minimum $w$ such that there exist a tree-decomposition of $G$ and a layering of $G$ such that the size of the intersection of any bag and any layer is at most $w$.

Layered treewidth is a common generalisation of treewidth and Euler genus of graphs.
A number of classes of graphs with some geometric properties have bounded layered treewidth.
For example, Dujmovi\'{c}, Morin and Wood \cite{DMW17} showed that for every nonnegative integer $g$, graphs that can be embedded in a surface of Euler genus at most $g$ have layered treewidth at most $2g+3$.
Moreover, the 2-dimensional grids with all diagonals have layered treewidth two while having unbounded treewidth and unbounded Euler genus as they can contain arbitrarily large complete minors.

Combining Theorem \ref{tw_ad_intro_2} with machinery developed by Brodskiy, Dydak, Levin and Mitra~\cite{BDLM}, we obtain the following result for graphs of bounded layered treewidth.

\begin{thm} \label{layered_tw_ad_intro}
For every positive integer $w$, the asymptotic dimension of the class of graphs of layered treewidth at most $w$ is at most 2.
\end{thm}

In fact, classes of graphs of bounded layered treewidth are of interest beyond minor-closed families.
Let $g,k$ be nonnegative integers.
A graph is {\it $(g,k)$-planar} if it can be drawn in a surface of Euler genus at most $g$ with at most $k$ crossings on each edge. 
So $(g,0)$-planar graphs are exactly the graphs of Euler genus at most $g$.
It is well-known that $(0,1)$-planar graphs (also known as {\it 1-planar graphs} in the literature) can contain an arbitrary graph as a minor.
So the class of $(g,k)$-planar graphs is not a minor-closed family.
On the other hand, Dujmovi\'{c}, Eppstein and Wood \cite{DEW17} proved that $(g,k)$-planar graphs have layered treewidth at most $(4g+6)(k+1)$.

Hence the following is an immediate corollary of Theorem \ref{layered_tw_ad_intro}. 

\begin{corollary} \label{gk_planar_ad_intro}
For any nonnegative integers $g$ and $k$, the class of $(g,k)$-planar graphs has asymptotic dimension at most 2.
\end{corollary}

Recall that Corollary \ref{gk_planar_ad_intro} (and hence Theorem \ref{layered_tw_ad_intro}) is optimal since the class of 2-dimensional grids has asymptotic dimension 2.
Other extensively studied graph classes that are known to have bounded layered treewidth include map graphs \cite{CGP02,DEW17} and string graphs with bounded maximum degree \cite{DJMNW18}.
We refer readers to \cite{DEW17,DJMNW18} for discussion of those graphs.

One weakness of layered treewidth is that adding apices can increase layered treewidth a lot.
Note that for any vertex $v$ in a graph $G$ and for any layering of $G$, the neighbours of $v$ must be contained in the union of three consecutive layers.
So if a graph has bounded layered treewidth, then the subgraph induced by the neighbours of any fixed vertex must have bounded treewidth.
However, consider the graphs that can be obtained from 2-dimensional grids by adding a new vertex adjacent to all other vertices.
Since 2-dimensional grids can have arbitrarily large treewidth, such graphs cannot have bounded layered treewidth.

In contrast to the fragility of layered treewidth about adding apices, we show that adding a bounded number of apices does not increase the asymptotic dimension.
Let $\F$ be a class of graphs.
For every nonnegative integer $n$, define $\F^{+n}$ to be the class of graphs such that for every $G \in \F^{+n}$, there exists $Z \subseteq V(G)$ with $\lvert Z \rvert \leq n$ such that $G-Z \in \F$.

\begin{thm} \label{apex_extension_ad_intro}
For every class of graphs $\F$ and nonnegative integer $n$,  the asymptotic dimension of $\F^{+n}$ equals the asymptotic dimension of $\F$.
\end{thm}

This leads to the following strengthening of Theorem \ref{layered_tw_ad_intro}.

\begin{corollary} \label{apex_layered_tw_ad_intro}
Let $k$ be a nonnegative integer.
Let $w$ be a positive integer.
Let $\F$ be a class of graphs such that for every $G \in \F$, there exists $Z \subseteq V(G)$ with $\lvert Z \rvert \leq k$ such that $G-Z$ has layered treewidth at most $w$.
Then the asymptotic dimension of $\F$ is at most 2.
\end{corollary}

\subsection{Assouad-Nagata dimension} \label{subsec:AN_dim}

Gromov~\cite{Gr93} noticed that the notion of asymptotic dimension of a metric space can be refined by restricting the growth rate of the control function in its definition.
Although this function can be chosen to be linear in many cases, its complexity can be significantly worse in general, as there exist (Cayley) graphs of asymptotic dimension $n$ for which any $n$-dimensional control function $D(r)$ grows as fast as $\Omega(\exp \exp \cdots \exp r^k)$, for any given height of the tower of exponentials~\cite{Now07}.

A particularly interesting refinement of the asymptotic dimension is the Assouad-Nagata dimension introduced by Assouad~\cite{As82} (see~\cite{LS05} for more results on this notion).
A control function $D_X$ for a metric space $X$ is said to be
a \emph{dilation} if there is a constant $c>0$ such that $D_X(r)\le cr$, for any $r>0$.
A metric space $(X,d)$ has \emph{Assouad-Nagata dimension at
 most $n$} if $X$ has an $n$-dimensional control function which is a dilation. 
The definition extends to families of metric spaces in a natural way.
Observe that the Assouad-Nagata dimension is at least the asymptotic dimension.

We prove that when $H=K_{3,p}$ (for any fixed $p>0$), Theorem \ref{thm:minorch} can be extended to weighted graphs as well as in the setting of Assouad-Nagata dimension.

\begin{thm}\label{thm:main}
 For any integer $p>0$, the class of weighted  (finite or infinite) graphs excluding the complete bipartite graph $K_{3,p}$ as a minor has Assouad-Nagata dimension at most 2.
\end{thm}

When $p \geq 3$, 2-dimensional grids are $K_{3,p}$-minor free, so the bound in Theorem \ref{thm:main} is optimal.
In addition, for any fixed integer $g\ge 0$, the class of graphs embeddable in a surface of Euler genus
$g$ excludes $K_{3,2g+3}$ as a minor.
So Theorem \ref{thm:main} immediately implies the following result. 

\begin{corollary} \label{cor:genus_intro}
For any integer $g\ge 0$, the class of weighted  (finite or infinite) graphs embeddable in 
a surface of Euler genus $g$ has Assouad-Nagata dimension 2. 
\end{corollary}

Corollary \ref{cor:genus_intro} can be used for proving the following result about complete 2-dimensional connected Riemannian manifolds without boundary.

\begin{thm} \label{cor:manifold_intro}
The Assouad-Nagata dimension of any complete Riemannian surface of finite  Euler genus is at most 2.
\end{thm}

Note that Corollary \ref{cor:genus_intro} and Theorem \ref{cor:manifold_intro} extend results of J\o rgensen and Lang \cite{JL20} on the plane.

Similar to the notion of dilation, a control function $D_X$ for a metric space $X$ is said to be {\it linear} if there is a constant $c>0$ such that $D_X(r)\le cr+c$ for any $r>0$.
We say that a metric space $(X,d)$ has \emph{asymptotic dimension at most $n$ of linear type} if $X$ has a linear $n$-dimensional control function. 
The definition extends to families of metric spaces in a natural way. 
This notion is sometimes called \emph{asymptotic dimension with Higson property}~\cite{DZ04}.  
Obviously, the asymptotic dimension of linear type is between the asymptotic dimension and the Assouad-Nagata dimension.

Nowak~\cite{Now07} proved that the asymptotic dimension of linear type is not bounded by any function of the asymptotic dimension, by constructing (Cayley) graphs of asymptotic dimension 2 and infinite asymptotic dimension of linear type.
We give another such example by showing that some classes of graphs of bounded layered treewidth do not have bounded asymptotic dimension of linear type, even though Theorem \ref{layered_tw_ad_intro} shows that their asymptotic dimension is at most 2.

\begin{thm}\label{thm:nofptw_intro}
There is no integer $d$ such that the class of graphs of layered treewidth at most 1 has asymptotic dimension of linear type at most $d$.
\end{thm}

So Theorem~\ref{thm:nofptw_intro} states that Theorem~\ref{layered_tw_ad_intro} cannot be extended to asymptotic dimension of linear type or Assouad-Nagata dimension, even if we replace 2 by an arbitrary constant. This is in contrast with the case of $K_{3,p}$-minor free graphs and graphs with bounded Euler genus, which have bounded layered treewidth \cite{DMW17}, yet   have Assouad-Nagata dimension at most 2 by Theorem~\ref{thm:main} and Corollary~\ref{cor:genus_intro}.

\subsection{Growth rate}

For any function $f$, a graph $G$ has \emph{growth} at most $f$ if for any integer $r$, any vertex $v\in V(G)$ has at most $f(r)$ vertices at distance at most $r$.  
Similarly, we say that a class of graphs has growth at most $f$ if all graphs in this class have growth at most $f$.
It is known that vertex-transitive graphs of polynomial growth have bounded asymptotic dimension, while some classes of graphs of exponential growth have unbounded asymptotic dimension~\cite{Hum17}. 

We prove the following result that not only removes the vertex-transitive requirement but also shows that polynomials are the fastest-growing growth rate that ensures finite asymptotic dimension.

\begin{thm}\label{thm:polygrowth_intro}
The following holds.
    \begin{enumerate}
        \item For any polynomial $f$, there exists $d \in {\mathbb N}$, such that the class of graphs of growth at most $f$ has bounded asymptotic dimension at most $d$. 
        \item For any superpolynomial function\footnote{We say that a function $f$ is \emph{superpolynomial} if it can be written as $f(r)=r^{g(r)}$ with $g(r)\to \infty$ when $r\to \infty$.} $f$ with $f(r) \geq 3r+1$ for every $r \in \mathbb{N}$, the class of graphs of growth at most $f$ has infinite asymptotic dimension.
        \item For any function $f: {\mathbb N} \rightarrow {\mathbb N}$ for which there exists $r_0 \in {\mathbb N}$ with $f(r_0) \leq 3r_0$,  the class of graphs of growth at most $f$ has asymptotic dimension at most 1.
    \end{enumerate}
\end{thm}

We remark that even though the growth rate is a pure combinatorial property, graphs with polynomial growth are closely related to bounded dimensional grids. See Section~\ref{sec:geom} for more details.

\subsection{Application 1: Asymptotic dimension of groups}\label{sec:cayley}

Given a (finite or infinite) group $\Gamma$ and a generating set $S$ (assumed to be symmetric, in the sense that $s\in S$ if and only if $s^{-1}\in S$), the \emph{Cayley graph} $\text{Cay}(\Gamma,S)$ is the (possible infinite) graph with vertex-set $ \Gamma$, with an edge between two vertices $u,v\in  \Gamma$ if and only if $u=vs$ for some $s\in S$. 
As observed by Gromov~\cite{Gr93},  when $\Gamma$ is finitely generated, the asymptotic dimension of $\text{Cay}(\Gamma,S)$ is independent of the choice
of the finite generating set $S$, and thus the asymptotic
dimension is a group invariant  for finitely generated groups. The {\it asymptotic dimension of a finitely generated group $\Gamma$} is defined to be the asymptotic dimension of $\text{Cay}(\Gamma,S)$ for some symmetric finite generating set $S$.

We say that $\text{Cay}(\Gamma,S)$ is \emph{minor excluded} if it is $H$-minor free for some (finite) graph $H$. 

\begin{qn}[Problem 4.1 in~\cite{OR15}]\label{qn:or}
  Let $\Gamma$ be a finitely generated group and $S$ a finite generating set such that $\text{Cay}(\Gamma,S)$ is minor excluded. 
  Does it follow that $\Gamma$ has asymptotic dimension at most 2?
\end{qn}

An immediate corollary of Statement 1 of Corollary~\ref{minor_infinite_intro} gives a positive answer to Question~\ref{qn:or}, even when the group is not finitely generated.

\begin{corollary} \label{cor:Cay_minor}
Let $\Gamma$ be a group and $S$ a symmetric (not necessarily finite) generating set such that $\text{Cay}(\Gamma,S)$ is minor excluded. 
Then $\text{Cay}(\Gamma,S)$ has asymptotic dimension at most 2.
Furthermore, if $S$ is finite, then $\Gamma$ has asymptotic dimension at most 2. 
\end{corollary}

Note that Corollary \ref{cor:Cay_minor} is actually stronger. 
The finitely generated condition for $\Gamma$ is only used for ensuring that the asymptotic dimension of $\Gamma$ is independent of the choice of $S$.
We remark that the choice of the generating set $S$ of $\Gamma$ affects whether $\text{Cay}(\Gamma,S)$ is $H$-minor free or not. 

Similarly, Corollary \ref{gk_planar_ad_intro} leads to the following immediate corollary. 

\begin{corollary} \label{cor:Cay_gk}
For every group $\Gamma$ with a symmetric (not necessarily finite) generating set $S$, if there exist nonnegative integers $g$ and $k$ such that the Cayley graph for $(\Gamma,S)$ is $(g,k)$-planar, then the asymptotic dimension of $\text{Cay}(\Gamma,S)$ is at most 2.
In particular, if $\Gamma$ is finitely generated, then $\Gamma$ has asymptotic dimension at most 2.
\end{corollary}

\subsection{Application 2: Sparse partitions} \label{sec:sparsecover}

A \emph{ball of radius $r$} (or \emph{$r$-ball}) centered in a point $x$ in a metric space $X$, denoted by $B_r(x)$, is the set of points of $X$ at distance at
most $r$ from $x$.
For a real $r\ge 0$ and an integer $n\ge 0$,  a family $\mathcal{U}$ of subsets of elements of $X$ has \emph{$r$-multiplicity}
at most $n$ if each $r$-ball in $X$ intersects at most $n$ sets of $\mathcal{U}$. 
It is not difficult to see that if $D_X(r)$ is an $n$-dimensional control function for a metric space $X$, then for any $r>0$, $X$ has a $D_X(2r)$-bounded cover of $r$-multiplicity at most $n+1$.
Gromov~\cite{Gr93} proved that a converse of this result also holds, in the sense that the asymptotic dimension of $X$ is exactly the least integer $n$ such that for any real number $r > 0$, there is a real number $D_X'(r)$ such that $X$ has a $D_X'(r)$-bounded cover of $r$-multiplicity at most $n+1$. Moreover, the function $D_X'$ has the same type as the
$n$-dimensional control function $D_X$ of $X$: $D_X'$ is linear if and only if $D_X$ is linear, and $D_X'$ is a dilation if and only if $D_X$ is a dilation.

As a consequence, the notion of Assouad-Nagata dimension is closely related to the well-studied notions of sparse covers and sparse partitions in theoretical computer science. 
A weighted graph
$G$ admits a \emph{$(\sigma,\tau)$-weak sparse partition scheme} if for any $r \geq 0$,  the vertex-set of $G$ has a partition into $(\sigma\cdot r)$-bounded sets of $r$-multiplicity at most $\tau$, and such a partition can be computed in polynomial time. As before, we say
that a family of graphs admits a $(\sigma,\tau)$-weak sparse partition
scheme if all graphs in the family admit a $(\sigma,\tau)$-weak sparse partition scheme. This definition was
introduced in~\cite{JLNRS05}, and is equivalent to
the notion of weak sparse cover scheme of Awerbuch and
Peleg~\cite{AP90} (see~\cite{Fil20}). Note that if a family of graphs
admits a $(\sigma,\tau)$-weak sparse partition scheme then its
Assouad-Nagata dimension is at most $\tau-1$.  
Conversely, if a family
of graphs has Assouad-Nagata dimension at most $d$ and the covers can
be computed efficiently, then the family admits a $(\sigma,d+1)$-weak
sparse partition scheme, for some constant $\sigma$.

All our proofs are constructive and yield polynomial-time algorithms to compute the corresponding covers. 
In particular, whenever we obtain that a fixed class has Assouad-Nagata dimension at most $n$, then we in fact obtain
a $(O(1),n+1)$-weak partition scheme for the class.
So we have the following corollary from Theorem~\ref{thm:main} and Corollary~\ref{cor:genus_intro}. 

\begin{corollary}
For every integer $g \geq 0$,  there exists a real number $N$ such that the following hold.
    \begin{enumerate}
        \item There exists an $( N,3)$-weak partition scheme for the class of weighted $K_{3,g}$-minor graphs.
        \item There exists an $( N,3)$-weak partition scheme for the class of weighted graphs embeddable in a surface of Euler genus $g$. 
    \end{enumerate}
 \end{corollary}
 
We remark that while the Assouad-Nagata dimension and sparse partition are almost equivalent, the emphasis is on different parameters. 
In the case of the Assouad-Nagata dimension, the goal is to minimise the dimension (or equivalently $\tau$ in the sparse partition scheme), while in the $(\sigma,\tau)$-weak sparse partition scheme, the goal is usually to minimise a function of $\sigma$ and $\tau$ which depends on the application. 
As an example, it was proved in~\cite{JLNRS05} that if an $n$-vertex  graph admits a $(\sigma,\tau)$-weak sparse partition scheme, then the graph has a universal Steiner tree with stretch $O(\sigma^2 \tau \log_\tau n)$, so in this case the goal is to minimise $\sigma^2\cdot \tfrac{\tau}{\log \tau}$.

\subsection{Application 3: Weak diameter colouring and clustered colouring}\label{sec:clustered}

The asymptotic dimension of a graph is closely related to weak diameter colourings of its powers.

For a graph $G$, the \emph{weak diameter in $G$ of a subset} $S$ of $V(G)$ is the maximum distance in $G$ between two vertices of $S$; the \emph{weak diameter in $G$ of a subgraph $H$ of $G$ is the weak diameter in $G$ of $V(H)$} (thus we are taking distances in $G$ rather than $H$).
Given a colouring $c$ of a graph $G$, a \emph{$c$-monochromatic component} (or simply a \emph{monochromatic component} if $c$ is clear from the context) is a connected component of the subgraph of $G$ induced by some colour class of $c$.

A graph $G$ is \emph{$k$-colourable with weak diameter in $G$ at most $d$} if each vertex of $G$ can be assigned a colour from $\{1,\ldots,k\}$ so that all monochromatic components have weak diameter in $G$ at most $d$. 

Weak diameter colouring is also studied under the name of \emph{weak diameter network decomposition} in distributed computing (see~\cite{AGLP}), although in this context $k$ and $d$
usually depend on $|V(G)|$ (they are typically of order $\log|V(G)|$), while here we will only consider the case where $k$ and $d$ are constants.
Observe that the case $d=0$ corresponds to the usual notion of (proper) colouring. 
Note also that  weak diameter colouring should not be confused with the stronger  notion that requires that each monochromatic component has bounded diameter,  where the distance is computed in the monochromatic component instead of in $G$
(see for instance~\cite[Theorem 4.1]{LO18}).   
  
For any integer $\ell\ge 1$, the {\it $\ell$-th power of a graph $G$}, denoted by $G^\ell$, is the graph obtained from $V(G)$ by adding an edge $xy$ for each pair of distinct vertices of $G$ with distance at most $\ell$.
Note that for any graph $G$, $G^1$ coincides with $G$.
The following simple observation allows us to study asymptotic dimension in terms of weak diameter colouring.
(For completeness, we will include a proof in Appendix \ref{sec:appendix_dia_asdim}.)

\begin{proposition}\label{obs:weakdiameter}
Let $\F$ be a class of graphs.
Let $m>0$ be an integer.
Then $\ad(\F) \leq m-1$ if and only if there exists a function $f: {\mathbb N} \rightarrow {\mathbb N}$ such that for every $G \in \F$ and $\ell \in {\mathbb N}$, $G^\ell$ is $m$-colourable with weak diameter\footnote{Note that for any function $c:V(G) \rightarrow [m]$, $c$ is an $m$-colouring of $G$ and an $m$-colouring of $G^\ell$, but the $c$-monochromatic components in $G$ are different from the $c$-monochromatic components in $G^\ell$.} in $G^\ell$ at most $f(\ell)$.
\end{proposition}

We say that a class $\mathcal{G}$ of graphs has \emph{weak diameter chromatic number} at most $k$ if there is a constant $d$ such that every graph $G$ in $\mathcal{G}$ is $k$-colourable with weak diameter in $G$ at most $d$. 
By Proposition~\ref{obs:weakdiameter}, Theorems~\ref{thm:minorch} and \ref{tw_ad_intro_1} and Corollary~\ref{apex_layered_tw_ad_intro} have the following immediate corollary.

\begin{corollary}\label{cor:wdcol}
For every integer $\ell \geq 1$, the following holds.
    \begin{enumerate}
        \item If $\F$ is a proper minor-closed family (for example, the class of planar graphs and any class of graphs embeddable in a fixed surface), then the class $\{G^\ell: G \in \F\}$ has weak diameter chromatic number at most 3.
        \item If $\F$ is a proper minor-closed family that does not contain some planar graph, then the class $\{G^\ell: G \in \F\}$ has weak diameter chromatic number at most 2.
        \item If $\F$ is a class of graphs such that there exist integers $w,k$ such that for every graph $G \in \F$, there exists $Z \subseteq V(G)$ with $\lvert Z \rvert \leq k$ such that $G-Z$ has layered treewidth at most $w$, then the class $\{G^\ell: G \in \F\}$ has weak diameter chromatic number at most 3.
    \end{enumerate}
\end{corollary}

Note that the case $\ell\ge 2$ of (1) is not a direct consequence of the case $\ell=1$ of (1), (2), or (3). This is because the $2$nd power of a tree can contain arbitrarily large complete subgraphs, so the classes $\{G^\ell:G \in \F\}$ for $\ell \geq 2$ mentioned in Corollary~\ref{cor:wdcol} are not minor-closed families and have unbounded layered treewidth, even after the deletion of a bounded number of vertices. 

\medskip

Weak diameter colouring is related to clustered colouring which is a variation of the traditional notion of a proper colouring and has received wide attention.
A graph $G$ is \emph{$k$-colourable with clustering $c$} if each vertex of $G$ can be assigned a colour from $\{1,\ldots,k\}$ so that all monochromatic components have at most $c$ vertices. 
The case $c=1$ corresponds to the usual notion of (proper) colouring, and there is a large body of work on the case where $c$ is a fixed constant. 
In this context, we say that a class $\mathcal{G}$ of graphs has \emph{clustered chromatic number} at most $k$ if there is a constant $c$ such that every graph of $\mathcal{G}$ is $k$-colourable with clustering $c$ (see~\cite{Woo18} for a recent survey).

\begin{corollary} \label{obs:cluster}
Let $\Delta \geq  2$ be an integer.
If $\mathcal{G}$ is a class of graphs with asymptotic dimension at most $k$ such that every graph in ${\mathcal G}$ has maximum degree at most $\Delta$, then for every integer $\ell >0$, the clustered chromatic number of the class $\{G^\ell: G \in \mathcal{G}\}$ is at most $k+1$. 
\end{corollary}

\begin{proof}
By Proposition~\ref{obs:weakdiameter}, for every integer $\ell >0$, there is a constant $d$ such that for any graph $G\in \mathcal{G}$, $G^\ell$ is $(k+1)$-colourable with weak
diameter in $G^\ell$ at most $d$. 
Consider such a colouring of $G^\ell$.
Since $G$ has maximum degree at most $\Delta$, $G^\ell$ has maximum degree at most $\Delta^{\ell+1}$, so each monochromatic component of $G^\ell$ has size at most $\Delta^{(\ell+1)(d+1)}$. 
This implies that $\{G^\ell: G \in \mathcal{G}\}$ has clustered chromatic number at most $k+1$. 
\end{proof}

A direct consequence of the case $\ell=1$ in Corollaries~\ref{cor:wdcol} and \ref{obs:cluster} is that under the bounded maximum degree condition, $H$-minor free graphs, bounded treewidth graphs, and graphs obtained from bounded layered treewidth graphs by adding a bounded number of apices are (respectively) 3-colourable, 2-colourable, and 3-colourable with bounded clustering (these results were originally proved in \cite{LO18}, \cite{ADOV03}, and \cite{LW19}, respectively).
In addition, as the $\ell$-th power of a graph of maximum degree $\Delta$ has maximum degree at most $\Delta^{\ell+1}$, the above result can be extended to any integer $\ell$.

\subsection{Outline of the paper}

The first goal of this paper is to prove Theorems~\ref{thm:minorch} and \ref{tw_ad_intro_2}, where the latter is equivalent to Theorem \ref{tw_ad_intro_1}  by the Grid Minor Theorem and Theorem~\ref{compact_graphs}.
The key tool (Theorem~\ref{tree_extension_ad}) that we develop in this paper to prove Theorems~\ref{thm:minorch} and \ref{tw_ad_intro_2} might be of independent interest.
It allows us to show that generating a new class of graphs from hereditary classes by using tree-decompositions of bounded adhesion does not increase the asymptotic dimension.

Using Proposition~\ref{obs:weakdiameter}, we will prove Theorem~\ref{tree_extension_ad} in its equivalent form stated in terms of weak diameter colouring.
In order to do so, we will prove a stronger version that shows that we can extend any reasonable precolouring to a desired colouring.
In Section~\ref{sec:center}, we build the machinery for extending such a precolouring and use this machinery to prove Theorem~\ref{apex_extension_ad_intro}.
In Section~\ref{sec:tree}, we will prove Theorem~\ref{tree_extension_ad} and use it to prove Theorem~\ref{tw_ad_intro_2}.

We require another essential tool (Theorem~\ref{thm:bd}) to prove Theorem~\ref{thm:minorch}.
Theorem~\ref{thm:bd} allows us to bound the asymptotic dimension if there exists a layering of graphs such that the asymptotic dimension of any subgraph induced by a union of finitely many consecutive layers is under control.
We develop this machinery in Section~\ref{sec:control}.
In Section~\ref{sec:minor}, we will use this tool to derive our result on layered treewidth (Theorem~\ref{layered_tw_ad_intro}) from the result on treewidth (Theorem~\ref{tw_ad_intro_2}) and prove Theorem~\ref{thm:minorch} by combining Theorems~\ref{layered_tw_ad_intro}, \ref{apex_extension_ad_intro} and \ref{tree_extension_ad}.

Sections~\ref{sec:tw}, \ref{sec:k3p} and \ref{sec:surfaces} address Assouad-Nagata dimension.
We will prove Theorem~\ref{thm:nofptw_intro} in Section~\ref{sec:tw}, showing that no analogous result of Theorem~\ref{layered_tw_ad_intro} for Assouad-Nagata dimension can hold.
In Section~\ref{sec:k3p}, we will prove the result for $K_{3,p}$-minor free graphs (Theorem~\ref{thm:main}) by using the tool developed in Section~\ref{sec:control} and a notion called ``fat minor''.
In Section~\ref{sec:surfaces}, we show that the result for complete Riemannian surfaces (Theorem~\ref{cor:manifold_intro}) follows from Theorem~\ref{thm:main} by a simple argument about quasi-isometry.

Finally, in Section~\ref{sec:geom} we investigate the asymptotic dimension of classes of graphs with a low-dimensional representation and prove Theorem~\ref{thm:polygrowth_intro}.

We conclude this paper in Section~\ref{sec:ccl} with some open problems.

\subsection{Remark on the notation}
In this paper, ${\mathbb R}^+$ denotes the set of all positive real numbers and ${\mathbb N}$ denotes the set of all positive integers.
All graphs in this paper are simple. That is, there exist no parallel edges or loops.

We recall that for a graph $G$ and a subset $S$ of $V(G)$, the subgraph of $G$ induced by $S$ is the graph, denoted by $G[S]$, whose vertex-set is $S$ and whose edge-set consists of the edges of $G$ with both ends in $S$; for a weighted graph $(G,\phi)$ and a subset $S$ of $V(G)$, the (weighted) subgraph of $(G,\phi)$ induced by $S$ is the weighted graph $(G[S],\phi|_{E(G[S])})$.
Note that the (weighted) subgraph induced by $S$ is a (weighted) graph, so it defines a pseudometric space, but this pseudometric space is not necessarily equal to the induced metric which is the space $S$ together with the distance function computed in $G$ or $(G,\phi)$.

\section{Centered sets} \label{sec:center}

Given a set of vertices $S\subseteq V(G)$, and an integer $r\ge 0$, we denote by $N_G^{\le r}(S)$ the set of vertices in $V(G)$ which are at distance at most $r$ from at least one vertex in $S$. Given an integer $k\ge 0$, if a set of vertices $Z$ is contained in $N_G^{\le r}(S)$ for some set $S$ of size at most $k$, we say that $Z$ is a \emph{$(k,r)$-centered set}  in $G$.

In this section we prove that given some graph $G$ and a $(k,r)$-centered set $Z$, the union of an arbitrary colouring of $Z$  and a colouring of $G-Z$ with bounded weak diameter in $G-Z$ gives a colouring of $G$ of bounded weak diameter in $G$. We will make extensive use of this result in the proof of Theorem~\ref{tree_extension_ad}, from which we derive all our results about the asymptotic dimension of minor-closed families.

For $i \in \{1,2\}$, let $f_i$ be a function with domain $S_i$.
If $f_1(x)=f_2(x)$ for every $x \in S_1 \cap S_2$, we define $f_1 \cup f_2$  to be the function with domain $S_1 \cup S_2$ such that for every
$i\in \{1,2\}$ and $x \in S_i$, $(f_1 \cup f_2)(x) = f_i(x)$.

\begin{lemma} \label{deleting_centered_set}
For any integers $k \geq 0, r\ge 0, N\ge 1$ there exists an integer $N^*$ such that the following holds.
For every graph $G$, any integers $m\ge 1, \ell\ge 1$, and every $(k,r)$-centered set $Z$ in $G$, if $c$ is an $m$-colouring of $G-Z$ with weak diameter in $(G-Z)^\ell$ at most $N$, and $c_Z:Z\to[m]$ is an arbitrary $m$-colouring  of $Z$, then $c\cup c_Z$ is an $m$-colouring of $G$ with weak diameter in $G^\ell$ at most $N^*$.
\end{lemma}

\begin{proof}
We define $f(\alpha,x,y)=2^\alpha(y+2x+2) - 2x-2$ for any integers $\alpha \geq 0$, $x \geq 0$ and $y >0$.
 Note that
 			\begin{itemize}
				\item $f(0,x,y)= y$, and 				
                \item for every $\alpha \in \N$, $f(\alpha,x,y) = 2x+2+2f(\alpha-1,x,y)$. 
 			\end{itemize}

It suffices to show that the $m$-colouring $c \cup c_Z$ of $G^\ell$ has weak diameter in $G^\ell$ at most $f(k,r,N)$, since we can take $N^*=f(k,r,N)$. 

 We shall prove it by induction on $k$. 
Since $Z$ is $(k,r)$-centered, there exists a subset $S$ of $V(G)$ with size at most $k$ such that $Z \subseteq N_G^{\le r}(S)$.
 When $k=0$, $S=Z=\emptyset$, and hence we are done since $f(0,r,N) \geq N$.
 So we may assume that $k \geq 1$ and this lemma holds when $k$ is smaller.

Let $s_1 \in S$.
Let $S'=S-\{s_1\}$ and $Z'=N^{\le r}_G(S') \cap Z$.
We denote $Z_1=Z - Z' \subseteq N^{\le r}_G(\{s_1\})$.
Since $Z_1$ contains no vertex at distance at most $r$ in $G$ from $S'$, we infer that $Z' \subseteq N_{G-Z_1}^{\leq r}(S')$.
We can apply the induction hypothesis on $G-Z_1$ and $Z'$, since $c$ is indeed an $m$-colouring of $(G-Z_1)-Z'=G-Z$ with weak diameter in $((G-Z_1)-Z')^\ell$ at most $N$, and $Z'$ is a $(k-1,r)$-centered set. We obtain that the $m$-colouring $c' \coloneqq c\cup c_{Z'}$ has weak diameter in $(G-Z_1)^\ell$ at most $f(k-1,r,N)$, where $c_{Z'}$ is the restriction of $c_Z$ to $Z'$.

Writing $c_{Z_1}$ for the restriction of $c_Z$ to $Z_1$, we have $c \cup c_Z = c' \cup c_{Z_1}$. So it remains to show that $c' \cup c_{Z_1}$ has weak diameter in $G^\ell$ at most $f(k,r,N)$.

Let $M$ be a $(c' \cup c_{Z_1})$-monochromatic component in $G^\ell$. If $M$ is disjoint from $N_G^{\leq \ell}(Z_1)$, then all its vertices are at distance more than $\ell$ in $G$ from any vertex of $Z_1$ with the same colour, so $M$ is a $c'$-monochromatic component and has weak diameter in $(G-Z_1)^\ell \subseteq G^\ell$ at most $f(k-1,r,N) \leq f(k,r,N)$.
We can now assume that $M$ intersects  $N_G^{\leq \ell}(Z_1)$. 
Then $M-Z_1$ consists of a (possibly empty) union of  $c'$-monochromatic components in $(G-Z_1)^{\ell}$, where each of them contains a vertex within distance in $G$ at most $\ell$ from $Z_1$, hence at distance  in $G^\ell$ at most $1$ from $Z_1$, and has weak diameter in $(G-Z_1)^{\ell}  \subseteq G^\ell$ at most $f(k-1,r,N)$. 
Therefore $s_1$ is at distance  in $G^{\ell}$ at most $r+1+f(k-1,r,N)$ from any vertex of $M$, and so the weak diameter in $G^{\ell}$ of $M$  is at most twice that value, which is $f(k,r,N)$, as desired.
\end{proof}

Let $\ell>0,m>0,N>0$ be integers.
We say a class $\F$ of graphs is {\it $(m,\ell,N)$-nice} if for every $G \in \F$, $G^\ell$ is $m$-colourable with weak diameter in $G^\ell$ at most $N$.

Recall that for every integer $n \geq 0$, $\F^{+n}$ is the class of graphs such that for every $G \in \F^{+n}$, there exists $Z \subseteq V(G)$ with $\lvert Z \rvert \leq n$ such that $G-Z \in \F$. 
Notice that such $Z$ is an $(n,0)$-centered set in $G$.
Through a direct application of Lemma~\ref{deleting_centered_set} with $(k,r)=(n,0)$, we  obtain the following observation.

\begin{obs} \label{apex_extension}
For all integers $n \geq 0,N\ge 1$, there exists $N^*\ge 1$ such that  for all integers $\ell \ge 1, m\ge 1$, if $\F$ is an $(m,\ell,N)$-nice class of graphs, then $\F^{+n}$ is an $(m,\ell,N^*)$-nice class. 
\end{obs}

Now we are ready to prove Theorem \ref{apex_extension_ad_intro}.
The following is a restatement.

\begin{thm} \label{apex_extension_ad}
Let $\F$ be a class of graphs and $n \geq 0$ an integer.
Then $\ad(\F^{+n})=\ad(\F)$.
\end{thm}

\begin{proof}
Since $\F^{+n} \supseteq \F$, we have $\ad(\F^{+n})\ge\ad(\F)$, so we
only have to prove that $\ad(\F^{+n})\le\ad(\F)$.

Let $m=\ad(\F)$. 
By Proposition~\ref{obs:weakdiameter}, there is a function $f: {\mathbb N} \rightarrow \N$ such that for every $\ell \in {\mathbb N}$, $\F$ is $(m+1,\ell,f(\ell))$-nice.
By Observation~\ref{apex_extension}, there exists a function $g: {\mathbb N} \rightarrow \N$ such that for every $\ell \in {\mathbb N}$, $\F^{+n}$ is $(m+1,\ell,g(f(\ell)))$-nice.
By Proposition~\ref{obs:weakdiameter}, $\ad(\F^{+n}) \leq m$.
\end{proof}

A {\it vertex-cover} of a graph $G$ is a subset $S$ of $V(G)$ such that $G-S$ has no edge. 
As the class of edgeless graphs has asymptotic dimension 0, the following is a direct consequence of Theorem~\ref{apex_extension_ad}.

\begin{obs} \label{vc_ad}
For any integer $k$, the class of all graphs with a vertex-cover of
size at most $k$ has asymptotic dimension 0.
\end{obs}

By applying Lemma~\ref{deleting_centered_set} with $V(G)=Z$ and using the fact that a colouring of the empty graph has weak diameter $0$,  we obtain the following.

\begin{obs} \label{all_centered}
For any integers $k \geq 0,r \geq 0$, there exists an integer $N^*\ge 1$ such that for every graph $G$, and for all integers $m\ge 1,\ell\ge 1$, if $V(G)$ is $(k,r)$-centered, then any $m$-colouring of $G^\ell$ has weak diameter in $G^\ell$ at most $N^*$. 
\end{obs}

\section{Gluing along a tree} \label{sec:tree}

 In this section, we prove one of the main technical results of this paper, from which we derive our results about the asymptotic dimension of classes of graphs of bounded treewidth, and of minor-closed families of graphs. This result roughly states that if every graph in a class $\G$ has a tree-decomposition with bounded adhesion where each bag belongs to some family $\F$, then the asymptotic dimension of $\G$ is essentially that of $\F$ -- 
 or more precisely that of a class of graphs that can be obtained from the graphs of $\F$ by adding vertices and edges in a specific way.

A class $\F$ of graphs is {\it hereditary} if for every $G \in \F$, every induced subgraph of $G$ belongs to $\F$.

\begin{thm} \label{tree_extension_ad}
Let $\cC$ be a hereditary class of graphs, and let $\theta \geq 1$ be an integer.
Let $\G$ be a class of graphs such that for every $G \in \G$, there exists a tree-decomposition $(T,\X)$ of $G$ of adhesion at most $\theta$, where $\X=(X_t: t \in V(T))$, such that $\cC$ contains all graphs which can be obtained from any $G[X_t]$ by adding, for each neighbour 
$t'$ of $t$ in $T$, a set of pairwise non-adjacent new vertices whose neighbourhoods are contained in $X_t\cap X_{t'}$. 
Then $\ad(\G) \leq \max\{\ad(\cC),1\}$.
\end{thm}

\subsection{Sketch of the proof of Theorem~\ref{tree_extension_ad}}

The proof of Theorem~\ref{tree_extension_ad} follows a rather technical induction, whose precise statement is that of Lemma~\ref{tree_extension}.
For convenience of the reader, and in order to ease the understanding of the purpose of the incoming set-up, we begin by sketching the main steps of the proof.

Let $G \in \G$ and $(T,\X)$ be a tree-decomposition of $G$ as stated in Theorem \ref{tree_extension_ad}.
As $(T,\X)$ has bounded adhesion, we can treat $T$ as a rooted tree so that the bag of the root, denoted by $X_{t^*}$, has size at most the adhesion, up to creating a redundant bag if necessary.
We shall prove a stronger statement: for every $Z \subseteq N_G^{\leq 3\ell}(X_{t^*})$, every precolouring $c|_Z$ on $Z$ with at most $m$ colours extends to an $m$-colouring of $G^\ell$ with bounded weak diameter in $G^\ell$, by induction on the adhesion of $(T,\X)$, and subject to this, induction on $\lvert V(G) \rvert$. (See Lemma \ref{tree_extension} for a precise statement.)

By first extending $c|_Z$ to $N_G^{\leq 3\ell}(X_{t^*})$, we may assume $Z=N_G^{\leq 3\ell}(X_{t^*})$.
Hence the subgraph $T_0$ of $T$ induced by the nodes whose bags intersect $Z$ is a subtree of $T$ containing $t^*$.
Let $G_0$ be the subgraph of $G$ induced by the bags of the nodes in $T_0$.
Let $U_E$ be the set of edges of $T$ with exactly one end in $T_0$.
For every $e \in U_E$, let $G_e$ be the subgraph of $G$ induced by the bags of the nodes in the component of $T-e$ disjoint from $T_0$.
Then for each $e \in U_E$,
there is a partition $\P_e$ of $V(G_0) \cap V(G_e)$ into sets such that two vertices in $V(G_0) \cap V(G_e)$ are not far from each other in $G_e$ if and only if they are contained in the same part of $\P_e$.
This can be done as $\lvert V(G_0) \cap V(G_e) \rvert$ is bounded by the adhesion of $(T,\X)$.

Note that $Z$ is an $(\lvert X_{t^*} \rvert,3\ell)$-centered set, and $G_0-Z$ has a tree-decomposition of smaller adhesion by the definition of $G_0$ and $T_0$.
So $(G_0-Z)^\ell$ has a colouring by induction.
Hence the precolouring $c_Z$ on $Z$ can be extended to $G_0^\ell$ by Lemma \ref{deleting_centered_set}.
However, it is troublesome to further extend the colouring to $G^\ell$, as no information about $G_e$ for $e \in U_E$ can be seen from $G_0$ and edges of $G^\ell$ with ends in $V(G_0)$ cannot be completely told from $G_0^\ell$.
To overcome this difficulty, we add ``gadgets'' to $G_0$ to obtain a graph $H$ such that extending $c_Z$ from $Z$ to $H^\ell$ gives sufficient information about how to further extend it to $G^\ell$.
The gadgets we add to form $H$ are a vertex $v_Y$ for each $e \in U_E$ and each part $Y \in \P_e$, and the edges between $v_Y$ and $Y$.

However, $H-Z$ possibly does not have a tree-decomposition of smaller adhesion, so the induction hypothesis cannot be applied to $H-Z$.
Instead, we setup a more technical induction hypothesis to overcome this difficulty.
This is the motivation of $(\eta,\theta,\F,\F')$-constructions mentioned in Section \ref{sec:key_lemma_gluing_tree}.
So we can extend $c_Z$ to an $m$-colouring of $H^\ell$ by using this technical setting.

Note that no vertex in $Z$ is in $\bigcup_{e \in U_E}G_e$.
Then for each $e \in U_E$, we colour vertices in $G_e-V(G_0)$ that have distance in $G_e$ at most $\ell$ from $V(G_e) \cap V(G_0)$ according to the colours on $v_Y$ for $Y \in \P_e$.
Call the set of vertices coloured in this step $Z_1$.
Then we colour every uncoloured vertex in $G_e$ that has distance in $G_e$ at most $\ell$ from $Z_1$ colour 1.
Call the set of vertices coloured in this step $Z_2-Z_1$, and let $Z_2$ contain $Z_1$.
Then we colour every uncoloured vertex in $G_e$ that has distance in $G_e$ at most $\ell$ from $Z_2$ colour 2.
This ensures that no matter how we further colour the other vertices, every monochromatic component intersecting $G_0$ must be contained in $V(G_0) \cup Z_1 \cup Z_2$; and we can show that such monochromatic components have small weak diameter due to our definition of $\P_e$ and $v_Y$.

At this point, for every $e \in U_E$, the vertices coloured in $G_e$ are contained in $N_G^{\leq 3\ell}(V(G_0) \cap V(G_e))$.
Hence for each $e \in U_E$, we can extend this precolouring to an $m$-colouring of $G_e^\ell$ with bounded weak diameter in $G_e^\ell$ by induction, since $G_e$ has fewer vertices than $G$.
Every monochromatic component not intersecting $G_0$ must be contained in $G_e$ for some $e \in U_E$ and hence has small weak diameter.
This completes the sketch of the proof of Lemma \ref{tree_extension} (and Theorem \ref{tree_extension_ad}, which follows as a simple consequence).

\subsection{Key lemma for proving Theorem~\ref{tree_extension_ad}} \label{sec:key_lemma_gluing_tree}

Let $\ell,N,m\ge 1$ be integers. 
Recall that a class $\F$ of graphs is {\it $(m,\ell,N)$-nice} if for every $G \in \F$, $G^\ell$ is $m$-colourable with weak diameter in $G^\ell$ at most $N$. 

Let $G$ be a graph, and let $(T,\X)$ be a tree-decomposition of $G$, where $\X=(X_t: t \in V(T))$.
For every $S \subseteq V(T)$, we define $X_S = \bigcup_{t \in S}X_t$; 
when $T'$ is a subgraph of $T$, we write $X_{T'}$ instead of $X_{V(T')}$.

\smallskip

A {\it rooted tree} is a directed graph whose underlying graph is a tree
where all vertices have in-degree $1$, except one which has in-degree $0$ and that we call the \emph{root}.
A {\it rooted tree-decomposition} of a graph $G$ is a tree-decomposition $(T,\X)$ of $G$ such that $T$ is a rooted tree.

We denote $\F^{+1}$ by $\F^+$.

Let $\F$ and $\F'$ be classes of graphs.
Let $\eta,\theta$ be nonnegative integers with $\eta \leq \theta$.
A rooted tree-decomposition $(T,\X)$ of a graph $G$ is called an {\it $(\eta,\theta,\F,\F')$-construction} of $G$ if it has adhesion at most $\theta$ and satisfies the following additional properties:
	\begin{itemize}
		\item for every edge $tt' \in E(T)$, if $\lvert X_t \cap X_{t'} \rvert > \eta$, then one end of $tt'$ has no child, say $t'$, and the set $X_{t'}-X_t$ contains at most 1 vertex,
		\item for the root $t^*$ of $T$, 
		\begin{itemize}
		    \item $\lvert X_{t^*} \rvert \leq \theta$, 
			\item if $\eta>0$, then $X_{t^*} \neq \emptyset$, 
		\end{itemize}
		\item for every $t \in V(T)$,
			\begin{itemize}
				\item if $t$ has a child in $T$, then $G[X_t] \in \F$,
				\item if $t$ has no child in $T$, then $G[X_t] \in \F^+$, and 
				\item $\F'$ contains all graphs which can be obtained from $G[X_t]$ by adding, for each child $t'$ of $t$ in $T$, vertices whose neighbourhoods are contained in $X_t\cap X_{t'}$. (Note that this final property only applies to nodes $t$ that have children, since otherwise the precondition ``for each child $t'$ of $t$" is void.) 
			\end{itemize}
	\end{itemize}
We say that a graph $G$ is {\it $(\eta,\theta,\F,\F')$-constructible} if there exists an $(\eta,\theta,\F,\F')$-construction of $G$. 

For every rooted tree $T$, define $I(T)$ to be the set of nodes of $T$ that have at least one child.

Let $G$ be a graph and $m>0$ an integer.
Let $S \subseteq V(G)$.
Let $c: S \rightarrow [m]$ be a function.
Let $c'$ be an $m$-colouring of $G$ such that $c'(v)=c(v)$ for every $v \in S$.
Then we say that $c$ can be {\it extended} to $c'$.

Recall that a class $\F$ of graphs is {\it hereditary} if for every $G \in \F$, every induced subgraph of $G$ belongs to $\F$.
Note that if $\F$ is hereditary, then so is $\F^+$.

\begin{lemma} \label{tree_extension}
For any integers $\ell \geq 1,N \geq 1,m \geq 2,\theta \geq 0$, there exists a function $f^*: {\mathbb N} \cup \{0\} \rightarrow {\mathbb N}$ such that the following holds.
Let $\F$ and $\F'$ be $(m,\ell,N)$-nice hereditary classes. 
Let $\eta$ be a nonnegative integer with $\eta \leq \theta$.
Let $G$ be an $(\eta,\theta,\F,\F')$-constructible graph with an $(\eta,\theta,\F,\F')$-construction $(T,\X)$. 
Denote $\X$ by $(X_t: t \in V(T))$.
Let $t^*$ be the root of $T$.
For every $Z \subseteq N_G^{\leq 3\ell}(X_{t^*})$, every function $c_Z: Z \rightarrow [m]$ can be extended to an $m$-colouring of $G^\ell$ with weak diameter in $G^\ell$ at most $f^*(\eta)$.
\end{lemma}

\begin{proof}
A visual summary of some of the notation introduced throughout the proof is depicted in Figure~\ref{fig:tg}.
Let $\ell \geq 1,N \geq 1,m \geq 2,\theta \geq 0$ be integers.
By Observation \ref{apex_extension}, there exists an integer $N_{\F^+}$ (that depends only on $N$) such that $\F^+$ is $(m,\ell,N_{\F^+})$-nice.
Note that $\F \subseteq \F^+$, so we may assume that $N_{\F^+} \geq N$.

We define the following.
	\begin{itemize}
		\item Let $f_1: {\mathbb N} \rightarrow {\mathbb N}$ be the function such that for every $x \in {\mathbb N}$, $f_1(x)$ is the integer $N^*$ mentioned in Lemma \ref{deleting_centered_set} by taking $(k,r,N)=(\theta,3\ell,x)$. 
		\item Let $N_\theta$ be the integer $N^*$ mentioned in Lemma \ref{deleting_centered_set} by taking $(k,r,N)=(\theta,0,1)$.
			Note that we may assume that $N_\theta \geq \theta+1$  by possibly replacing $N_\theta$ by $\max\{N_\theta,\theta+1\}$. 
		\item Let $N'_\theta$ be the number $N^*$ mentioned in Observation \ref{all_centered} by taking $(k,r)=(\theta,3\ell)$. 
		\item Define $f^*: ({\mathbb N} \cup \{0\}) \rightarrow {\mathbb N}$ to be the function such that
		    \begin{itemize}
		        \item $f^*(0)=N_{\F^+} +N'_\theta+N_\theta+f_1(N)$, and 
		        \item for every $x \in {\mathbb N}$, $f^*(x)= \max\{(14\theta+4)\ell + 7\theta\ell^2f_1(f^*(x-1)),f^*(0)\}$.  
		    \end{itemize}
	\end{itemize}

Let $\eta,G,(T,\X),t^*,Z,c_Z$ be as defined in the lemma. 
We shall prove this lemma by a  triple induction, first on $\eta$, then on $\lvert I(T) \rvert+\lvert V(G)-Z \rvert+\lvert V(G) \rvert$,  and then on $\lvert V(T) \rvert$.
Since $(T,\X)$ is an $(\eta,\theta,\F,\F')$-construction, $\lvert X_{t^*} \rvert \leq \theta$.
So $Z$ is $(\theta,3\ell)$-centered.
If $Z=V(G)$, then $c_Z$ is itself an $m$-colouring of $G^\ell$ with weak diameter in $G^\ell$ at most $N'_\theta \leq f^*(\eta)$ by Observation \ref{all_centered}. So we may assume $Z \neq V(G)$.

\begin{clm}[Base case]\label{cl:1}
The lemma holds for $\eta=0$.
\end{clm}

\begin{claimproof}
Assume that $\eta=0$. 
Note that for every component $Q$ of $G^\ell$, $Q=C^\ell$ for some component $C$ of $G$.
So it suffices to show that for each component $C$ of $G$, $c_Z|_{Z\cap V(C)}$ extends to an $m$-colouring of $C^\ell$ with weak diameter in $C^\ell$ at most $f^*(\eta)$. 

Let $C$ be a component of $G$.
Let $W=\{tt' \in E(T): X_t \cap X_{t'}=\emptyset\}$. Then $V(C) \subseteq X_{T_C}$ for some component $T_C$ of $T-W$. Let $t_C$ be the root of $T_C$, and let us write $G_C$ as a shorthand for $G[X_{T_C}]$.
For every edge $tt'\in E(T_C)$ it holds that $|X_t\cap X_{t'}|>0 = \eta$, and it follows from the fact that $(T,\X)$ is an $(\eta,\theta,\F,\F')$-construction that one end of $tt'$ has no child, say $t'$, and $|X_{t'}-X_t|\leq 1$. In particular $T_C$ is star and $G_C$ is obtained from $G[X_{t_C}]$ by adding, for each child $t'$ of $t_C$ in $T_C$, a single vertex whose neighbourhood is contained in $X_{t_C} \cap X_{t'}$. So $G_C \in \F'$.

Since $\F'$ is an $(m,\ell,N)$-nice hereditary class, if $t_C\neq t^*$, then  $Z\cap V(C)=\emptyset$, so $(c_Z)|_{Z\cap V(C)}=(c_Z)|_{\emptyset}$ can be extended to an $m$-colouring of $C^\ell$ with weak diameter in $C^\ell$ at most $N\leq f^*(\eta)$. So we may assume $t_C=t^*$. 
Then $Z \subseteq N_{G}^{\leq 3\ell}(X_{t^*}) \cap V(G_C)= N_{G_C}^{\leq 3\ell}(X_{t^*})$, so $Z$ is $(\theta,3\ell)$-centered in $G_C$. We find $G_C-Z \in \mathcal{F'}$ since $\F'$ is hereditary, so using that $\mathcal{F'}$ is $(m,\ell,N)$-nice and applying Lemma \ref{deleting_centered_set} we find that $c_Z|_{Z\cap X_{T_C}}$ can be extended to an $m$-colouring of $G_C^\ell$ with weak diameter in $G_C^\ell$ at most $f_1(N) \leq f^*(\eta)=f(0)$.
This proves the claim since  $C$ is a component of $G_C$.
\end{claimproof}

Henceforth we assume that we have proven the lemma for all instances for which $(\eta,\lvert I(T) \rvert+\lvert V(G)-Z \rvert+\lvert V(G) \rvert)$ is lexicographically smaller, and assume $\eta \geq 1$.

\begin{clm}
We may assume that $G$ is connected.
\end{clm}

\begin{claimproof}
Assume that $G$ is disconnected.
It suffices to show that for each component $C$ of $G$, $c_Z|_{Z\cap V(C)}$ extends to an $m$-colouring of $C^\ell$ with weak diameter in $C^\ell$ at most $f^*(\eta)$. 

Let $C$ be a component of $G$. 
Since $G$ is disconnected, $\lvert V(C) \rvert < \lvert V(G) \rvert$.
Let $T_C$ be the subtree of $T$ induced by $\{t \in V(T): X_t \cap V(C) \neq \emptyset\}$. Let $\X_C = (X_t \cap V(C): t \in V(T_C))$.
If $t^* \in V(T_C)$, then $t^*$ is the root of $T_C$ and $X_{t^*} \cap V(C) \neq \emptyset$, so in this case $(T_C,\X_C)$ is an $(\eta,\theta,\F,\F')$-construction of $C$ with 
\[
\lvert I(T_C) \rvert+\lvert V(C)-(Z \cap V(C)) \rvert+\lvert V(C)\rvert< \lvert I(T) \rvert+\lvert V(G)-Z \rvert+\lvert V(G) \rvert,
\]
and hence the induction hypothesis applies to $C$.
So we may assume $t^* \not \in V(T_C)$. Then $Z \cap V(C)=\emptyset$.
Let $T_C'$ be the rooted tree obtained from $T_C$ by adding a node $t^*_C$ adjacent to the root of $T_C$, where $t^*_C$ is the root of $T_C'$.
Let the bag at $t^*_C$ be the set consisting of a single vertex in the intersection of $V(C)$ and the bag of the root of $T_C$.
Then since $\eta \geq 1$, we obtain an $(\eta,\theta,\F,\F')$-construction of $C$ with underlying tree $T_C'$.
(Note that $\eta \geq 1$ implies that $t^*_C$ satisfies the last condition of being an $(\eta,\theta,\F,\F')$-construction, since  stars are the only graphs that can be obtained by adding new vertices to the bag of $t^*_C$, and $\eta \geq 1$ implies that $\F'$ contains stars.)
Since $t^* \not \in V(T_C)$, we find $\lvert I(T_C') \rvert\le \lvert I(T) \rvert$. 
Since $\lvert V(C) \rvert < \lvert V(G) \rvert$, we obtain
\[
\lvert I(T_C') \rvert+\lvert V(C)-(Z \cap V(C)) \rvert+\lvert V(C) \rvert < \lvert I(T) \rvert+\lvert V(G)-Z \rvert+\lvert V(G) \rvert,
\]
so the induction hypothesis applies to $C$.
\end{claimproof}

So henceforth we may assume that $G$ is connected.

\begin{clm} \label{claim_Zbasic}
We may assume that $Z=N_G^{\leq 3\ell}(X_{t^*})$ and $Z-X_{t^*} \neq \emptyset$.
\end{clm}

\begin{claimproof}
If there exists $v \in N_G^{\leq 3\ell}(X_{t^*})-Z$, then let $Z' = Z \cup \{v\}$ and let $c': Z' \rightarrow [m]$ be the function obtained from $c_Z$ by further defining $c'(v)=m$. Note that $|V(G)-Z'|<|V(G)-Z|$, so by the induction hypothesis, $c'$ (and hence $c_Z$) can be extended to an $m$-colouring of $G^\ell$ with weak diameter in $G^\ell$ at most $f^*(\eta)$.
Hence we may assume that $Z=N_G^{\leq 3\ell}(X_{t^*})$. 
In particular, since $Z\neq V(G)$ and $G$ is connected, we find $Z-X_{t^*}\neq \emptyset$. 
\end{claimproof}

\begin{figure}[htb]
 \centering
 \includegraphics[scale=1.1]{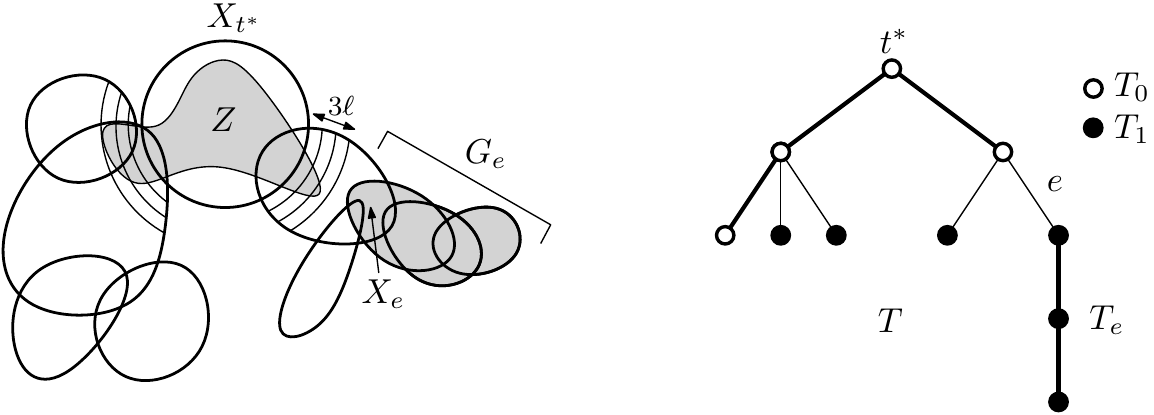}
 \caption{A snapshot of the setting of Lemma~\ref{tree_extension} and the notations used in the proof. Note that by Claim~\ref{claim_Zbasic} we know that we can assume that $Z=N_G^{\leq 3\ell}(X_{t^*})$, but we only include the original setting of this lemma in the picture.}
 \label{fig:tg}
\end{figure}

For each $v \in X_{t^*}$, let $T_v$ be the subgraph of $T$ induced by $\{t \in V(T): N_G^{\leq 3\ell}(\{v\}) \cap X_t \neq \emptyset\}$.
Since $(T,\X)$ is a tree-decomposition, $T_v$ is a subtree of $T$ containing $t^*$.
So $\bigcup_{v \in X_{t^*}}T_v$ is a subtree of $T$ containing $t^*$.

Let $T_0 = \bigcup_{v \in X_{t^*}}T_v$.
Since $Z = N_G^{\leq 3\ell}(X_{t^*})$, $Z \subseteq X_{T_0}$.
Let $U_E = \{e \in E(T):$ exactly one end of $e$ is in $V(T_0)\}$ (see Figure~\ref{fig:tg}, right, where the edges of $U_E$ are depicted with thin lines).
Note that for every vertex $t\in V(T_0)$ incident with an edge $e \in U_E$, the component of $T-e$ disjoint from $t$ is disjoint from $T_0$.

For each $e \in E(T)$, define $X_e$ to be the intersection of the bags of the ends of $e$. 
Since the tree-decomposition $(T,\X)$ has adhesion at most $\theta$, $|X_e|\le \theta$ for every $e\in E(T)$.
For each $e \in E(T)$, let $T_e$ be the component of $T-e$ disjoint from $t^*$. 
Since $G$ is connected and $X_{t^*} \neq \emptyset$, we may assume that $X_e \neq \emptyset$ for every $e\in E(T)$, for otherwise $X_{T_e}=\emptyset$ so that we can delete $T_e$ from $T$ to decrease  $\lvert V(T) \rvert$ and apply induction.

Let $T_1=\bigcup_{e\in U_E}T_e$.
Note that $V(T_1)=V(T)-V(T_0)$.
Let us define $G_0=G[X_{T_0}]$, $G_1=G[X_{T_1}]$, and $G_e=G[X_{T_e}]$ for every edge $e\in U_E$ (see Figure~\ref{fig:tg} for a visual summary of some of the notation introduced in this paragraph). 
By Claim~\ref{claim_Zbasic}, $Z \cap X_t \neq \emptyset$ for every $t \in V(T_0)$, and $Z \cap X_t = \emptyset$ for every $t \in V(T_1)$.

A naive plan is to first extend $c_Z$ to $G_0^\ell$ by using the induction hypothesis, and then further extend this to $G_1$ using the induction hypothesis on the graphs $G_e$ for $e\in U_E$. 
An issue of this naive plan is that some edges of $G^\ell$ with both ends in $V(G_0)$ cannot be told from $G_0$, so colouring $G_0$ without knowing those edges in $G^\ell$ might make some monochromatic component in $G^\ell$ contain arbitrarily many monochromatic components in $G_0^\ell$. 
To overcome this difficulty as well as other potential issues, we add ``gadgets'' to $G_0$ to obtain a graph $H$ such that extending $c_Z$ to $H^\ell$ gives sufficient information about how to extend it to $G^\ell$.

For each $e \in U_E$, we define a partition $\P_e$ of $X_e$ such that each part of $\P_e$ is a connected component of the subgraph of $(G_e)^{7\ell}$ induced by $X_e$. 
In other words, two vertices $x,y\in X_e$ are in the same part of $\P_e$ if and only if there exists a sequence $a_1,\dots,a_\theta$ of (not necessarily distinct) elements of $X_e$ such that $a_1=x$, $a_\theta=y$, and for every $i \in [\theta-1]$, there exists a path in $G_e$ from $a_i$ to $a_{i+1}$ of length at most $7\ell$.

Define $H$ to be the graph obtained from $G_0$ by adding, for each $e \in U_E$ and $Y \in \P_e$, a new vertex $v_Y$ whose neighbourhood in $H$ is $Y$.

\begin{clm}
\label{claim4}
There is an $m$-colouring $c_H$ of $H^\ell$ with weak diameter in $H^\ell$ at most $f_1(f^*(\eta-1))$ such that $c_H(v)=c_Z(v)$ for every $v \in Z$. 
\end{clm}
\begin{claimproof}
We first show that it is enough to give an
$(\eta-1,\theta,\F,\F')$-construction of $H-Z$. By the induction hypothesis, this would imply that there exists an $m$-colouring $c_H'$ of $(H-Z)^\ell$ with weak diameter in $(H-Z)^\ell$ at most $f^*(\eta-1)$.
By Lemma \ref{deleting_centered_set}, $c_H=c_Z \cup c_H'$ is then an $m$-colouring of $H^\ell$ with weak diameter in $H^\ell$ at most $f_1(f^*(\eta-1))$, and by definition $c_H(v)=c_Z(v)$ for every $v \in Z$. 

The remainder of the proof of this claim is devoted to showing the existence of an $(\eta-1,\theta,\F,\F')$-construction of $H-Z$.
We start by showing the existence of an $(\eta,\theta,\mathcal{F},\mathcal{F}')$-construction for $H$. 

Define $T'$ to be the rooted tree obtained from $T_0$ by adding, for each $e \in U_E$ and $Y \in \P_e$, a node $t_Y$ adjacent to the end of $e$ in $V(T_0)$.
Note that for every $e \in U_E$, since $X_e \neq \emptyset$, the end of $e$ in $V(T_0)$ has a child in both $T$ and $T'$.
For each $t' \in V(T_0)$, define $X'_{t'}=X_{t'}$; for each $t' \in V(T')-V(T_0)$, $t'=t_Y$ for some $e \in U_E$ and $Y \in \P_e$, and we define $X'_{t'}=X_e \cup \{v_Y\}$. 
Let $\X'=(X'_t: t \in V(T'))$.

Clearly, $(T',\X')$ is a tree-decomposition of $H$ of adhesion at most $\max_{e \in U_E}\{\theta,\lvert X_e \rvert\} = \theta$.
For each $tt' \in E(T')$, say $t'$ is a child of $t$, if $tt' \in E(T_0)$, then $X'_t=X_t$, $X'_{t'}=X_{t'}$, $t$ has a child in both $T$ and $T'$, and $t'$ has a child in $T'$ if and only if $t'$ has a child in $T$ (since $X_{e} \neq \emptyset$ for every $e \in U_E$); if $tt' \not \in E(T_0)$, then $t \in V(T_0)$ and $t' \not \in V(T_0)$, and $\lvert X'_{t'}-X'_t \rvert = 1$.
Hence for every $tt' \in E(T')$, if $\lvert X'_t \cap X'_{t'} \rvert > \eta$, then one end of $tt'$, say $t'$, has no child, and $\lvert X'_{t'}-X'_t \rvert \leq 1$.

Furthermore, $t^* \in V(T_0) \subseteq V(T')$ and $X'_{t^*}=X_{t^*}$, so $\lvert X'_{t^*} \rvert = \lvert X_{t^*} \rvert \leq \theta$.
Since $\eta \geq 1$, $X'_{t^*} = X_{t^*} \neq \emptyset$.
In addition, for every $t \in V(T')$, if $t$ has a child in $T'$, then $t \in V(T_0) \subseteq V(T)$ has a child in $T$, so $G[X'_t]=G[X_t] \in \F$; if $t$ has no child in $T'$, then either $t \in V(T)$ has no child in $T$ (so $G[X'_t] = G[X_t] \in \F^+$), or $t \in V(T')-V(T)$ and $G[X'_t]$ can be obtained by adding a vertex to $G[X_e] \in \F$ for some $e \in U_E$, so $G[X'_t] \in \F^+$.
If $t$ has a child in $T'$, then $t \in V(T_0) \subseteq V(T)$, so $\F'$ contains every graph that can be obtained from $G[X'_t]=G[X_t]$ by adding, for each child $t'$ of $t$ in $T'$, new vertices whose neighbourhoods are contained in $X'_t \cap X'_{t'}$.
Therefore, $(T',\X')$ is an $(\eta,\theta,\F,\F')$-construction of $H$.

For every $t \in V(T')$, let $X''_t = X'_t-Z$.
Let $\X''=(X''_t: t\in V(T'))$.
So $(T',\X'')$ is a tree-decomposition of $H-Z$ of adhesion at most $\theta$.
Since for every $t \in V(T_0)$, $X_t \cap Z \neq \emptyset$ and $Z=N_G^{\leq 3\ell}(X_{t^*})$, we know that for every $e \in E(T_0)$, $X_e \cap Z \neq \emptyset$.
Note that $X''_{t^*} = X_{t^*}-Z=\emptyset$.
If $\eta-1=0$, then let $T'''=T'$ and $\X'''=\X''$; otherwise,  let $t_0$ be a node of $T'$ with $X''_{t_0} \neq \emptyset$ closest to $t^*$, let $v_0$ be a vertex in $X''_{t_0}$, let $T'''$ be the rooted tree obtained from $T'$ by adding a new node $t_0^*$ adjacent to $t^*$, where $t_0^*$ is the root of $T'''$, and let $\X'''=(X'''_t: t \in V(T'''))$, where $X'''_{t_0^*}=\{v_0\}$, $X'''_t=X''_t \cup \{v_0\}$ if $t \neq t_0^*$ and $t$ is in the path in $T'$ between $t^*$ and $t_0$, and $X'''_{t}=X''_t$ otherwise.

Then $(T''',\X''')$ is a tree-decomposition of $H-Z$.
Since $(T',\X')$ is an $(\eta,\theta,\F,\F')$-construction of $H$, and $G$ is connected, and $\F,\F'$ and $\F^+$ are hereditary, $(T''',\X''')$ is an
$(\eta-1,\theta,\F,\F')$-construction of $H-Z$ as desired.
\end{claimproof}

Let $S = \bigcup_{e \in U_E} X_e$.
Recall that $G_1=\bigcup_{e\in U_E} G_e$.
So $S=V(G_0)\cap V(G_1)$.
For $i \in [3]$, let $Z_i = N_{G_1}^{\leq i\ell}(S)$.  

\begin{clm}
\label{claim5}
For every $e \in U_E$, $i \in [3]$ and $v \in Z_i \cap X_{T_e}-X_e$, there exists $Y \in \P_e$ such that $v \in N_{G_e}^{\leq i\ell}(Y)$, and for every $Y' \in \P_e-\{Y\}$, $v \not \in N_{G_e}^{\leq i\ell}(Y')$.
\end{clm}
\begin{claimproof}
Since $v \in Z_{i} \cap X_{T_e}-X_e$, there exists a path $P$ in $G_1$ from $v$ to $X_e$ internally disjoint from $X_e$ of length at most $i\ell$.
Since $P$ is internally disjoint from $X_e$, $P$ is a path in $G_e$.
Let $y$ be the vertex in $V(P) \cap X_e$.
Let $Y$ be the member of $\P_e$ containing $y$.
So $v \in N_{G_e}^{\leq i\ell}(Y)$.
Let $Y'$ be any member of $\P_e-\{Y\}$.
If $v \in N_{G_e}^{\leq i\ell}(Y')$, then there exists a walk in $G_e$ from $Y$ to $Y'$ of length at most $2i\ell \leq 6\ell$, so $Y=Y'$ by the definition of $\P_e$, a contradiction.
Hence $v \not \in N_{G_e}^{\leq i\ell}(Y')$.
\end{claimproof}
For every $v \in Z_1-V(H) = Z_1-S \subseteq Z_{3}-S$, there is a unique $e_v\in U_E$ with $v\in Z_3\cap X_{T_{e_v}}-X_{e_v}$, hence
by Claim \ref{claim5}, there exists a unique  pair $(e_v,Y_v)$ with $e_v \in U_E$ and $Y_v \in \P_e$ such that $v \in N_{G_{e_v}}^{\leq 3\ell}(Y_v)$.

Let $c_H$ be the function from Claim \ref{claim4}. This assigns a colour to all vertices from $V(H)$ and in particular all vertices from $V(G_0)$. 
We now use the structure of $H$ to extend this colouring of $G_0$ to $G_1$.
Let $c_3: Z_3 \rightarrow [m]$ be the function such that 
	\begin{itemize}
		\item $c_3(u)=c_H(u)$ for every $u \in Z_1 \cap V(H) = S$, 
		\item $c_3(u) = c_H(v_{Y_u})$ for every $u \in Z_1-V(H) = Z_1-S$, 
		\item $c_3(u)=1$ for every $u \in Z_2- Z_1$, and 
		\item $c_3(u)=2$ for every $u \in Z_3-Z_{2}$.
	\end{itemize}
	
For every $e \in U_E$, let 
$Z_e = Z_3 \cap V(G_e)$, and let $c_e: Z_e \rightarrow [m]$ such that $c_e(v)=c_3(v)$ for every $v \in Z_e$.
By Claim \ref{claim5}, for every $e \in U_E$, $Z_e \subseteq N_{G_e}^{\leq 3\ell}(X_e)$.\\

\begin{clm}
\label{claim 6}
For every $e \in U_E$, $c_e$ can be extended to an $m$-colouring $c_e'$ of $G_e^\ell$ with weak diameter in $G_e^\ell$ at most $f^*(\eta)$.
\end{clm}
\begin{claimproof}
We aim to apply the induction hypothesis to $G_e$.
Let $e \in U_E$.
If $\lvert X_e \rvert > \eta$, then since $(T,\X)$ is an $(\eta, \theta,\F,\F')$-construction of $G$, $\lvert V(G_e) \rvert \leq \lvert X_e \rvert+1 \leq \theta+1$, so $c_e$ can be extended to an $m$-colouring of $G_e^\ell$ with weak diameter in $G_e^\ell$ at most $\lvert V(G_e) \rvert \leq \theta+1 \leq f^*(\eta)$.

So we may assume that $\lvert X_e \rvert \leq \eta$.
We may also assume that $X_e \neq \emptyset$, since if $X_e = \emptyset$, then $V(G_e)=\emptyset$ (since $G$ is connected and $X_{t^*} \neq \emptyset$).
Define $Q_e$ to be the rooted tree obtained from $T_e$ by adding a node $r_e$ adjacent to the end of $e$ in $V(T_e)$, where $r_e$ is the root of $Q_e$.
Let $W_{r_e} = X_e$; for every $t \in V(T_e)$, let $W_t=X_t$.
Let $\W=(W_t: t \in V(Q_e))$.
Then $(Q_e,\W)$ is a rooted tree-decomposition of $G_e$ of adhesion at most $\theta$ such that $\lvert W_{r_e} \rvert = \lvert X_e \rvert \leq \eta$.
So if $tt' \in E(Q_e)$ with $\lvert W_t \cap W_{t'} \rvert > \eta$, then $tt' \in E(T_e)$, so $W_t=X_t$ and $W_{t'}=X_{t'}$.
Since $G[W_{r_e}]=G[X_e]$ and $\F$ and $\F'$ are hereditary, $(Q_e,\W)$ is an $(\eta,\theta,\F,\F')$-construction of $G_e$.

Note that $I(Q_e) = \{r_e\} \cup (I(T) \cap V(T_e))$.
Note that every vertex that belongs to the shortest directed path in $T$ containing $t^*$ and an end of $e$ belongs to $I(T)-V(T_e)$.
So $\lvert I(Q_e) \rvert \leq \lvert I(T) \rvert$, and equality holds only when $t^*$ is an end of $e$. 
If $t^*$ is an end of $e$, then since $X_e \neq \emptyset$, and $Z\supseteq X_{t^*}$ (by Claim~\ref{claim_Zbasic}), $X_t \cap Z \supseteq X_t \cap X_{t^*} \neq \emptyset$, where $t$ is the end of $e$ other than $t^*$, so $t \in V(T_0)$, a contradiction.
Hence $\lvert I(Q_e) \rvert < \lvert I(T) \rvert$.

Recall that $Z_e \subseteq N_{G_e}^{\leq 3\ell}(X_e) = N_{G_e}^{\leq 3\ell}(W_{r_e})$.
Hence by the induction hypothesis, $c_e$ can be extended to an $m$-colouring of $G_e^\ell$ with weak diameter in $G_e^\ell$ at most $f^*(\eta)$.
\end{claimproof}

For every $e \in U_E$, let $c_e'$ be the $m$-colouring from Claim \ref{claim 6}.
Define 
\[
c = c_H|_{V(G_0)} \cup \bigcup_{e \in U_E}c'_e.
\]
If $v\in V(G_0)\cap V(G_e)$ for some $e\in U_E$, then $c_e'(v)=c_3(v)=c_H(v)$ and hence  $c$ is well-defined. Moreover, $c|_Z=c_H|_Z=c_Z$. 

To prove this lemma, if suffices to show that $c$ has weak diameter in $G^\ell$ at most $f^*(\eta)$.
Let $M$ be a $c$-monochromatic component of $G^\ell$.  
The rest of the proof is devoted to showing that $M$ has weak diameter  in $G^\ell$ at most $f^*(\eta)$.

\begin{clm} \label{claim_oneZ}
$V(M)$ cannot intersect both $Z_1$ and $(Z_3- Z_2)$.
\end{clm} 

\begin{claimproof}
Suppose that $V(M)$ intersects both $Z_1$ and $Z_3-Z_2$.
Since $M$ is a connected subgraph of $G^\ell$, $M$ would then also intersect $Z_2- Z_1$. 
But then $M$ would contain a vertex $v_3$ in $Z_3- Z_2$ and a vertex $v_2$ in $Z_2- Z_1$.
Since $c(v_3)=2$ and $c(v_2)=1$, $M$ could not be $c$-monochromatic, a contradiction. 
\end{claimproof}

\begin{clm} \label{claim_intersect}
We may assume that for every $e\in U_E$, $V(M) \not \subseteq V(G_e)-Z_1$.
\end{clm} 

\begin{claimproof}
Assume that $V(M)\subseteq V(G_e)-Z_1$ for some $e\in U_E$.
Then every edge in $M \subseteq G^\ell$ is an edge in $G_e^\ell$.
So $M$ is a $c_e'$-monochromatic component in $G_e^\ell$.
Hence the weak diameter  in $G_e^\ell$ of $M$ is at most $f^*(\eta)$.
That is, for any vertices $x,y$ in $M$, there exists a path $P_{x,y}$ in $G_e^\ell$ between $x$ and $y$ of length at most $f^*(\eta)$.
Since $G_e \subseteq G$, $G_e^\ell \subseteq G^\ell$.
So for any vertices $x,y$ in $M$, $P_{x,y}$ is a path in $G^\ell$ between $x$ and $y$ of length at most $f^*(\eta)$.
Hence the weak diameter in $G^\ell$ of $M$  is at most $f^*(\eta)$. 
\end{claimproof}

In particular,  Claim \ref{claim_intersect} shows that for all $e\in U_E$, if $M$ contains a vertex from $V(G_e)-X_e$, then we can assume that it also contains a vertex from $Z_1 \cap (V(G_e)-X_e)$, since the distance between $V(G_e)-Z_1$ and $X_e$ is greater than $\ell$. 
Let 
\[
B = \{e \in U_E: V(M) \cap (V(G_e)-X_e) \neq \emptyset\}.
\]

\begin{clm} \label{claim_G0_and_others}
$V(M) \subseteq V(G_0) \cup \bigcup_{e \in B}(X_{T_e} \cap Z_2)$.
\end{clm}

\begin{claimproof}
Let $e \in B$.  Since $V(M) \cap (V(G_e)-X_e) \neq \emptyset$, we just saw that this implies $V(M) \cap (Z_1-X_e) \cap V(G_e) \neq \emptyset$. 
In particular, $M$ contains a vertex from $Z_1$. 
By Claim~\ref{claim_oneZ}, $V(M) \cap Z_3-Z_2 = \emptyset$.
Since $M \subseteq G^\ell$, $V(M) \cap X_{T_e} \subseteq Z_2 \cap X_{T_e}$.
This proves the claim.
\end{claimproof}

\begin{clm} \label{claim8}
For every $x \in V(M)-V(G_0)$, there exists $e \in B$ with $x \in X_{T_e} \cap Z_2-X_e$, and there exists a unique $Y_x \in \P_e$ such that $x \in N_{G_e}^{\leq 2\ell}(Y_x)$. Moreover, we have $c(M)=c_H(v_{Y_x})$.
\end{clm}

\begin{claimproof}
Let $x \in V(M)-V(G_0)$.
By Claim \ref{claim_G0_and_others}, there exists $e \in B$ such that $x \in X_{T_e} \cap Z_2-X_e$.
So by Claim \ref{claim5}, there exists a unique $Y_x \in \P_e$ such that $x \in N_{G_e}^{\leq 2\ell}(Y_{x})$.
Furthermore, $V(M) \cap (Z_1-X_e) \cap V(G_e) \neq \emptyset$ since $V(M)\cap (V(G_e)-X_e)\neq \emptyset$.
So there exists $x' \in V(M) \cap (Z_1-X_e) \cap V(G_e)$ such that $c(x)=c(x')=c(M)$.
We further choose such a vertex $x'$ such that the distance in $M$ between $x$ and $x'$ is as small as possible.
(Note that it is possible that $x=x'$.)
Let $P$ be a shortest path in $M \subseteq G^\ell$ between $x$ and $x'$.
By the choice of $x'$, $V(P)-\{x'\} \subseteq V(M) \cap X_{T_e} \cap Z_2-Z_1$.
Denote $P$ by  $x_1x_2\dots x_{\lvert V(P) \rvert}$. 
So for every $i \in [\lvert V(P) \rvert-1]$, there exists a path in $G_e$ between $x_i$ and $x_{i+1}$ with length in $G_e$ at most $\ell$. 

By Claim \ref{claim5}, there exists a unique $Y_{x'} \in \P_e$ such that $x' \in N_{G_e}^{\leq \ell}(Y_{x'})$, and for each $v \in V(P)-\{x'\}$, there exists a unique $Y_{v} \in \P_e$ such that $v \in N_{G_e}^{\leq 2\ell}(Y_{v})$.
For any $i \in [\lvert V(P) \rvert-1]$, if $Y_{x_i} \neq Y_{x_{i+1}}$, then there exists a walk in $G_e$ from $Y_{x_i}$ to $Y_{x_{i+1}}$ with length at most $2\ell+\ell+2\ell \leq 5\ell$, so $Y_{x_i}=Y_{x_{i+1}}$ by the definition of $\P_e$, a contradiction.
So $Y_{x_i}=Y_{x_{i+1}}$ for every $i \in [\lvert V(P) \rvert-1]$.
In particular, $v_{Y_x}=v_{Y_{x'}}$.
Hence by the definition of $c_3$, $c(M)=c(x')=c_3(x')=c_H(v_{Y_{x'}})=c_H(v_{Y_{x}})$.
\end{claimproof}
For every $x \in V(M)-V(G_0)$, define $Y_x$ to be the set mentioned in Claim \ref{claim8}.
Let $M'$ be the graph obtained from $M$ by identifying, for each $e \in B$ and $Y \in \P_e$, all vertices $x \in V(M)-V(G_0)$ with $Y_x=Y$ into a single vertex $v_Y$. Observe that since $M$ is connected, $M'$ is also connected.
Note that there is a natural injection from $V(M')$ to $V(H)$ obtained by the identification mentioned in the definition of $M'$.
So we may assume $V(M') \subseteq V(H)$.

\begin{clm} \label{claim10}
$M'$ is contained in a $c_H$-monochromatic component in $H^\ell$.
\end{clm} 

\begin{claimproof}
By Claim \ref{claim8}, for every $x \in V(M)-V(G_0)$, $c(M)=c_H(v_{Y_x})$.
So all vertices in $M'$ have the same colour in $c_H$.
Hence to prove that $M'$ is contained in a $c_H$-monochromatic component in $H^\ell$, it suffices to prove that $M'$ is a connected subgraph of $H^\ell$.
Since $M'$ is connected and $V(M') \subseteq V(H)=V(H^\ell)$, it suffices to prove that $E(M') \subseteq E(H^\ell)$.

Note that for any $e \in U_E$, distinct vertices $x,y \in X_e$ and path $P$ in $G_e$ between $x$ and $y$ internally disjoint from $X_e$ of length at most $\ell$ having at least one internal vertex, the ends of $P$ are contained in the same part (say $Y$) of $\P_e$, so there exists a path $\overline{P}=xv_Yy$ in $H$ of length two between $x$ and $y$; since $P$ has at least one internal vertex, the length of $\overline{P}$ is at most the length of $P$.
Hence, for every path $P$ in $G$ of length at most $\ell$ between two distinct vertices in $V(G_0)$, we can replace each maximal subpath $P'$ of $P$ of length at least two whose all internal vertices are not in $V(G_0)$ by $\overline{P'}$ to obtain a walk $\overline{P}$ in $H$ of length at most the length of $P$ having the same ends as $P$.

Hence if $xy$ is an edge of $M'$ with $x,y \in V(G_0)$, then $x,y \in V(M)$, and since $M \subseteq G^\ell$, there exists a path $P_{xy}$ in $G$ of length at most $\ell$ between $x$ and $y$, so $\overline{P_{xy}}$ is a walk in $H$ of length at most $\ell$ between $x$ and $y$, so $xy \in E(H^\ell)$.

Now assume that there exists an edge $xy$ of $M'$ with $x \in V(G_0)$ and $y \not \in V(G_0)$.
Since $y \not \in V(G_0)$, there exists $y_0 \in V(M)$ with $v_{Y_{y_0}}=y$ such that $xy_0 \in E(M)$, and there exists $e_y \in U_E$ such that $y_0 \in X_{T_{e_y}}-X_{e_y}$.
Since $M \subseteq G^\ell$, there exists a path $P_{xy}$ in $G$ of length at most $\ell$ between $x$ and $y_0$.
Let $y'$ be the vertex in $V(P_{xy}) \cap X_{e_y}$ such that the subpath of $P_{xy}$ between $y_0$ and $y'$ is contained in $G_{e_y}$.
Then $y' \in Y_{y_0}$ by Claim \ref{claim5}.
So $x\overline{P'_{xy}}y'y$ is a walk in $H$ of length at most $\ell$, where $P'_{xy}$ is the subpath of $P_{xy}$ between $x$ and $y'$.
So $xy \in E(H^\ell)$.

Hence every edge of $M'$ incident with a vertex of $V(G_0)$ is an edge of $H^\ell$.

Now assume that there exist $e \in B$ and distinct $Y,Y' \in \P_e$ such that $v_Yv_{Y'} \in E(M')$.
So there exists $ab \in E(M)$ such that $a,b \in V(M) \cap X_{T_e}-X_e$, $Y_a=Y$ and $Y_b=Y'$.
By Claim \ref{claim8}, there exist a path $P_a$ in $G_e$ from $a$ to $Y_a$ with length at most $2\ell$ and a path $P_b$ in $G_e$ from $b$ to $Y_b$ of length at most $2\ell$.
Since $ab \in E(M) \subseteq E(G^\ell)$, there exists a path $P_{ab}$ in $G$ of length at most $\ell$ from $a$ to $b$.
If $V(P_{ab}) \subseteq X_{T_e}$, then $P_a \cup P_{ab} \cup P_b$ is a walk in $G_e$ from $Y_a$ to $Y_b$ of length at most $2\ell+\ell+2\ell<7\ell$, contradicting that $Y_a=Y$ and $Y_b=Y'$ are distinct parts of $\P_e$.
So $V(P_{ab}) \not \subseteq X_{T_e}$.
In particular, there exist distinct $a',b' \in V(P_{ab}) \cap X_e$ such that the subpath $P_{a'}$ of $P_{ab}$ between $a$ and $a'$ and the subpath $P_{b'}$ of $P_{ab}$ between $b$ and $b'$ are contained in $G_e$.
Since the length of $P_{a'}$ and $P_{b'}$ are at most $\ell$, $a' \in Y_a$ and $b' \in Y_b$ by Claim \ref{claim5}.
So $v_Ya'\overline{P'}b'v_{Y'}$ is path in $H$ of length at most the length of $P_{ab}$, where $P'$ is the subpath of $P_{ab}$ between $a'$ and $b'$.
Hence $v_Yv_{Y'} \in H^\ell$.

Finally, assume that there exist distinct $e_1,e_2 \in B$, $Y_1 \in \P_{e_1}$ and $Y_2 \in \P_{e_2}$ such that $v_{Y_1}v_{Y_2} \in E(M')$.
So for each $i \in [2]$, there exists $x_i \in V(M) \cap Z_2 \cap X_{T_{e_i}}-X_{e_i}$ by Claim \ref{claim8} such that $Y_{x_i}=Y_i$, and there exists a path $P_{x_1x_2}$ in $G$ of length at most $\ell$ between $x_1$ and $x_2$.
For each $i \in [2]$, let $y_i$ be the vertex in $V(P_{x_1x_2}) \cap X_{e_i}$ such that the subpath of $P_{x_1x_2}$ between $x_i$ and $y_i$ is contained in $G_{e_i}$.
Then $v_{Y_1}y_1\overline{P'_{x_1x_2}}y_2v_{Y_2}$ is a walk in $H$ of length at most $\ell$, where $P'_{x_1x_2}$ is the subpath of $P_{x_1x_2}$ between $y_1$ and $y_2$.
Therefore, $v_{Y_1}v_{Y_2} \in E(H^\ell)$.

This proves $E(M') \subseteq E(H^\ell)$, and hence $M'$ is contained in a $c_H$-monochromatic component in $H^\ell$.
\end{claimproof}

In particular, the weak diameter of $M'$ in $H^\ell$ is at most $f_1(f^*(\eta-1))$ by Claim~\ref{claim4}.
We shall use this fact to bound the weak diameter in $G^\ell$ of $M$.
We now give a relation between the weak diameter in $G^\ell$ and  in $H^\ell$.

\begin{clm}
\label{claim11}
For every pair of vertices $u,v \in V(G_0)$, if $P$ is a path in $H$ between $u$ and $v$, then there exists a walk $\hat{P}$ in $G$ between $u$ and $v$ of length at most $7\theta\ell\lvert E(P) \rvert$.
\end{clm}
\begin{claimproof}
Let $Q$ be a path in $H$ such that there exists $e \in U_E$ such that $Q$ is from $X_e$ to $X_e$, internally disjoint from $X_e$, and contains at least one internal vertex corresponding to a vertex in $X_{T_e}-X_e$.
Then $Q$ has length two, and there exists $Y_Q \in \P_e$ containing both ends of $Q$.
So there exists a path $\hat{Q}$ in $G_e$ between the ends of $Q$ of length at most $7\theta\ell$ by the definition of $\P_e$.

Let $u,v \in V(G_0) \subseteq V(G) \cap V(H)$.
Let $P$ be a path in $H$ between $u$ and $v$.
Consider each subpath $P'$ of $P$ in which there exists $e_{P'} \in U_E$ such that $P'$ is from $X_{e_{P'}}$ to $X_{e_{P'}}$, internally disjoint from $X_{e_{P'}}$, and contains at least one internal vertex corresponding to a vertex in  $X_{T_{e_{P'}}}-X_{e_{P'}}$, and replace it with $\hat{P'}$ as obtained in the above paragraph. We obtain a walk in $G$ between $u$ and $v$ of length at most $7\theta\ell \lvert E(P) \rvert$. 
\end{claimproof}

For a vertex $u \in V(M)-V(G_0)$ and a path $P$ in $M'$ on at least 2 vertices having $v_{Y_u}$ as an end, we let $h_P(u)$ be the neighbour of  $v_{Y_u}$ in $P$. 
Note that since all the vertices $v_Y$ are pairwise non-adjacent in $H$, it follows that $h_P(u) \in Y_u$.

\begin{clm} \label{claim_hu}
For a vertex $u \in V(M)-V(G_0)$ and a path $P$ in $M'$ on at least 2 vertices having $v_{Y_u}$ as an end, there exists a path in $G$ from $u$ to $h_P(u)$ of length in $G$ at most $(7\theta+2)\ell$.
\end{clm}

\begin{claimproof}
By Claim~\ref{claim8}, there exists $e \in U_E$ such that $u \in (X_{T_e}-X_e) \cap N_{G_e}^{\leq 2\ell}(Y_u)$, so there exists a path $P_u$ in $G_e$ from $u$ to $Y_u$ with length at most $2\ell$.
Let $u'$ be the end of $P_u$ in $Y_u$.
Since $Y_u$ contains both $u'$ and $h_P(u)$, there exists a path in $G$ from $u'$ to $h_P(u)$ of length at most $7\theta\ell$ by the definition of $\P_e$.
So there exists a walk in $G$ from $u$ to $h_P(u)$ of length in $G$ at most $(7\theta+2)\ell$.
\end{claimproof}

\begin{clm} \label{claim_wdM}
For every pair of  vertices $x,y \in V(M)$, there exists a path in $G$ between $x$ and $y$ with length in $G$ at most $f^*(\eta)$. 
\end{clm}

\begin{claimproof}
For every vertex $u \in V(M)$, if $u \in V(G_0)$, then define $h_0(u)=u$; otherwise, define $h_0(u)=v_{Y_u}$.
Let $x,y \in V(M)$.
We may assume that $x \neq y$, for otherwise we are done.
By Claims~\ref{claim4} and \ref{claim10}, there exists a path $P_0$ in $H^\ell$ between $h_0(x)$ and $h_0(y)$ of length in $H^\ell$ at most $f_1(f^*(\eta-1))$.
So there exists a path $P$ in $H$ of length in $H$ at most $\ell \cdot f_1(f^*(\eta-1))$ between $h_0(x)$ and $h_0(y)$.

For every $u \in \{x,y\}$, if $u \in V(G_0)$, then let $h(u)=u$; if $u \not \in V(G_0)$, then let $h(u)=h_{P}(u)$.
For every $u \in \{x,y\}$, let $Q_u$ be a shortest path in $G$ from $u$ to $h(u)$.
Note that if $u \in V(G_0)$, then $Q_u$ has length 0; if $u \not \in V(G_0)$, then by Claim~\ref{claim_hu}, $Q_u$ has length in $G$ at most $(7\theta+2)\ell$.

Since $h(x)$ and $h(y)$ are in $V(G_0)$, and the subpath of $P$ between $h(x)$ and $h(y)$ is a path in $H$ of length in $G$ at most $\ell \cdot f_1(f^*(\eta-1))$, by Claim~\ref{claim11}, there exists a walk $W$ in $G$ between $h(x)$ and $h(y)$ of length in $G$ at most $7\theta\ell \cdot \ell \cdot f_1(f^*(\eta-1))$.
Therefore, $Q_x \cup W \cup Q_u$ is a walk in $G$ between $x$ and $y$ of length in $G$ at most $2(7\theta+2)\ell + 7\theta\ell^2f_1(f^*(\eta-1)) \leq f^*(\eta)$.
\end{claimproof}

By Claim \ref{claim_wdM}, the weak diameter in $G$ of $M$ is at most $f^*(\eta)$.
Therefore, the weak diameter in $G^\ell$ of $M$ is at most $f^*(\eta)$.
This proves the lemma.
\end{proof}

\subsection{Proof of Theorems~\ref{tree_extension_ad} and \ref{tw_ad_intro_2}}

We are now ready to prove Theorem~\ref{tree_extension_ad}, which we restate here for the convenience of the reader.

\begin{repthm}{tree_extension_ad}
Let $\cC$ be a hereditary class of graphs, and let $\theta \geq 1$ be an integer.
Let $\G$ be a class of graphs such that for every $G \in \G$, there exists a tree-decomposition $(T,\X)$ of $G$ of adhesion at most $\theta$, where $\X=(X_t: t \in V(T))$, such that $\cC$ contains all graphs which can be obtained from any $G[X_t]$ by adding, for each neighbour $t'$ of $t$ in $T$, a set of new vertices whose neighbourhoods are contained in $X_t\cap X_{t'}$.
Then $\ad(\G) \leq \max\{\ad(\cC),1\}$.
\end{repthm}

\begin{proof} 
By Proposition~\ref{obs:weakdiameter}, there exists a function $f: {\mathbb N} \rightarrow {\mathbb N}$ such that $\cC$ is $(\ad(\cC)+1,\ell,f(\ell))$-nice for every $\ell \in {\mathbb N}$.
Define $g: {\mathbb N} \rightarrow {\mathbb N}$ to be the function such that for every $x \in {\mathbb N}$, $g(x)=f^*_x(\theta)$, where $f^*_x$ is the function $f^*$ mentioned in Lemma \ref{tree_extension} by taking  $(\ell,N,m,\theta)=(x,f(x),\max\{\ad(\cC ),1\}+1,\theta)$. 

Let $G \in \G$.
So there exists a tree-decomposition $(T,\X)$ of $G$ of adhesion at most $\theta$, where $\X=(X_t: t \in V(T))$, such that $\cC$ contains all graphs which can be obtained from any $G[X_t]$ by adding, for each neighbour $t'$ of $t$ in $T$, a set of new vertices whose neighbourhoods are contained in $X_t\cap X_{t'}$ (in particular for every $t \in V(T)$, $G[X_t] \in \cC$).
Let $t_0$ be a node of $T$ with $X_{t_0} \neq \emptyset$, and let $v_0$ be a vertex in $X_{t_0}$.
Let $T'$ be the rooted tree obtained from $T$ by adding a new node $t_0'$ adjacent to $t_0$, where $t_0'$ is the root of $T'$.
Let $X'_{t_0'}=\{v_0\}$; for every $t \in V(T)$, let $X'_t=X_t$.
Let $\X'=(X'_t: t \in V(T'))$.
Then $(T',\X')$ is a $(\theta,\theta,\cC,\cC)$-construction of $G$.
For every $\ell \in {\mathbb N}$, applying Lemma \ref{tree_extension} by taking $(\ell,N,m,\theta,\F,\F',\eta,Z)=(\ell,f(\ell),\max\{\ad(\cC),1\}+1,\theta,\cC,\cC,\theta,\emptyset)$, $G^\ell$ is $(\max\{\ad(\cC),1\}+1)$-colourable with weak diameter in $G^\ell$ at most $g(\ell)$.

Hence $\G$ is $(\max\{\ad(\cC),1\}+1,\ell,g(\ell))$-nice for every $\ell \in {\mathbb N}$.
By Proposition~\ref{obs:weakdiameter}, $\ad(\G) \leq \max\{\ad(\cC),1\}$.
\end{proof}

We can now prove Theorem \ref{tw_ad_intro_2}, which we also restate here for convenience.

\begin{thm} \label{tw_ad}
For any integer $w$, the class of graphs of treewidth at most $w$ has asymptotic dimension at most 1.
\end{thm}

\begin{proof}
Let $\F$ be the class of graphs of treewidth at most $w$, and let $\cC$ be the class  of graphs that have a vertex-cover of size at most $w+1$.
Note that $\cC$ is a hereditary class and $\ad(\cC) = 0$ by Observation \ref{vc_ad}.

Note that for every graph $G$ of treewidth at most $w$, there exists a tree-decomposition $(T,\X)$ of $G$ of adhesion at most  $w$, where $\X=(X_t: t \in V(T))$, such that  $\cC$ contains all graphs which can be obtained from any $G[X_t]$ by adding, for each neighbour $t'$ of $t$ in $T$, a set of new vertices whose neighbourhoods are contained in $X_t\cap X_{t'}$  (since $X_t$ is a vertex-cover of size at most $w+1$ of such graphs).
Hence by Theorem \ref{tree_extension_ad}, $\ad(\F) \leq \max\{\ad(\cC),1\}=1$.
\end{proof}

\section{Control functions and layerings}\label{sec:control}

In this section we introduce some notation and state a result of Brodskiy, Dydak, Levin and Mitra~\cite{BDLM} which is an extension of a result of Bell and Dranishnikov~\cite{BD06}, and we will use it to derive a result about graph layering.

\subsection{Real projections  and layerings}\label{sec:layer}

Given a metric space $(X,d)$, and a real $c>0$, a function $f: X\to \mathbb{R}$ is \emph{$c$-Lipschitz} if for any
$x,y\in X$, $|f(x)-f(y)|\le c\cdot d(x,y)$ (such functions can be defined between any two metric spaces, but here we will only consider $\mathbb{R}$ as the codomain). 
When $f$ is 1-Lipschitz, we call it a \emph{real projection} of $(X,d)$.

In the context of graphs, an interesting example of real projections comes from layerings. 
Recall that a layering $L=(L_i)_{i \in \N}$ or $L=(L_i)_{i \in \Z}$ 
of a graph $G$ is  an ordered partition of $V(G)$ into (possibly empty) sets (called the layers), such that for any edge $uv$ of $G$, $u$ and $v$ lie in the same layer or in a union of two consecutive layers (i.e.\ in $L_i \cup L_{i+1}$ for some $i$). 
Note that a layering  can also be seen as a function $L: V(G)\to \mathbb{N}$ or $L: V(G)\to \mathbb{Z}$ such that for any edge $uv$, $|L(u)-L(v)|\le 1$.
In particular, a layering can be seen as a real projection by the triangle inequality.

\subsection{\texorpdfstring{$r$}\ -components and \texorpdfstring{$(r,s)$}\ -components}

Let $(X,d)$ be a metric space. Recall that a subset $S\subseteq X$ is
$r$-bounded if for any $x,x'\in S$, $d(x,x')\le r$.  
Given a subset $A\subseteq X$, two points $x,x' \in A$
are \emph{$r$-connected}\footnote{This definition should not be confused with the usual definition of $k$-connected graphs in graph theory (graphs on at least $k+1$ vertices that remain connected after the deletion of any set of at most $k-1$ vertices). Since we talk here about $r$-connected points, or vertices, instead of graphs, we hope there is no danger of that.} in $A$ if there are points $x_1=x,x_2,\dots,x_\ell=x'$ in $A$, for some $\ell \in {\mathbb N}$, such
that for any $1\le i \le \ell-1$, $d(x_i,x_{i+1}) \leq r$. 
A maximal set of
$r$-connected points in $A$ is called an \emph{$r$-component} of $A$.
Note that these $r$-components form a partition of $A$. 
Observe that in an unweighted graph $G$, the 1-components of a subset $U\subseteq V(G)$ of vertices are exactly the vertex-sets of the connected components of $G[U]$, the subgraph of $G$ induced by $U$.
(See Figure~\ref{fig:rscomponent}(a), for an example of 2-components of a subset $A$ of vertices.)

\begin{figure}[htb]
 \centering
 \includegraphics[scale=0.8]{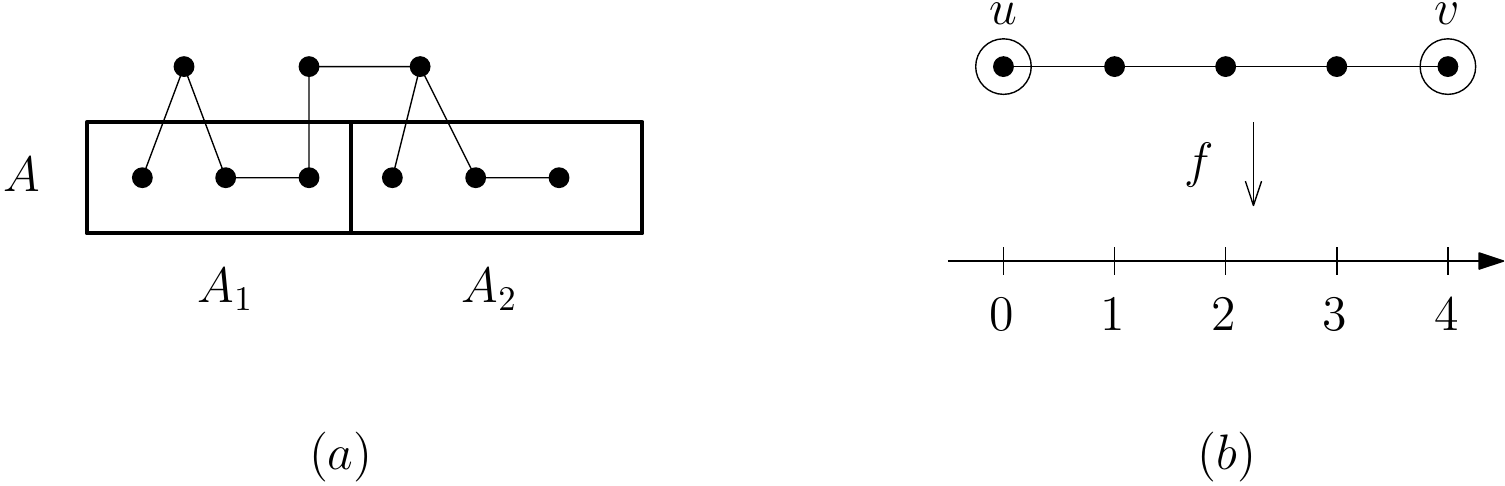}
 \caption{(a) The two 2-components $A_1$ and $A_2$ of a subset $A$, and (b) an example of a set $B=\{u,v\}$ which has one 4-component but two $(4,3)$-components.}
 \label{fig:rscomponent}
\end{figure}

Recall that $D_X:\mathbb{R}^+\to
\mathbb{R}^+$ is an
$n$-dimensional control function for $X$ if for any $r>0$, $X$ has a cover
$\mathcal{U}=\bigcup_{i=1}^{n+1}\mathcal{U}_i$, such that each
$\mathcal{U}_i$ is $r$-disjoint and each element of $\mathcal{U}$ is
$D_X(r)$-bounded. 
Observe that $D_X:\mathbb{R}^+\to \mathbb{R}^+$ is an $n$-dimensional control function for $X$ if and only if for any $r>0$, $X$ is a union of
$n+1$ sets whose $r$-components are $D_X(r)$-bounded. 
In this
section it will be convenient to work with this definition of control
functions.

\medskip

Let $(X,d)$ be a metric space and let $f: X\to \mathbb{R}$ be a real projection of $X$.
A subset $A\subseteq  X$ is said to be \emph{$(r,s)$-bounded with respect to
$f$ (and $d$)} if for all $x,x'\in A$
we have $d(x,x')\leq r$ and $|f(x)-f(x')|\leq s$
(when $f$ is clear from the context we often omit ``with respect to $f$'', and similarly for $d$).
Two vertices $x,x'$ of $A$
are {\em $(r,s)$-connected} in $A$ if there are vertices  $x_1=x,x_2,\dots,x_\ell=x'$ in $A$, for some $\ell \in {\mathbb N}$, such that for any $1\le i \le \ell-1$, $\{x_i,x_{i+1}\}$ is $(r,s)$-bounded (i.e.\  $d(x_i,x_{i+1}) \leq r$ and $|f(x)-f(x')|\leq s$). 
A maximal set of
$(r,s)$-connected vertices in $A$ is called an {\em $(r,s)$-component} of $A$.
Note that these $(r,s)$-components form a partition of $A$.

Note that by the definition of a real projection, any
$r$-bounded set is also $(r,r)$-bounded, and similarly being
$r$-connected is equivalent to being $(r,r)$-connected, and
an $r$-component is the same as an $(r,r)$-component. Observe that in Figure~\ref{fig:rscomponent}(b), the vertices $u$ and $v$ are 4-connected in $\{u,v\}$ but they are not $(4,3)$-connected in $\{u,v\}$ with respect to $f$, where $f$ is defined to be the function $f(x)=d(u,x)$ for any $x$. This shows that the notions of $r$-components and $(r,s)$-components differ when $s<r$.

\subsection{Control functions for real projections}\label{sec:controlfun}

We have seen above that $D_X:\mathbb{R}^+\to
\mathbb{R}^+$ is an $n$-dimensional control function for a metric space $(X,d)$ if for any $r>0$, $X$  is a union of $n+1$ sets whose $r$-components are $D_X(r)$-bounded. 
It will be convenient to extend this definition to real projections, as follows. 
For a metric space  $(X,d)$ and a real projection $f:X\to \mathbb{R}$, we say that $D_f:\mathbb{R}^+\times \mathbb{R}^+\to
\mathbb{R}^+$ is an \emph{$n$-dimensional control function for $f$} if for any real numbers $r,S>0$, any $(\infty,S)$-bounded subset $A\subseteq X$  is a union of $n+1$ sets whose $r$-components are $D_f(r,S)$-bounded. 
We say that the control function $D_f$ is \emph{linear} if there are constants $a,b,c>0$ such that $D_f(r,S)\le ar+bS+c$ for any  $r,S>0$. 
We say that the control function $D_f$ is a \emph{dilation}  if there are constants
$a,b>0$ such that $D_f(r,S)\le ar+bS$ for any  $r,S>0$.

The following is a special case of a combination of Proposition 4.7 in~\cite{BDLM} and Theorem 4.9 in~\cite{BDLM}.
\footnote{Note that every real projection defined in this paper is a large-scale uniform function, as defined in \cite[Definition 3.4]{BDLM}, having the identity function as  coarseness control function (in particular the control function is a dilation).}

\begin{thm}[\cite{BDLM}] \label{thm:mthm49}
Let $(X,d)$ be a metric space and $f:X\to \mathbb{R}$ be a real  projection of $(X,d)$. 
If $f$ admits an $n$-dimensional control function $D_f$, then $X$ admits an $(n+1)$-dimensional control function $D_X$ such that $D_X(r)$ only depends on $r$, $D_f$ and $n$. 
Moreover, if $D_f$ is linear then $D_X$ is also linear; if $D_f$ is
a dilation then $D_X$ is also a dilation.
\end{thm}

\subsection{Intrinsic control of real projections}\label{sec:intrinsic}

Let $(G,\phi)$ be a weighted graph, and let $A$ be a subset of $V(G)$.
The \emph{weighted subgraph of $(G,\phi)$ induced by $A$} is the weighted graph $(G[A],\phi|_{E(G[A])})$.

Let $(G,\phi)$ be a weighted graph.
Let $L:V(G)\to \mathbb{R}$ be a real projection of $(V(G),d_{(G,\phi)})$.
A function $D_L:\mathbb{R}^+\times \mathbb{R}^+\to \mathbb{R}^+$ is an \emph{$n$-dimensional intrinsic control function} for $L$ if for all $r,S>0$, for any maximal $(\infty,S)$-bounded set $A$ of $(V(G),d_{(G,\phi)})$ with respect to $L$ and $d_{(G,\phi)}$, $A$ is a union of $n+1$ sets whose $r$-components are $D_L(r,S)$-bounded, where the definitions of $r$-components and $D_L(r,S)$-bounded are with respect to the metric $d_{(G[A],\phi|_{E(G[A])})}$.

Since every (unweighted) graph can be viewed as a weighted graph whose weight on each edge is 1, the definition for intrinsic control function is also defined for graphs $G$ and corresponding function $L$.

\smallskip

As before, we say that an intrinsic control function $D_L$ for a real projection is \emph{linear} if there are constants $a,b,c>0$ such that $D_L(r,S)  \leq ar+bS+c$, for any real numbers $r,S>0$.
We also say that $D_L$ is \emph{a dilation} if there are constants $a,b>0$ such that $D_L(r,S) \leq  ar+bS$, for any real numbers  $r,S>0$.

\smallskip

We now prove that intrinsic control functions can be transformed into (classical)
control functions.

\begin{lemma}\label{lem:intrinsic}
Let $(G,\phi)$ be a weighted graph.
Let $L:V(G)\to \mathbb{R}$ be a real projection of $(V(G),d_{ (G,\phi)})$ such that $L$ admits an $n$-dimensional intrinsic control function $D$. 
Then $D'(r,S):=D(r,S+2r)$ for any $r,S>0$ is an $n$-dimensional control function for $L$. 
In particular, if $D$ is linear, then $D'$ is also linear, and if $D$ is a dilation, then $D'$ is a dilation.
\end{lemma}

\begin{proof}
Note that we may assume that $D$ is a non-decreasing function.
  Let $S>0$ be a real number, and let $X$ be an $(\infty,S)$-bounded subset of vertices of $G$ with respect to $L$ and  $d_{(G,\phi)}$. 
  It follows that there is some $a>0$ such that for any $x\in X$, $L(x)\in [a,a+S]$. 
  Fix some $r>0$, and denote by $X^+$ the preimage of the interval $[a-r,a+S+r]$ under $L$. 
  Note that $X^+\supseteq X$ and $X^+$ is a maximal $(\infty,r')$-bounded subset of $V(G)$ with respect to $L$ and  $d_{(G,\phi)}$ for some real number $r' \leq S+2r$.
  Let $H^+$ be  $(G[X^+],\phi|_{E(G[X^+])})$. 
  Then by definition of $D$, $(H^+,d_{H^+})$ has a cover by $n+1$ sets $U_1^+, U_2^+, \dots, U_{n+1}^+$, whose $r$-components are $D(r,S+2r)$-bounded  with respect to $d_{H^+}$. 
  For each $1\le i \le n+1$, let $U_i = {U}_i^+ \cap X$. 
  It follows that $\bigcup_{i=1}^{n+1}{U}_i = X$. 
  
  Consider now an $r$-component $C$ of ${U}_i$, for  some $1\le i \le n+1$ (where
  the distance in the definition of $r$-components  is with respect to
  the metric  $d_{(G,\phi)}$). 
  For any $u,v\in C$, there are $u_0=u,u_1,\ldots,u_t=v$ in $U_i$, for some $t \in {\mathbb N}$, such that for any $0\le j \le t-1$,  $d_{(G,\phi)}(u_j,u_{j+1})\le r$. 
It follows that $u_j$ and $u_{j+1}$ are two vertices in $X$ connected by a path $P_j$ (of length at most $r$) in $G$,  so $V(P_j) \subseteq X^+$, and thus  $u_j$ and $u_{j+1}$ are also connected by $P_j$ in $H^+= (G[X^+],\phi|_{E(G[X^+])})$. 
Hence, $u_j$ and $u_{j+1}$ lie in the same $r$-component of $U_i^+$ (where the distance in the definition of $r$-component is with respect to $d_{H^+}$). 
Since all $r$-components of $U_i^+$ (with respect to $d_{H^+}$) are $D(r,S+2r)$-bounded with respect to $d_{H^+}$, they are also $D(r,S+2r)$-bounded with respect to $d_G$.
This shows that $D'(r,S):=D(r,S+2r)$ is an $n$-dimensional control function for $L$.  
\end{proof}

\subsection{Layerability}\label{sec:layerable}

In this section, we will consider graph layerings (or more generally real projections) such that any constant number of consecutive layers induce a graph from a ``simpler'' class of graphs. For instance, any $d$-dimensional grid has a layering in which each layer induces a $(d-1)$-dimensional grid, and moreover any constant number of consecutive layers induce a graph that is a $d$-dimensional grid where the height in one dimension is bounded and hence is significantly simpler than a general $d$-dimensional grid.

Given a class $\mathcal{C}$ of weighted graphs and a sequence of classes
$\mathcal{L}=(\mathcal{L}_i)_{i\in \mathbb{N}}$ of weighted graphs with $\mathcal{L}_1\subseteq\mathcal{L}_2\subseteq\cdots$, we say
that $\mathcal{C}$ is \emph{$\mathcal{L}$-layerable} if there is a function $f:\mathbb{R}^+\to \mathbb{N}$ such that any weighted graph $(G,\phi)$ from ${\mathcal C}$ has a real projection $L: V(G)\to \mathbb{R}$ such that for any $S>0$, any maximal $(\infty,S)$-bounded set in $(G,\phi)$ with respect to $L$ and $d_{ (G,\phi)}$ induces a weighted graph from $\mathcal{L}_{f(S)}$. 
If there is a constant $c>0$ such that $f(S)\le \max\{cS,1\}$ for any $S>0$, we say that
$\mathcal{C}$ is \emph{$c$-linearly $\mathcal{L}$-layerable}.

\begin{thm}\label{thm:bd}
Let $\mathcal{L}=(\mathcal{L}_1, \mathcal{L} _2,\ldots)$ be a sequence of classes of weighted graphs of asymptotic dimension at most $n$.
Let $\mathcal{C}$ be an $\mathcal{L}$-layerable class of weighted graphs. 
Then the asymptotic dimension of $\mathcal{C}$ is at most $n+1$. 
  
In addition, assume moreover that the class $\mathcal{C}$ is $c$-linearly $\mathcal{L}$-layerable for some $c > 0$. 
  If there exist $a,b,d \geq 0$ such that each class $\mathcal{L}_i$ has an $n$-dimensional control function $D_i$ with $D_i(r)\le a r+b(i-1)+d$ for every $r>0$, then $\mathcal{C}$ has asymptotic dimension at most $n+1$ of linear type. 
  If moreover $d=0$, then $\mathcal{C}$ has Assouad-Nagata dimension at most $n+1$.
\end{thm}

\begin{proof}
For any $i\ge 1$, let $D_i$ be an $n$-dimensional control function for the graphs  in $\mathcal{L}_i$.
Since $\L_1 \subseteq \L_2 \subseteq \cdots$, we may assume that for every $x$, $D_1(x) \leq D_2(x) \leq \cdots$. 
Since ${\mathcal C}$ is $\L$-layerable, there exists a function $f: {\mathbb R}^+ \rightarrow {\mathbb N}$ such that for any weighted graph $(G,\phi) \in \mathcal{C}$, there exists a real projection $L_{(G,\phi)}: V(G)\to \mathbb{R}$ such that for any $S>0$, any maximal $(\infty,S)$-bounded set in $(G,\phi)$ with respect to $L_{(G,\phi)}$ and $d_{(G,\phi)}$ induces a weighted graph in $\mathcal{L}_{f(S)}$.

Note that $D(r,S):=D_{f(S)}(r)$ is an $n$-dimensional intrinsic control function for  $L_{(G,\phi)}$ for any $(G,\phi) \in {\mathcal C}$. 
By Lemma~\ref{lem:intrinsic}, $D'(r,S):=D(r,S+2r)$ is an
$n$-dimensional control function for $L_{(G,\phi)}$ for any $(G,\phi) \in {\mathcal C}$. 
By Theorem~\ref{thm:mthm49},  for every $(G,\phi) \in {\mathcal C}$, $(G,\phi)$ admits an $(n+1)$-dimensional control function $D_{(G,\phi)}$ such that $D_{(G,\phi)}(r)$ only depends on $r$, $D$ and $n$, and hence only depends on $r,f, (D_i)_{i \in {\mathbb N}}$ and $n$. 
Let $(G_0,\phi_0) \in {\mathcal C}$.
Then $D_{(G_0,\phi_0)}$ is an $n$-dimensional control function for all members of ${\mathcal C}$.
This shows that $\mathcal{C}$ has asymptotic dimension at most $n+1$. 

If the class $\mathcal{C}$ is $c$-linearly $\mathcal{L}$-layerable for some $c>0$, then it holds that $D'(r,S)=D(r,S+2r) \leq D_{ \max\{c(S+2r),1\}}(r)$. 
If moreover  there exist $a,b,d \geq 0$ such that $D_i(r)\le a r+b(i-1)+d$ for every $i \in {\mathbb N}$ and $r>0$, then 
\[
D'(r,S)\le a r+b \cdot (\max\{c(S+2r),1\}-1)+d \leq ar+b \cdot c(S+2r)+d= (a+2bc) r+bcS+d,
\]
which means that $D'$ is linear. In this case it follows from
Theorem~\ref{thm:mthm49} that $D_{(G,\phi)}$ is linear, and thus $\mathcal{C}$ has asymptotic dimension at
most $n+1$ of linear type.

If moreover $d=0$,  we have $D'(r,S)\le
  (a+2bc)r+bcS$. By Theorem~\ref{thm:mthm49},
$\mathcal{C}$ has Assouad-Nagata
dimension at
most $n+1$.
\end{proof}

\section{Layered treewidth and minor-closed families}\label{sec:minor}

We now apply the tool developed in the previous section to prove that graphs of bounded layered treewidth have asymptotic dimension at most 2 (Theorem~\ref{layered_tw_ad_intro}).

\begin{lemma} \label{layerable}
For any integer $w>0$, the class of graphs of layered treewidth at most $w$ is $(\L_i)_{i\in \N}$-layerable, where for every $i \in {\mathbb N}$, $\L_i$ is the class of graphs of treewidth at most $i$.
\end{lemma}

\begin{proof}
Let $w>0$.
Let $f:{\mathbb R}^+ \rightarrow {\mathbb N}$ be the function such that $f(x)= \lceil wx \rceil$ for every real $x$.
Let $G$ be a graph with layered treewidth at most $w$.
So it has a tree-decomposition $(T,\X)$ and a layering ${\mathcal V}=(V_1,V_2,\dots)$ such that each bag of $(T,\X)$ intersects each layer of ${\mathcal V}$ in at most $w$ vertices. 
Note that for any $\ell>0$ and any set $U$ which is a union of at most
$\ell$ consecutive layers, the intersection of $U$ and any bag of
$(T,\X)$ has size at most $w\ell$, so the treewidth of $G[U]$ is at most $w\ell\le f(\ell)$.
That is, $G[U] \in \L_{f(\ell)}$.

Let $L: V(G) \rightarrow {\mathbb R}$ such that for every $v \in V(G)$, $L(v)$ is the index such that $v \in V_{L(v)}$. 
So $L$ is a real projection of $(V(G),d_G)$.
For any $\ell>0$, any $(\infty,\ell)$-bounded set with respect to $L$ and $d_G$ is a subset of a union of at most $\ell$ consecutive layers, so it induces a graph in $\L_{f(\ell)}$.
Therefore, $\F$ is $(\L_i)_{i \in {\mathbb N}}$-layerable.
\end{proof}

Now we are ready to prove Theorem~\ref{layered_tw_ad_intro}, which we restate here.

\begin{thm} \label{layered_tw_ad}
For any integer $w>0$, the class of graphs of layered treewidth at most $w$ has asymptotic dimension at most 2.
\end{thm}

\begin{proof}
Let $\F$ be the class of graphs of layered treewidth at most $w$.
For every integer $i>0$, let $\L_i$ be the class of graphs of treewidth at most $i$.
By Theorem~\ref{tw_ad}, $\ad(\L_i) \leq 1$ for each $i \in {\mathbb N}$.
By Lemma~\ref{layerable}, $\F$ is $(\L_i)_{i \in {\mathbb N}}$-layerable.
So by Theorem~\ref{thm:bd}, $\ad(\F) \leq 2$.
\end{proof}

To prove Theorem~\ref{thm:minorch}, we need the following lemma.

\begin{lemma} \label{small_exten_minor}
Let $p>0$ be an integer.
For every integer $x>0$, let $\F_x$ be the class of graphs of layered treewidth at most $x$.
Let $\W$ be the class of graphs such that for every $G \in \W$, $G$ can be obtained from a graph $G' \in \F_p^{+p}$ by adding new vertices and edges incident with these vertices, where the neighbourhood of each new vertex is contained in a clique in $G'$. 
Then $\W \subseteq \F_{p+1}^{+p}$.
\end{lemma}

\begin{proof}
Let $G \in \W$.
So there exists $H \in \F_p^{+p}$ such that $G$ can be obtained from $H$ by adding new vertices whose neighbourhoods are contained in cliques of  $H$.
Since $H \in \F_p^{+p}$, there exists $Z \subseteq V(H)$ with $\lvert Z \rvert \leq p$ such that $H-Z \in \F_p$.
Hence $G-Z$ can be obtained from $H-Z$ by adding new vertices such that for each vertex $v \in V(G)-(V(H) \cup Z)=V(G)-V(H)$, the neighbourhood of $v$ is contained in a clique $C_v$ in $H-Z$.

Since $H-Z \in \F_p$, there exist a layering $(V_1,V_2,\dots)$ of $H-Z$ and a tree-decomposition $(T,\X=(X_t: t \in V(T))$ of $H-Z$ such that the intersection of any $V_i$ and any $X_t$ has size at most $p$.
For every $v \in V(G)-V(H)$, since $C_v$ is a clique in $H-Z$, there exist $t_{v} \in V(T)$ with $C_v \subseteq X_{t_{v}}$ and an integer $i_{v}>0$ such that $C_v \subseteq V_{i_{v}} \cup V_{i_{v}+1}$.

For each integer $i>0$, define $V_i' = V_i \cup \{v \in V(G)-V(H): i_{v}=i\}$.
Then $(V_1',V_2',\dots)$ is a layering of $G-Z$.
Let $T'$ be the tree obtained from $T$ by adding, for every $v \in V(G)-V(H)$, a new node $t_v'$ adjacent to $t_{v}$.
For every $t \in V(T)$, define $X'_t = X_t$; for every $t \in V(T')-V(T)$, $t=t_v'$ for some (unique) $v \in V(G)-V(H)$, and we define $X'_t = C_v \cup \{v\}$.
Let $\X'=(X'_t: t \in V(T'))$.
Then $(T',\X')$ is a tree-decomposition of $G-Z$.

Let $i>0$ be an integer, and let $t \in V(T')$.
If $t \in V(T)$, then $X_t' \subseteq V(H)-Z$, so $\lvert X'_t \cap V'_i \rvert = \lvert X_t \cap V_i \rvert \leq p$.
If $t \in V(T')-V(T)$, then there exists $v \in V(G)-V(H)$ such that $t=t_v'$, so $X'_t \cap V_i' = (C_v \cap V_i) \cup (\{v\} \cap V_i') \subseteq (X_{t_{v}} \cap V_i) \cup \{v\}$, and hence $\lvert X'_t \cap V_i' \rvert \leq p+1$.

Therefore, the layered treewidth of $G-Z$ is at most $p+1$.
So $G-Z \in \F_{p+1}$.
Hence $G \in \F_{p+1}^{+p}$.
This shows $\W \subseteq \F_{p+1}^{+p}$.
\end{proof}

We can now prove Theorem~\ref{thm:minorch}. 
The following is a restatement.

\begin{thm}\label{thm:minorch2}
For any graph $H$, the class of $H$-minor free graphs has asymptotic dimension at most 2.
\end{thm}

\begin{proof}
Let $\F$ be the class of $H$-minor free graphs.
For every integer $x>0$, let $\F_x$ be the class of graphs of layered treewidth at most $x$.
By \cite[Theorem 1.3]{RS03} and \cite[Theorem 20]{DMW17}, there exists an integer $p>0$ such that for every graph $G \in \F$, there exists a tree-decomposition $(T,\X)$ of $G$ of adhesion at most $p$ such that for every $t \in V(T)$, the torso\footnote{Given a tree-decomposition $(T,\X)$ of $G$, where $\X=(X_t: t \in V(T))$, the {\it torso} at $t$ is the graph obtained from $G[X_t]$ by adding edges such that $X_t \cap X_{t'}$ is a clique for each neighbour $t'$ of $t$ in $T$.} at $t$ belongs to $\F_p^{+p}$.

Let $\W$ be the class of graphs such that for every $G \in \W$, $G$ can be obtained from a graph $G' \in \F_p^{+p}$ by adding new vertices and edges incident with these vertices, where the neighbourhood of each new vertex is contained in a clique in $G'$. 
By Lemma \ref{small_exten_minor}, $\W \subseteq \F_{p+1}^{+p}$.

Note that $\F_{p+1}^{+p}$ is closed under taking subgraphs.
Hence for every $G \in \F$, there exists a tree-decomposition $(T,\X)$ of $G$ of adhesion at most $p$, where $\X=(X_t: t \in V(T))$, such that for every $t \in V(T)$, $\F_{p+1}^{+p}$ contains all graphs which can be obtained from any $G[X_t]$ by adding, for each neighbour $t'$ of $t$ in $T$, a set of new vertices whose neighbourhoods are contained in $X_t\cap X_{t'}$.

Therefore, by Theorem \ref{tree_extension_ad}, $\ad(\F) \leq \max\{\ad(\F_{p+1}^{+p}),1\} \leq 2$, where the last inequality follows from Theorems \ref{apex_extension_ad} and \ref{layered_tw_ad}. 
\end{proof}

\section{Assouad-Nagata dimension of bounded layered treewidth graphs}\label{sec:tw}

A natural question is whether the asymptotic dimension can be
replaced by the Assouad-Nagata dimension in Theorems~\ref{thm:minorch}, \ref{tw_ad_intro_2}, or \ref{layered_tw_ad_intro}. 
We now show that the answer to this stronger question is negative in the case of layered treewidth. 
That is, Theorem~\ref{layered_tw_ad_intro} cannot be extended by replacing asymptotic dimension by Assouad-Nagata dimension.

Given a graph $G$ and an integer $k \geq 0$, the \emph{$k$-subdivision} of $G$, denoted by $G^{(k)}$, is the graph obtained from $G$ by replacing each edge of $G$ by a path on $k+1$ edges.
We start with the following simple observation.

\begin{lemma}\label{obs:subdv}
Let $k \geq 0$ and $n \geq 0$ be integers.
Let $c>0$ be a real number.
Let $D: {\mathbb R}^+ \rightarrow {\mathbb R}^+$ with $D(r) \leq cr$ for every $r>0$.
Let $G$ be a graph.
If $D$ is an $n$-dimensional control function of $G^{(k)}$, then the function $f(x):=cx$ for every $x>0$ is an $n$-dimensional control function of $G$.
\end{lemma}

\begin{proof}
Fix some real $r>0$. 
Let $U_1,\ldots,U_{n+1}$ be subsets of $V(G^{(k)})$ whose $(k+1)r$-components are $D((k+1)r)$-bounded such that $\bigcup_{i=1}^{n+1}U_i=V(G^{(k)})$.
For any $1 \leq i \leq n+1$, let $U_i'=U_i \cap V(G)$.
Each $r$-component of $U_i'$ (for some $1\le i \le n+1$) in $(V(G),d_G)$ is contained in a $(k+1)r$-component of $U_i$ in $(V(G^{(k)}),d_{G^{(k)}})$.
Thus each $r$-component of $U_i'$ is $D((k+1)r)$-bounded in $(V(G^{(k)}),d_{G^{(k)}})$.
Note that $D((k+1)r)  \leq c\cdot (k+1)r$ and for any two vertices $u$ and $v$ in $G$, we have $d_{G^{(k)}}(u,v)=(k+1)\cdot d_G(u,v)$. 
It follows that for any $1\le i \le n+1$, each $r$-component of $U_i'$ is $cr$-bounded. 
So the function $f(x):=cx$ for any $x>0$ is an $n$-dimensional control function of $G$.
\end{proof}

Note that Lemma~\ref{obs:subdv} does not assume any independence between $k$ and $G$, and in particular we will later apply it with $k=|E(G)|$.

Recall that a graph is \emph{1-planar} if it has a drawing in the plane so that each edge contains at most one edge-crossing (see~\cite{PT97} for more details on $k$-planar graphs). 
Note that Corollary~\ref{gk_planar_ad_intro} directly implies that the class of 1-planar graphs has asymptotic dimension 2. We now prove that the analogous result does not hold for the Assouad-Nagata dimension.

\begin{lemma} \label{01_unbounded_AN}
There is no integer $d$ such that the class of 1-planar graphs has Assouad-Nagata dimension at most $d$.
\end{lemma}

\begin{proof}
Take any family $\mathcal{F}$ of graphs of unbounded
Assouad-Nagata dimension (for instance, grids of increasing size and dimension).
For each graph $G \in \mathcal{F}$, let $G^*:=G^{(|E(G)|)}$.

Note that
$G^*$ is $1$-planar (this can be seen by placing the vertices of $G$ in general position in the plane, joining adjacent vertices in $G$ by straight-line segments, and then subdividing each edge at least once between any two consecutive
crossings involving this edge). Define $\mathcal{F}^*:=\{G^*\,|\,G
\in \mathcal{F}\}$, and note that all the graphs of  $\mathcal{F}^*$
are $1$-planar. 

Suppose to the contrary that there exists an integer $d$ such that the class of $1$-planar graphs has Assouad-Nagata dimension at most $d$.
Then $\F^*$ has Assouad-Nagata dimension at most $d$.
So there exists a $d$-dimensional control function $D$ of all graphs in $\F^*$ such that $D$ is a dilation.
Hence there exists a constant $c>0$ such that $D(r) \leq cr$ for every $r>0$, and for every $G \in \F$, $D$ is a $d$-dimensional control function of $G^*=G^{(\lvert E(G) \rvert)}$.
Let $f: {\mathbb R}^+ \rightarrow {\mathbb R}^+$ be the function with $f(x)=cx$ for every $x \in {\mathbb R}^+$.
By Lemma~\ref{obs:subdv}, for every $G \in \F$, $f$ is a $d$-dimensional control function of $G$.
Hence $f$ is a $d$-dimensional control function of $\F$.
But $f$ is a dilation, so the Assouad-Nagata dimension of $\F$ is at most $d$, a contradiction.
\end{proof}

Note that \cite[Theorem 3.1]{DEW17} implies that every $1$-planar graph has layered treewidth at most $12$.
So Lemma \ref{01_unbounded_AN} implies a weaker version of Theorem~\ref{thm:nofptw_intro} that replaces the number 1 by $12$.
Inspired by a comment of an anonymous referee of an earlier version of this paper, the number 12 can be dropped to 1 by using a standard argument about quasi-isometry and \cite[Lemma 3]{BDJMW}.

Two metric spaces $(X,d_X)$ and $(Y,d_Y)$ are \emph{quasi-isometric} if there is a map $f: X \rightarrow Y$ and constants $\epsilon\ge 0$, $\lambda\ge 1$, and $C\ge 0$ such that for any $y\in Y$ there is $x\in X$ such that $d_Y(y,f(x))\le C$, and for every $x_1,x_2\in X$, $$\frac1{\lambda}d_X(x_1,x_2)-\epsilon\le d_Y(f(x_1),f(x_2))\le \lambda d_X(x_1,x_2)+\epsilon.$$
It is not difficult to check that the definition is symmetric.  
Moreover, if for every $r> 0$, $X$ has a cover by $n$ sets whose $r$-components are $D_X(r)$-bounded and there exists a map $f: X \rightarrow Y$ as above, then for every $r  > 0$, $Y$ has a cover by $n$-sets whose $r$-components are $D_Y(r)$-bounded, where $D_Y$ only depends on $D_X$ and the constants $\lambda$, $\epsilon$, and $C$ in the definition of $f$. 
Moreover, $D_X$ is linear if and only if $D_Y$ is linear.
This implies that asymptotic dimension (of linear type) is invariant under quasi-isometry. 
Moreover,  if each member of a family  $\mathcal{X}$ of metric spaces is quasi-isometric to some metric space  in a family ${\mathcal Y}$ of metric spaces, with uniformly bounded constants $\lambda$, $\epsilon$, and $C$ in the definition of the quasi-isometry map, then $\mathrm{asdim}\, \mathcal{X}\le\mathrm{asdim}\,{\mathcal Y}$, and the same holds for the asymptotic dimension of linear type.

\medskip

We can now derive the following, which  is a restatement of Theorem~\ref{thm:nofptw_intro}. 

\begin{corollary}
There is no integer $d$ such that the class of graphs of layered treewidth at most $1$ has asymptotic dimension of linear type at most $d$. 
\end{corollary}

\begin{proof}
Let ${\mathcal Y}$ be the class of graphs of layered treewidth at most $1$.
Suppose to the contrary that there exists an integer $d$ such that ${\mathcal Y}$ has asymptotic dimension of linear type at most $d$.

Let $\X$ be the class of graphs of layered treewidth at most $12$.
By \cite[Lemma 3(a)]{BDJMW} (with $k=12$), for every graph $G$ with layered treewidth at most $12$, $G$ can be made into a graph $G'$ with layered treewidth at most $1$ by replacing each edge by a path on at most 24 vertices.
So the mapping $\iota: V(G) \rightarrow V(G')$ with $\iota(v)=v$ for every $v \in V(G)$ satisfies that for every $y \in V(G')$, there exists $x \in V(G)$ such that $d_{G'}(y,\iota(x)) \leq 11$, and for every $x_1,x_2 \in V(G)$, $\frac{1}{23}d_G(x_1,x_2) \leq d_{G'}(\iota(x_1),\iota(x_2)) \leq 23d_G(x_1,x_2)$.
That is, every member of $\X$ is quasi-isometric to a member of ${\mathcal Y}$ with uniformly bounded constants.
Since ${\mathcal Y}$ has asymptotic dimension of linear type at most $d$, so does $\X$.

Hence ${\mathcal X}$ has a $d$-dimensional control function $D$ defined by $D(x):=cx+c$ for some real number $c>0$.
Since every edge has weight 1, the function $f(x):=2cx$ is a $d$-dimensional control function for  ${\mathcal X}$.
Hence the class of graphs of layered treewidth at most 12 has Assouad-Nagata dimension at most $d$.
By \cite[Theorem 3.1]{DEW17}, every $1$-planar graph has layered treewidth at most $12$, and so the Assouad-Nagata dimension of the class of 1-planar graphs is at most $d$. This contradicts Lemma~\ref{01_unbounded_AN}.
\end{proof}

\section{\texorpdfstring{$K_{3,p}$}\ -minor free graphs}\label{sec:k3p}

\subsection{Terminology}

Let $(G,\phi)$ be a weighted graph.
For two subsets $A,B$ of  $V(G)$, we define
$$d_{(G,\phi)}(A,B):=\min\{d_{(G,\phi)}(u,v)\,|\,(u,v)\in A\times B\}.$$

We say that a subset $S$ of vertices of $G$ is
\emph{connected} if $G[S]$ is connected. 

\subsection{Fat minors}\label{sec:fatminor}

An equivalent way to define minors is the following: a (weighted or unweighted) graph $G$ contains a graph $H$ as a \emph{minor} if $V(G)$ contains $|V(H)|$
vertex-disjoint subsets $\{T_v\,|\,v \in V(H)\}$, each inducing a
connected subgraph in $G$, and such that for every edge $uv$ in $H$,
$T_u$ and $T_v$ are connected by an edge in $G$. 
We will also need the following interesting metric variant of minors: for some integer $q\ge 1$, a weighted graph  $(G,\phi)$ contains a graph $H$ as a \emph{$q$-fat minor} if $V(G)$ contains $|V(H)|$
vertex-disjoint subsets $\{T_v\,|\,v \in V(H)\}$ such that
\begin{itemize}
\item each subset $T_v$ induces a connected subgraph of $G$;
  \item any two sets $T_u$ and $T_v$ are at distance at least $q$ apart in  $(G,\phi)$;
\item for every edge $uv$ in $H$,
  $T_u$ and $T_v$ are connected by a path $P_{uv}$ (of length in $(G,\phi)$ at least $q$) in $G$, such that 
  \begin{itemize}
\item for any pair of distinct edges $uv$ and $xy$ of $H$ (possibly sharing a vertex), the paths $P_{uv}$ and $P_{xy}$ are at distance at least $q$ in $(G,\phi)$, and 
  \item for any edge $uv$ in $H$ and any vertex $w$ distinct from $u$ and $v$, $P_{uv}$ is at distance at least $q$ from $T_w$ in  $(G,\phi)$.  
  \end{itemize}
\end{itemize}
Note that if $(G,\phi)$ contains $H$ as a $q$-fat minor, then $G$ contains $H$
as a minor. An example of $q$-fat $K_3$-minor is depicted in Figure~\ref{fig:fatbanana}, where the last property (stating that sets $T_w$ are far from paths $P_{uv}$) is not mentioned explicitly, for the sake of readability.

\medskip

For a real interval $I$, we say a weighted graph is an \emph{$I$-weighted graph} if the weight of each edge is in $I$. Note that for $(0,1]$-weighted graphs, the assumption that two sets $A$ and $B$ are at distance at least $q$ implies that any path connecting $A$ and $B$ contains at least $q$ vertices.

\medskip

\begin{figure}[htb]
 \centering
 \includegraphics[scale=0.7]{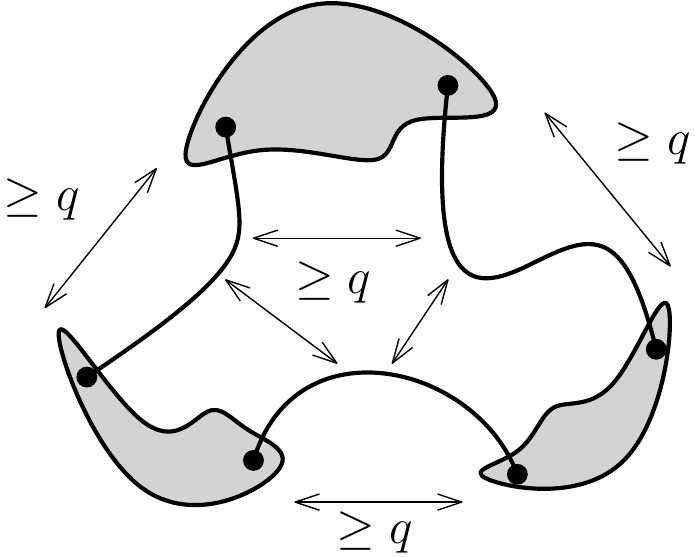}
 \caption{A $q$-fat $K_3$-minor.}
 \label{fig:fatbanana}
\end{figure}

In this section, we will prove a technical lemma on graphs with no $q$-fat $K_{2,p}$-minors, generalising a result of Fujiwara and Papasoglu (Theorem 3.1 in \cite{FP20})  about similar objects called ``thetas" in their paper.

\begin{lemma}\label{lem:thetaminor}
Let $p\ge 2$ and $q\ge 1$ be integers and $r >0,\kappa>0$ be real numbers. 
Let $(G,\phi)$ be a $(0,\kappa]$-weighted graph. If $(G,\phi)$ does not contain the complete bipartite graph $K_{2,p}$ as a $q$-fat minor, then  $(G,\phi)$ has a cover by two sets whose $r$-components are $(5r+9q+9\kappa)p$-bounded.
\end{lemma}

\begin{proof}
Let  $(G,\phi)$ be a $(0,\kappa]$-weighted graph with no $q$-fat $K_{2,p}$-minor. 
We  may assume that  $G$ is connected, for otherwise we can consider each weighted connected component of $(G,\phi)$ separately. 
We fix a root vertex $v\in  V(G)$, and recall that an
\emph{annulus} is a set of the form $A(a,b)=\{u  \in V(G)\,|\, a\le d_{(G,\phi)}(u,v)<b\}$ for some $0<a<b$. 
For any integer $k\geq 1$, let $A_k:=A(kr, (k+1)r)$. For any integer $k$, we say that the annuli $A_k$ and $A_{k+1}$ are \emph{consecutive}.

Let $k_0$ be the smallest integer such that $k_0r\ge r+3q+3\kappa$. 
Note that by definition,
$(k_0-1)r\le r+3q+3\kappa$, and thus $k_0r\le 2r+3q+3\kappa$. 
Let $A_0$ be the set of vertices at distance less than $k_0r$ from $v$. 
Note that $A_0$ has weak diameter in $(G,\phi)$ at most $2k_0r\le 4r+6q+6\kappa\le (5r+9q+9\kappa)p$. 

We define $C_0:=\bigcup \{A_k\,|\, k= k_0+i, \,i \geq 0\text{
  is an even integer}\} $ and define $C_1:=A_0\cup \bigcup \{A_k\,|\, k= k_0+i, \,i  \geq 0\text{  is an odd integer}\} $. 
Note that these two sets clearly cover $V(G)$, so we only
need to show that each $r$-component in $C_0$ or $C_1$ is
$(5r+9q+9\kappa)p$-bounded. By the definition of the annuli $A_k$, observe
that each $r$-component of $C_0$ or $C_1$ is contained in $A_0$ or in
some annulus $A_k$ with $k\ge k_0$. 
Since $A_0$ is $(5r+9q+9\kappa)p$-bounded, it thus suffices to show that for any $k\ge k_0$, each $r$-component of $A_k$ is also $(5r+9q+9\kappa)p$-bounded.

Fix some $k\ge k_0$, and let $C$ be an $r$-component of $A_k$. 
Assume for the sake of contradiction that there exist $x$ and $y$ in $C\subseteq A_k$ with $d_{(G,\phi)}(x,y)> (5r+9q+9\kappa)p$. 
Since $C$ is an $r$-component, there exist $x_1,\dots,x_\ell\in C$  for some integer $\ell \geq 1$, where $x_1=x$, $x_\ell=y$ and
$d_{(G,\phi)}(x_i,x_{i+1})\leq r$ for $1\le i\le \ell-1$.

We now define a function $\iota : \{1,\ldots
,p\}\mapsto \N$, as follows: 
$\iota(1)= 1$, $\iota(p)=\ell$, and for any $1\le i\le p- 2$, $\iota(i+1) = 1+\max\{j \in {\mathbb N}: \iota(i) \leq j \leq \ell-1, d_{(G,\phi)}(x_{\iota(i)},x_{j})\leq 4r+9q+9\kappa\}$. 

Since for any $1\le i \le p-2$ we have
\begin{align*}
d_{(G,\phi)}(x_{\iota(i+1)},x_{\iota(i)})& \le  d_{(G,\phi)}(x_{\iota(i+1)},x_{\iota(i+1)-1})+d_{(G,\phi)}(x_{\iota(i+1)-1},x_{\iota(i)})\\
  &\leq r+(4r+9q+9\kappa) \\
  &\le 5r+9q+9\kappa,
\end{align*}
we obtain that for any $1\le i \le p-1$, $d_{(G,\phi)}(x_{1},x_{\iota(i)})\le
(5r+9q+9\kappa)i$.
As a consequence, for any $1 \leq i \leq p-1$,
\begin{align*}
d_{(G,\phi)}(x_{\iota(i)},x_\ell)&\geq d_{(G,\phi)}(x_1,x_\ell)-d_{(G,\phi)}(x_1,x_{\iota(i)})\\
&\geq (5r+9q+9\kappa)p-(5r+9q+9\kappa)i \ge 5r+9q+9\kappa> 4r+9q+9\kappa.
\end{align*}
This shows that $\iota(i) \neq \ell$ for every $1 \leq i \leq p-1$.
In particular, $\iota$ is a strictly increasing function, and  for any $i,i'$ with $1 \leq i \leq p-1$ and $\iota(i+1) \leq i' \leq \ell$, we have $d_{(G,\phi)}(x_{\iota(i)},x_{i'})> 4r+9q+9\kappa$.
Therefore, $d_{(G,\phi)}(x_{\iota(i)},x_{\iota(j)})> 4r+9q+9\kappa$ for any $1\le i<j \le p$.

\medskip

For each $1\le i \le p$,  let $P_i$ be a shortest path  in $(G,\phi)$ from $x_{\iota(i)}$ to the root $v$.
For each $1 \leq i \leq \ell-1$, let $R_i$ be a shortest path in $(G,\phi)$ from $x_i$ to $x_{i+1}$.
Note that each $R_i$ has length in $(G,\phi)$ at most $r$.
We define the following (see also Figure~\ref{fig:lemma71} for an illustration).
\begin{itemize}
\item Let $A$ be the set of vertices in $\bigcup_{1\le i \le p}
  V(P_i)$   with
  distance in $(G,\phi)$ at most $(k-1)r-(3q+3\kappa)$ from $v$ (note that since
  $k\ge k_0$, we have $kr\ge r+3q+3\kappa$, so $v \in A$ and hence $A \neq \emptyset$).
  \item Let $D$ be the set of vertices in $\bigcup_{1\le i \le p}
  V(P_i)$   with
  distance in $(G,\phi)$ at least $(k-1)r$ from $v$.
\item Let $B = D \cup \bigcup_{i=1}^{\ell-1}V(R_i)$. 
\end{itemize}

\begin{figure}[htb]
 \centering
 \includegraphics[scale=1.2]{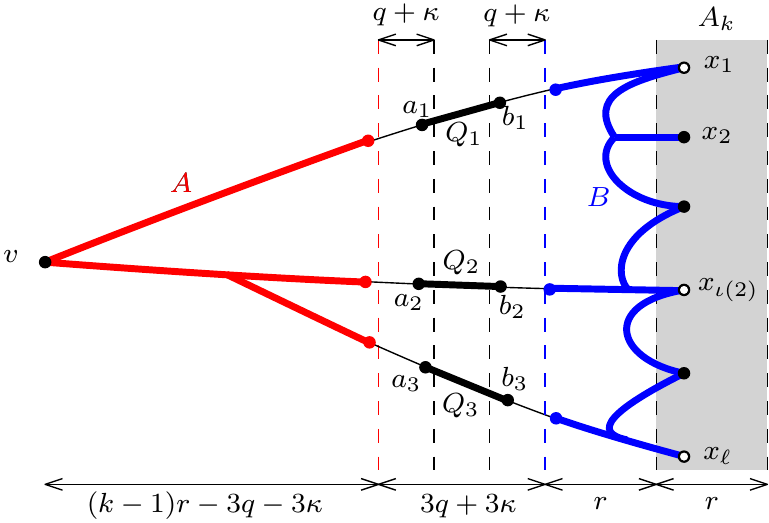}
 \caption{The construction of a $q$-fat $K_{2,p}$-minor with $p=3$ in the proof of Lemma~\ref{lem:thetaminor}. The figure is for illustrating the notation only and should not reflect any level of generality: for instance $r$ can be either much larger or much smaller than $q+\kappa$, and the vertices in $A$ or $B$ closest to $a_i$ or $b_i$ are not necessarily in $P_i$.}
 \label{fig:lemma71}
\end{figure}

Note that $A$ and $B$ induce connected subgraphs of $G$. 
Since the vertices $x_1,x_2,\ldots,x_\ell$ are in $A_k$, they are at distance at least $kr$ from $v$. 
Since each $R_i$ has length at most $r$, $B$ is at distance  in $(G,\phi)$ at least $(k-1)r$ from $v$ and thus at distance  in $(G,\phi)$ at least $(k-1)r-( (k-1)r-(3q+3\kappa) ) =3q+3\kappa$ from $A$. 
Since each $P_i$ is a shortest path between $v$ and $x_{\iota(i)}$, for every $x \in V(P_i)$, we have $d_{(G,\phi)}(v,x) = d_{(P_i,\phi|_{E(P_i)})}(v,x)$ and $d_{(G,\phi)}(x_{\iota(i)},x) = d_{(P_i,\phi|_{E(P_i)})}(x_{\iota(i)},x)$.

For each $1\le i \le p$, 
    \begin{itemize}
        \item let $a_i$ be the vertex in $P_i$ such that the distance in $(G,\phi)$ between $a_i$ and $A$ is at most $q+\kappa$, and subject to this, the subpath of $P_i$ between $a_i$ and $V(P_i) \cap V(B)$ is as short as possible,
        \item let $b_i$ be the vertex in the subpath of $P_i$ between $a_i$ and $V(P_i) \cap B$ such that the distance in $(G,\phi)$ between $b_i$ and $B$ is at most $q+\kappa$, and subject to this, the subpath of $P_i$ between $b_i$ and $V(P_i) \cap V(B)$ is as long as possible,
        \item let $Q_i$ be the subpath of $P_i$ between $a_i$ and $b_i$ (see again Figure~\ref{fig:lemma71} for an illustration),
        \item let $P_i^A$ be a path in $(G,\phi)$ from $A$ to $a_i$ whose length in $(G,\phi)$ equals the distance in $(G,\phi)$ between $A$ and $a_i$, and
        \item let $P_i^B$ be a path in $(G,\phi)$ from $B$ to $b_i$ whose length in $(G,\phi)$ equals the distance in $(G,\phi)$ between $B$ and $b_i$.
    \end{itemize}
Note that each $P_i^A$ and $P_i^B$ have length at most $q+\kappa$.
Since $A$ and $B$ are at distance at least $3q+3\kappa$ apart in $(G,\phi)$, the length of $Q_i$ is at least $(3q+3\kappa)-2(q+\kappa) \geq q+\kappa$, so $Q_i$ contains at least 2 edges and least one internal vertex.

We shall prove that the vertex sets $\{A,B,V(Q_i): 1 \leq i \leq p\}$ together with the connecting paths $\{P_i^A,P_i^B: 1 \leq i \leq p\}$ form a $q$-fat $K_{2,p}$-minor.

Since $(G,\phi)$ is $(0,\kappa]$-weighted, by the minimality of the length of the subpath of $P_i$ between $a_i$ and $V(B) \cap V(P_i)$, the distance in $(G,\phi)$ between $A$ and $V(Q_i)$ is at least $(q+\kappa)-\kappa=q$; similarly, the distance in $(G,\phi)$ between $B$ and $V(Q_i)$ is at least $q$.
In particular, $\{A,B,V(Q_i): 1 \leq i \leq p\}$ is a collection of pairwise disjoint sets inducing connected subgraphs of $G$.

By the definition of $A$, we know that for each $1 \leq i \leq p$ and $x \in V(Q_i)$, since $Q_i$ is disjoint from $A$, $d_{(P_i,\phi|_{E(P_i)})}(v,x) \geq (k-1)r-(3q+3\kappa)$.
For each $1 \leq i \leq p$, since $x_{\iota(i)} \in A_k$, $d_{(P_i,\phi|_{E(P_i)})}(v,x_{\iota(i)}) = d_{(G,\phi)}(v,x_{\iota(i)}) \leq (k+1)r$, so for every $x \in V(Q_i)$, $d_{(G,\phi)}(x_{\iota(i)},x)=d_{(P_i,\phi|_{E(P_i)})}(x_{\iota(i)},x) = d_{(P_i,\phi|_{E(P_i)})}(v,x_{\iota(i)})-d_{(P_i,\phi|_{E(P_i)})}(v,x) \leq (k+1)r-((k-1)r-(3q+3\kappa)) = 2r+3q+3\kappa$.
If there exist $1 \leq i <j \leq p$ such that the distance in $(G,\phi)$ between some vertex $x$ in $Q_i$ and some vertex $y$ in $Q_j$ is less than $3q+3\kappa$, then $d_{(G,\phi)}(x_{\iota(i)},x_{\iota(j)}) \leq d_{(G,\phi)}(x_{\iota(i)},x)+d_{(G,\phi)}(x,y)+d_{(G,\phi)}(y,x_{\iota(j)}) < (2r+3q+3\kappa) + (3q+3\kappa) + (2r+3q+3\kappa) = 4r+9q+9\kappa$, a contradiction.
So any two distinct paths $Q_i,Q_j$ are at distance at least $3q+3\kappa$ in $(G,\phi)$. 
This implies that $\{A,B,V(Q_i): 1 \leq i \leq p\}$ is a collection of sets with pairwise distance in $(G,\phi)$ at least $q$.
Furthermore, this together with the fact that each $P_i^A$ and $P_i^B$ has length at most $q+\kappa$ imply that for distinct $i,j$, the distance between $V(P_i^A) \cup V(Q_i) \cup V(P_i^B)$ and $V(P_j^A) \cup V(Q_j) \cup V(P_j^B)$ is at least $(3q+3\kappa)-2(q+\kappa)>q$.

For any $X \in \{A,B\}$ and $Y \in \{V(P_i^A),V(P_i^B): 1 \leq i \leq p\}$ with $X \cap Y = \emptyset$, the distance between $X$ and $Y$ is at least $2q+2\kappa$, for otherwise, since the length of each of $P_i^A$ and $P_i^B$ is at most $q+\kappa$, the distance between $X$ and $Y$ is less than $3q+3\kappa$, a contradiction.
This together with the fact that each $P_i^A$ and $P_i^B$ has length at most $q+\kappa$ imply that for any (not necessarily distinct) $i,j$, the distance between $P_i^A$ and $P_j^B$ is at least $(2q+2\kappa)-(q+\kappa) \geq q$.
This shows that $\{A,B,V(Q_i): 1 \leq i \leq p\}$ together with $\{P_i^A,P_i^B: 1 \leq i \leq p\}$ form a $q$-fat $K_{2,p}$-minor.
This contradiction completes the proof.
\end{proof}

Note that a graph containing a $q$-fat $K_{2,2}$-minor also contains a chordless cycle of length at least $4q$. 
It follows that graphs without chordless cycles of length at least $4q$ have no $q$-fat $K_{2,2}$-minor, so by Lemma~\ref{lem:thetaminor} (or directly by Theorem 3.1 in \cite{FP20}), for any $\ell\ge 4$ the class of graphs without chordless cycles of length at least $\ell$ has asymptotic dimension 1 (note that the asymptotic dimension is at least 1 because this class contains all trees).

\begin{lemma}\label{lem:layerqfat}
Let $H$ be a graph.
Let $H^*$ be the graph obtained from $H$ by adding a new vertex adjacent to all other vertices.
Let $(G,\phi)$ be a weighted graph such that $G$ has no $H^*$ minor. 
For any vertex $v\in V(G)$, and any real numbers $0< s < t$ and $q > 2(t-s)$, the weighted subgraph of $(G,\phi)$ induced by the vertices $\{u \in V(G)\,|\, s\le d_{(G,\phi)}(u,v)  \leq t\}$ does not contain $H$ as a $q$-fat minor.
\end{lemma}

\begin{proof}
Let $v$ be a fixed vertex of $G$.
Let $A=\{u \in V(G)\,|\,  d_{(G,\phi)}(u,v)<s\}$ and $B=\{u \in V(G)\,|\, s\le d_{(G,\phi)}(u,v) \leq t\}$. 
Let $(G',\phi')$ be the weighted subgraph of $(G,\phi)$ induced by $B$.
Assume for the sake of contradiction that $(G',\phi')$ contains $H$ as a $q$-fat minor. 
Let $\{T_u\,|\,u\in V(H)\}$ and $\{P_{uw}\,|\,uw \in E(H)\}$ be as in the definition of a $q$-fat minor. 
For each $u\in V(H)$, let $P_u$ be a path in $G$ from $T_u$ to $v$  such that the length in $(G,\phi)$ of $P_u$ equals $d_{(G,\phi)}(T_u,\{v\})$, and let $P_u^+$ be the maximal subpath of $P_u$ containing the vertex in $V(P_u) \cap T_u$ and contained in $G[B]$. 
Note that the length  in $(G',\phi')$ of $P_u^+$ is at most $t-s$. 
For any $u\in V(H)$, set $T_u^+:=T_u \cup V(P_u^+)$.

Each of the sets $T_u^+$ is connected and disjoint from $A$, and two sets $T_u^+$ and $T_w^+$ are at distance  in $(G',\phi')$ at least $q-2(t-s)> 0$ (in
particular the sets are pairwise disjoint). 
For each $u \in V(H)$, since the length in $(G,\phi)$ of $P_u$ is $d_{(G,\phi)}(T_u,\{v\}) \leq t$, $V(P_u) \subseteq A \cup B$, so there exists an edge of $P_u$ between $A$ and $T_u^+$.
For any edge $uw$ in $H$, let $P_{uw}^+$ be a minimum subpath of $P_{uw}$ between $T_u^+$ and $T_w^+$. 
Since each path $P_{uw}$ is at distance  in $(G',\phi')$ at least $q - (t-s) > t-s$ from all the
sets $T_x$ with $x \not\in \{u,w\}$, the subpaths $P_{uw}^+$ are disjoint from all sets $T_x^+$ with $x \not\in \{u,w\}$.
Moreover, since the paths  in $\{P_{uw}: uw \in E(H)\}$ are pairwise vertex-disjoint and disjoint from $A$, the paths in $\{P_{uw}^+: uw \in E(H)\}$ are also pairwise vertex-disjoint and disjoint from $A$.
It follows that $G$ contains
$H^*$ as a minor, where $A$ is the connected subset corresponding to the vertex in $V(H^*)-V(H)$, which is a contradiction.
\end{proof}

\subsection{\texorpdfstring{$K_{3,p}$}\ -minor free graphs}

Recall that Fujiwara and Papasoglu~\cite{FP20} proved that unweighted planar
graphs have Assouad-Nagata dimension at most 3. We now prove that the
dimension can be reduced to 2, and the class of graphs can be extended
to all graphs avoiding $K_{3,p}$ as a minor, for any fixed integer
$p\ge 1$. 
Moreover the result below holds for weighted graphs (this
will be used to derive our result on Riemannian surfaces in Section~\ref{sec:surfaces}). 

The following observation is obvious.

\begin{obs}\label{obs:subv}
Let $(G,\phi)$ be a weighted graph and let $(G',\phi')$ be obtained from $(G,\phi)$ by subdividing  an edge $e$ of $G$ once (i.e.\ replacing $e$ by a path  with two edges) and assigning weights summing to $\phi(e)$ to the two newly created edges, while all the other edges of $G$ retain their weight from $\phi$. 
Then any $n$-dimensional control function for $(G',\phi')$ is also an $n$-dimensional control function for $(G,\phi)$. 
\end{obs}

\begin{lemma}\label{lem:k3p}
For any integer $p\ge 1$,  the class of weighted graphs with no $K_{3,p}$-minor has asymptotic dimension at most 2.
\end{lemma}

\begin{proof}
Let $p \geq 1$ be an integer.
Since every graph with no $K_{3,1}$-minor has no $K_{3,2}$-minor, we may assume that $p \geq 2$.
Let $\F_1$ be the class of connected weighted graphs with no $K_{3,p}$-minor.
Since a function is an $n$-dimensional control function for a weighted graph $(G,\phi)$ if and only if it is an $n$-dimensional control function for each of the connected components of $(G,\phi)$, to prove this theorem, it suffices to prove that $\F_1$ has a 2-dimensional control function. 

Let $\F_2$ be the class of connected $(0,1]$-weighted graphs with no $K_{3,p}$-minor.
Note that any subdivision of a $K_{3,p}$-minor free graph is $K_{3,p}$-minor free.
So by Observation~\ref{obs:subv}, it suffices to prove that $\F_2$ has a 2-dimensional control function. 

For every integer $i \geq 1$, let $\L_i'$ be the class of $(0,1]$-weighted graphs with no $i$-fat $K_{2,p}$-minor.
By Lemma~\ref{lem:thetaminor},  the function $D_i'(x):=(5x+9i+9)p$ for every $x>0$ is a 1-dimensional control function of $\L_i'$.

For every integer $i \geq 1$, let $\L_i''$ be the class of $(0,1]$-weighted graphs $(G,\phi)$ such that there exists a vertex $v_G \in V(G)$ with $d_{(G,\phi)}(x,v_G) \leq i$ for every $x \in V(G)$.
Hence the function $D_i''(x):=2i$ for every $x>0$ is a 1-dimensional control function of $\L_i''$.

For every integer $i \geq 1$, let $\L_i = \L_i' \cup \L_i''$, and let $D_i(x) = D_i'(x)$ for every $x>0$.
Note that $D_i$ is a  1-dimensional control function of $\L_i$, since $D_i(x) = \max\{D_i'(x),D_i''(x)\}=D_i'(x)$. 
So by Theorem~\ref{thm:bd}, to show that $\F_2$ has a 2-dimensional control function, it suffices to show that $\F_2$ is  $(\L_i)_{i\in \N}$-layerable. 

Define $f: {\mathbb R}^+ \rightarrow {\mathbb N}$ to be the function such that $f(x)=3x$ for every $x>0$.

Let  $(G,\phi)$ be a member of $\F_2$.
Let $v_0$ be a vertex of $G$.
Define $L: V(G) \rightarrow {\mathbb R}$ such that $L(u)=d_{(G,\phi)}(v_0,u)$ for every vertex $u \in V(G)$.
Since $G$ is connected, $L$ is a well-defined real projection.

Let $S>0$ be a real number.
Let $W$ be a maximal $(\infty,S)$-bounded set in $(G,\phi)$ with respect to $L$ and $d_{(G,\phi)}$.
So there exist a real number $k \geq 0$ such that $k \leq d_{(G,\phi)}(v_0,u) \leq k+S$ for every $u \in W$.
By the maximality of $W$, we can indeed assume that $W = \{u \in V(G): k \leq d_{(G,\phi)}(v_0,u) \leq k+S\}$.
Let $H$ be the weighted subgraph of $G$ induced by $W$.

If $v_0 \in W$, then $k=0$, and thus $H \in \L_S'' \subseteq \L_{f(S)}'' \subseteq \L_{f(S)}$.
If $v_0 \not \in W$, then $k>0$; since $K_{3,p}$ is a subgraph of the graph that can be obtained from $K_{2,p}$ by adding a universal vertex, by Lemma~\ref{lem:layerqfat}, $H$ does not contain $K_{2,p}$ as a $3S$-fat minor, so $H \in \L'_{3S} \subseteq \L_{f(S)}$.

Hence $\F_2$ is 3-linearly $(\L_i)_{i=1}^\infty$-layerable.
This proves the lemma.
\end{proof}

We say that a class $\F$ of metric spaces is {\it scaling-closed} if for every $(X,d_X) \in \F$ and every real number $k>0$, the metric space $(X,k\cdot d_X)$ is also in the class $\F$.

\begin{lemma} \label{linear_to_AN}
Let $n \geq 1$ be an integer.
Let $\F$ be a scaling-closed class of (finite or infinite) metric spaces.
If $\F$ has asymptotic dimension at most $n$, then $\F$ has Assouad-Nagata dimension at most $n$.
\end{lemma}

\begin{proof}
Since $\F$ has asymptotic dimension at most $n$, there exists an $n$-dimensional control function $f$ of all metric spaces in $\F$.
Let $g: {\mathbb R}^+ \rightarrow {\mathbb R}^+$ be the function such that $g(x)=f(1)x$ for every $x>0$.
To prove this lemma, it suffices to prove that $g$ is an $n$-dimensional control function of all metric spaces in $\F$.

Let $(X,d_X)$ be a metric space in $\F$.
Let $r \in {\mathbb R}^+$.
Let $k_r = \frac{1}{r}$. 
Since $\F$ is scaling-closed, $(X,k_rd_X) \in \F$.
So $f$ is an $n$-dimensional control function of $(X,k_rd_X)$.
Hence there exist subsets $U_1,U_2,\dots,U_{n+1}$ of $X$ with $\bigcup_{i=1}^{n+1}U_i = X$ such that for each $1 \leq i \leq n+1$, every $k_rr$-component (in the metric $k_rd_X$) of $U_i$ is $f(k_rr)$-bounded (in the metric $k_rd_X$).

Let $S$ be an $r$-component of $U_i$ (in the metric $d_X$) for some $1 \leq i \leq n+1$.
Note that $S$ is contained in a $k_rr$-component of $U_i$ (in the metric $k_rd_X$).
So the weak diameter in $(X,k_rd_X)$ of $S$ is at most  $f(k_rr)=f(1)$.

Let $u,v$ be two points in $S$.
Since the weak diameter in $(X,k_rd_X)$ of $S$ is at most $f(1)$, $k_rd_X(u,v) \leq f(1)$.
Hence $d_X(u,v) \leq \frac{f(1)}{k_r} = f(1)r = g(r)$.

Therefore, the weak diameter in $(X,d_X)$ of $S$ is at most $g(r)$.
That is, $U_1,U_2,\dots,U_{n+1}$ are subsets of $X$ with $\bigcup_{i=1}^{n+1}U_i =X$ such that for each $1 \leq i \leq n+1$, every $r$-component (in the metric $d_X$) of $U_i$ is $g(r)$-bounded (in the metric $d_X$).
Hence $g$ is an $n$-dimensional control function of $\F$.
\end{proof}

We can now prove Theorem~\ref{thm:main}.
The following is a restatement.

\begin{corollary} \label{AN_K3p}
For any integer $p\ge 1$, the class of $K_{3,p}$-minor free weighted graphs has Assouad-Nagata dimension at most 2. 
\end{corollary}

\begin{proof}
It immediately follows from Lemmas~\ref{lem:k3p} and~\ref{linear_to_AN}.
\end{proof}

\section{Surfaces}\label{sec:surfaces}

\subsection{Graphs on surfaces}

Recall that in this paper, a surface is a non-null connected 2-dimensional manifold without boundary.

A compact surface can be orientable
or non-orientable. The \emph{compact orientable   surface of genus~$h$} is obtained by adding $h\ge0$ \emph{handles} to the sphere; while the \emph{compact non-orientable surface of genus~$k$} is formed by adding $k\ge1$ \emph{cross-caps} to the sphere. 
By the Surface Classification
Theorem, any compact surface is one of these two types (up to homeomorphism). 
The {\em Euler genus} of a compact surface $\Sigma$ is defined as twice its genus if $\Sigma$ is orientable, and as
its non-orientable genus otherwise.

A \emph{compact surface with boundary} of Euler genus $g$ is obtained from a compact
surface of Euler genus $g$ by deleting  a finite number of open disks bounded by pairwise disjoint contractible simple closed curves on the surface.

In this section, all the non-compact surfaces we consider are metrisable
(or equivalently triangulable). The Euler genus of a non-compact surface $S$
is the supremum Euler genus of all compact subsurfaces (with boundary) of
$S$ (note that the Euler genus of a non-compact surface might be infinite). There also exists a classification theorem
for non-compact surfaces, but we will not need it in this paper (see Section
3.5 in~\cite{MoTh} for more details about non-compact surfaces).

\begin{corollary}\label{cor:genus2}
For any integer $g\ge 0$,  the class of finite or infinite weighted graphs embeddable  in a surface 
of Euler genus at most $g$ has Assouad-Nagata dimension at most 2.
\end{corollary}

\begin{proof}
By Theorem~\ref{compact_graphs}, we can restrict ourselves to finite  weighted graphs, and we can thus assume that the surface under consideration is compact (with boundary).
For any integer $g\ge 0$, $K_{3,2g+3}$ cannot be embedded in a compact
surface of Euler genus at most $g$ (with or without boundary) by a simple consequence of Euler's Formula (see Proposition 4.4.4 in~\cite{MoTh}).
Since the class of graphs embeddable in a surface of Euler genus  at most $g$ is closed under taking minors, every graph embeddable in a surface of Euler genus at most $g$ is $K_{3,2g+3}$-minor free. 
Therefore, this corollary follows as an immediate consequence of Corollary~\ref{AN_K3p}.
\end{proof}

\subsection{From graphs on surfaces to Riemannian surfaces}

A \emph{Riemannian manifold} $M$ is a manifold $M$ together with a
metric defined by a scalar product on the tangent space of every point.
Recall that a surface is a non-null connected
2-dimensional manifold without boundary.
So a \emph{Riemannian surface} $S$ is a surface $S$ together with
a metric defined by a scalar product on the tangent space of every point. 
An important property that we will need is that for any point $p\in S$, there is a small open neighbourhood $N$ containing $p$ that is \emph{strongly convex}, i.e.\ any two points in $N$ are joined by a unique shortest path.
A Riemannian surface $S$ with metric $d$ is \emph{complete} if the metric space $(S,d)$ is complete. 
By the Hopf-Rinow theorem, in a complete Riemannian surface $S$, subsets that are bounded and closed are compact, and for any two points $p,q \in S$, there is a length-minimising geodesic connecting $p$ and $q$ (i.e.\ a path of length $d(p,q)$).
Note that compact Riemannian surfaces are also complete. 
For  more background on Riemannian surfaces, interested readers are referred to the standard textbook~\cite{Spi99}.

The following result appears to be well known in the area. For instance it can be deduced from the work of Saucan~\cite{Sau09}. 
Here we include a simple proof (suggested to us by Ga\"el Meignez) in dimension 2 for completeness of the paper.

\begin{lemma}\label{lem:rietogr}
Let $S$ be a complete Riemannian surface with distance function $d$. 
Then there is a finite or countable locally finite\footnote{Here {\it locally finite} means that any bounded region of $S$ contains only
finitely many vertices of $G$, and every vertex is incident with finitely many edges.} infinite $(0,1]$-weighted graph $(G,\phi)$ embedded in $S$ such that any point of $S$ is at distance at most 2 from a vertex of $G$ in $S$ and for any vertices $x,y \in V(G)$, $d(x,y) \le d_{(G,\phi)}(x,y) \le 5d(x,y)+2$. 
\end{lemma}

\begin{proof}
We consider an inclusion-wise maximal set $P$ of points of $S$ that
are pairwise at distance at least $\tfrac15$ apart in $(S,d)$. A
simple area computation shows that any bounded region $R$ of
$S$ contains a finite number of points of $P$ (by compactness, the region has bounded area and it follows from the Bertrand–Diguet–Puiseux Theorem~\cite{Ber48,Dig48,Pui48} that there
is a uniform lower bound on the areas of the balls of radius
$\tfrac1{10}$ centered in $P\cap N_{1/10}(R)$, while these balls are pairwise disjoint), 
and thus it follows that $P$ is countable.
By maximality of $P$, the open balls of radius $\tfrac25$
centered in $P$ cover $S$. Let $G'$ be the graph with vertex-set $P$,
in which two points $p,q\in P$ are adjacent if their closed balls of radius
$\tfrac12$ intersect. For each such pair $p,q$, we join $p$ to $q$ by a
shortest path (of length $d(p,q)\le 1$) on the surface $S$. Note that any two such shortest paths intersect in a finite number of points and
segments (since otherwise we could find an arbitrarily small
neighbourhood containing two points joined by two distinct shortest paths, contradicting the property that sufficiently small neighbourhoods are strongly convex), and each such shortest path can intersect only finitely many other such shortest paths. 
For each intersection point between two paths and each end of
an intersecting segment between the paths, we add a new vertex to
$G'$. Let $G$ be the resulting graph (where two vertices of $G$ are adjacent
if they are consecutive on some shortest path between vertices of $G'$). By definition, $G$ is countable and
locally finite, and properly embedded in $S$. Note that each edge $e$ of
$G$ corresponds to a shortest path between the two endpoints of $e$ in $S$ (we denote the length of this shortest path by $\ell_e$). 
So $(G,\phi)$ is a weighted graph, where $\phi$ maps each edge $e$ of $G$ to $\ell_e$.
Note that by definition, all the weights are in the interval $(0,1]$. For any two vertices
$p,q$ in $G$, we clearly have $d(p,q)\le d_{(G,\phi)}(p,q)$. 
Consider now a length-minimising geodesic $\gamma$ between $p$ and $q$ in $S$, and take $k\le 5 d(p,q)+2$ points $p_1,\ldots,p_k$ (in this
order) on $\gamma$, with $p_1=p$, $p_k=q$, and such that  $d(p_i,p_{i+1}) \leq \tfrac15$ for any $1 \leq i \leq k-1$. 
Recall that each point of $S$ is at distance at most $\tfrac25$ from a point of $P=V(G') \subseteq V(G)$. 
For each $i$, we let $r_i$ be a point in $P$ such that $d(p_i,r_i) \leq \tfrac25$.
So for any $1 \leq i \leq k-1$, $d(r_i,r_{i+1}) \leq d(r_i,p_i)+d(p_i,p_{i+1})+d(p_{i+1},r_{i+1}) \leq  \tfrac25+\tfrac15+\tfrac25=1$.
By definition of $G'$, for any $1\le i\le k-1$, $r_i$ and
$r_{i+1}$ coincide or are adjacent in $G'$, so $d_{(G,\phi)}(r_i,r_{i+1})=d(r_i,r_{i+1}) \leq 1$. 
Therefore, $d_{(G,\phi)}(p,q)=d_{(G,\phi)}(p_1,p_k) \leq d_{(G,\phi)}(p_1,r_1)+\sum_{i=1}^{k-1}d_{(G,\phi)}(r_i,r_{i+1})+d_G(r_k,p_k) \leq \tfrac25+(k-1) \cdot 1+\tfrac{2}{5} \leq k \leq 5d(p,q)+2$.
\end{proof}

%Combining the previous lemmas we obtain the following result. 

\begin{thm}\label{thm:riegenus2}
For any integer $g\ge 0$, the class of  complete Riemannian surfaces of Euler genus at most $g$  has Assouad-Nagata dimension at most 2. 
\end{thm}

\begin{proof}
Since $g$ is finite, it suffices to prove that the class $\F$ of complete Riemannian surfaces of Euler genus exactly $g$ has Assouad-Nagata dimension at most 2.
By Lemma~\ref{lem:rietogr}, every complete Riemannian surface $S$ of genus $g$
is quasi-isometric to some weighted finite or infinite weighted graph
embeddable  in $S$, with constants in the quasi-isometry that are
uniform (in fact the constants are even independent of $g$).
Since the asymptotic dimension of the class of weighted finite or infinite graphs embeddable in $S$ is at most 2 by Corollary~\ref{cor:genus2}, $\ad(\F) \leq 2$. 
Since $\F$ is scaling-closed, by Lemma~\ref{linear_to_AN}, the Assouad-Nagata dimension of $\F$ is at most 2.
\end{proof}

\section{Geometric graphs and graphs of polynomial growth}\label{sec:geom}

We now explore the asymptotic dimension of geometric graph classes and graph classes of polynomial growth.

\medskip

For an integer $d\ge 1$, and some real $C\ge 1$, let $\mathcal{D}^d(C)$ be the class
of graphs $G$ whose vertices can be mapped to points of $\mathbb{R}^d$
such that
\begin{itemize}
\item any two vertices of $G$ are mapped to points at (Euclidean) distance at
  least 1 apart, and
\item any two adjacent vertices of $G$ are mapped to points at distance at
  most $C$ apart.
\end{itemize}

The following is a simple consequence of results on coarse equivalence in $\mathbb{R}^d$ (see Section 3.2 in \cite{BD08}), but we include a self-contained proof for the sake of completeness.

\begin{thm}\label{thm:krlee}
For any integer $d\ge 1$ and real $C\ge 1$, the class $\mathcal{D}^d(C)$ has asymptotic dimension at most $d$.
\end{thm}

\begin{proof}
Fix an integer $\ell$ and a graph $G\in \mathcal{D}^d(C)$, together with a mapping $\pi$ of the vertices of $G$ in $\mathbb{R}^d$ as in the definition of $\mathcal{D}^d(C)$. Since $\mathbb{R}^d$ has asymptotic dimension $d$, there is a function $f$ (independent of $G$) and a covering of $\mathbb{R}^d$ by $d+1$ sets $\mathcal{C}_1,\cdots,\mathcal{C}_{d+1}$, such that for any $1\le i \le d+1$, $\mathcal{C}_i$  consists of a union of open sets of diameter at most $f(\ell)$, such that any two open sets in the union are at distance more than $\ell C$ apart. 

For any vertex $v$ of $G$, color $v$ with an arbitrary integer $1\le i \le d+1$ such that $\pi(v)\in \mathcal{C}_i$. Let $F$ be a monochromatic component in $G^\ell$, say of color $i$. Note that for any edge $uv$ in $F \subseteq G^\ell$, $\pi(u)$ and $\pi(v)$ are at Euclidean distance at most $\ell C$ in $\mathbb{R}^d$, and thus $\pi(u)$ and $\pi(v)$ lie in the same open set of $\mathcal{C}_i$. It follows that the images of all the vertices of $F$ lie in the same open set of $\mathcal{C}_i$, of diameter at most $f(\ell)$. Since any two elements of $F$ must be mapped to points at distance at least 1 apart, a simple volume computation shows that $F$ must contain at most $g(d,\ell,C)$ vertices, for some function $g$. Therefore $F$ has diameter at most $g(d,\ell,C)$ in $G^\ell$, which shows that $\mathcal{D}^d(C)$ has asymptotic dimension at most $d$.
\end{proof}

Recall that a graph  $G$ has growth at most $f$ if for every $r>0$, every $r$-ball  in the metric space generated by $G$ contains at most $f(r)$ vertices; a graph class has growth at most $f$ is all graphs in the class have growth at most $f$.
A graph has \emph{growth rate} at most $d$ if it has
growth at most  $f$, where $f(r)=r^d$ (for $r>1$). 
Hence classes of polynomial growth are exactly the classes of graphs whose growth rate is uniformly bounded. 
Krauthgamer and Lee~\cite{KL03} proved that
any graph $G$ with growth rate at most $d$ is in the class
$\mathcal{D}^{O(d\log d)}(2)$.
We thus obtain the following result as a direct corollary of
Theorem~\ref{thm:krlee}. This proves the first part of Theorem~\ref{thm:polygrowth_intro}.

\begin{corollary}\label{cor:polygrowth}
For any real $d> 1$, the class of
graphs of growth rate at most $d$ has asymptotic dimension
$O(d \log d)$. In particular, classes of graphs of polynomial growth
have bounded asymptotic dimension.
\end{corollary}

Again, Corollary~\ref{cor:polygrowth} is a fairly natural extension of the fact that
$d$-dimensional grids have bounded asymptotic dimension
($d$-dimensional grids form the basic example of graphs of polynomial
growth).

Let us now argue why the assumption that the growth is polynomial in
Corollary~\ref{cor:polygrowth} cannot be weakened, showing the second part of Theorem~\ref{thm:polygrowth_intro}.
Recall that a function $f$ is said to be superpolynomial if it can be written as $f(r)=r^{g(r)}$ with $g(r)\to \infty$ when $r\to \infty$.

Given $p\ge 0$, a \emph{$p$-subdivided 3-regular tree} is obtained from a tree in which all internal vertices have degree 3 by subdividing each edge $p$ times.
Given two integers $k \geq 1$ and $p  \geq 1$ and a graph $G$, we say that a graph $G'$ is a \emph{$(k,p)$-stretch} of $G$ if it is obtained from $G$ as follows. 
For each vertex $v\in V(G)$ we  define $T_v$ to be a $p$-subdivided 3-regular tree with $d_{G}(v)$ leaves,  and subject to this, such that the radius of $T_v$ is as small as possible; and each leaf of $T_v$ is indexed by a different neighbour of $v$ in $G$. 
Define $G'$ to be the graph obtained from the disjoint union of these trees $T_v$ (over all $v\in V(G)$) by adding, for each edge $uv$ of $G$, a path with $k$ internal vertices between the leaf of $T_v$ indexed $u$ and the leaf of $T_u$ indexed $v$ (see Figure~\ref{fig:grid} for an illustration).
For any set $S$ of vertices of $G'$, the \emph{projection} of $S$ in $G$ is the set of vertices $v$ of $G$ with $V(T_v)\cap S\ne \emptyset$.   

For every $d \geq 2$, let $\mathcal{G}_d$ be the class of
$d$-dimensional grids. 
Recall that $\mathcal{G}_d$ has asymptotic dimension $d$. 
For $k,p\ge 1$, we now define $\mathcal{G}^{k,p}_d$ to be the class of all $(k,p)$-stretches of graphs from $\mathcal{G}_d$. Note that all vertices $v$ of a graph of $\mathcal{G}_d$ have degree at most $2d$, and thus each tree $T_v$ as above has radius at most $(p+1)\lceil\log_2(2d)\rceil \le 5p\log_2 d$ and diameter at most $2 (p+1)\lceil\log_2(2d)\rceil\le 10 p \log_2 d$ (where the two inequalities follow from $d\ge 2, p\ge 1)$.

\begin{figure}[htb]
 \centering
 \includegraphics[scale=1]{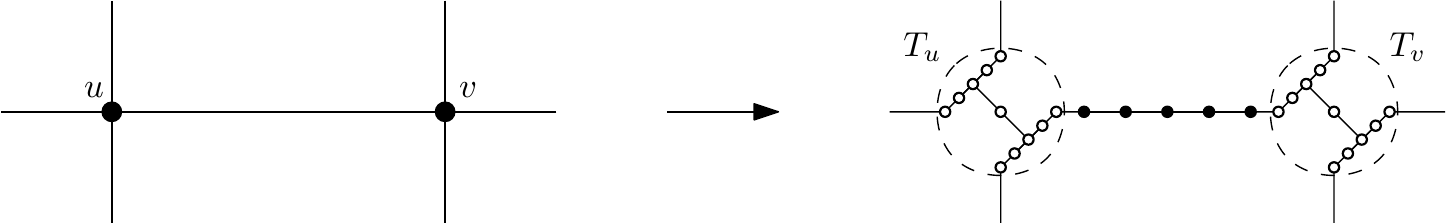}
 \caption{A local view of the $(k,p)$-stretch of a graph, for $k=5$ and $p=1$.}
 \label{fig:grid}
\end{figure}

\begin{lemma} \label{Gdpk_bdd_ad}
For any integers $d  \geq 2,k  \geq 1,p  \geq 1$, the class $\mathcal{G}^{k,p}_d$ has asymptotic dimension at least $d$.
\end{lemma}

\begin{proof}
Assume for a contradiction that there is a function $D'$ such that for every $r \in \mathbb{N}$, every graph in $\mathcal{G}^{k,p}_d$ can be covered by $d$ sets whose $r$-components are $D'(r)$-bounded. 
Consider a graph $G\in \mathcal{G}_d$, and a $(k,p)$-stretch $G' \in \mathcal{G}^{k,p}_d$ of $G$. 
Let $f(k,p,r)=(10p\log_2 d)(r+1)+(k+1)r$.
By assumption, there are $d$ sets $U_1',\ldots, U_d'$ that cover $G'$ so that for each $1\le i \le d$, each $f(k,p,r)$-component of $U_i'$ is $D'(f(k,p,r))$-bounded. For $1\le i \le d$, let $U_i$ be the projection of $U_i'$ in $G$. Note that $U_1, \ldots, U_d$ forms a cover of $G$, and every $r$-component of some $U_i$ is a subset of the projection of a $f(k,p,r)$-component of $U_i'$ in $G$. 
Therefore, for every $i$, every $r$-component of $U_i$ is $D'(f(k,p,r))$-bounded. 
By setting $D(r)=D'(f(k,p,r))$, this shows that $D$ is a $(d-1)$-dimensional control function for $\mathcal{G}_d$, a contradiction.
\end{proof}

We now show that the growth of $\mathcal{G}^{k,p}_d$ can be controlled by tuning the parameters $k$ and $p$.

\begin{lemma} \label{Gdpk_growth}
For every $d,k,p \in \mathbb{N}$ with $p \geq d  \geq 2$ and $k \geq 5p \log_2 d$, the class $\mathcal{G}^{k,p}_d$ has growth at most $r\mapsto 3r+1$ when $r \leq p/2$, $r \mapsto 4dr+1$ when $p/2 < r \leq k$, and $r \mapsto 8dk  \cdot(1+2\lceil r/k\rceil)^d$ when $r > k$.
\end{lemma}

\begin{proof}
Given a graph $G \in \mathcal{G}^{k,p}_d$, a vertex $u \in  V(G)$  and $r>0$, we bound the number $|B_r(u)|$ of vertices at distance at most $r$ of $u$ in $G$. 
Observe that if $B_r(u)$ induces a tree with at most $t$ leaves, then $|B_r(u)|\le tr+1$.

If $r \leq p/2$, then $B_r(u)$ is a tree with at most 3 leaves, and thus $|B_r(u)|\le3 r+1$.
If $p /2 < r \le k$, then $B_r(u)$ is a tree with at most $4d$ leaves (corresponding to the leaves of two adjacent trees $T_v$ and $T_w$, each having at most $2d$ leaves), and thus $|B_r(u)|\le 4dr+1$. 
Assume now that $r > k$. A graph in $\mathcal{G}_d$ has growth rate at most $r \mapsto  (2r+1)^d$. 
Moreover, a graph of $\mathcal{G}_d$ can be obtained from $G$ by contracting trees of radius at most $\tfrac{k+1}2+5p\log_2 d$ with at most $2d$ leaves (corresponding to each tree $T_v$ together with half of each of the $2d$ incident subdivided edges) into single vertices. By the observation above, each such tree contains at most $2d(\tfrac{k+1}2+5p\log_2 d)+1\le 4dk$ vertices. Therefore, the vertex $u$ is at distance at most $r$ from at most $2 \cdot 4dk  \cdot(1+2\lceil r/k\rceil)^d$ vertices in $G$.
\end{proof}

\begin{thm}\label{th:superpoly:nope}
For any superpolynomial function $f$ that satisfies $f(r) \geq 3r+1$ for every $r \in \mathbb{N}$, the class of graphs with growth at most $f$ has unbounded asymptotic dimension.
\end{thm}

\begin{proof}
Consider now a superpolynomial function $f$ with $f(r) \geq 3 r+1$ for every $r$, and fix an integer $d  \geq 2$. Since $f$ is superpolynomial, there exists an integer $p_d \geq d$ such that $f(r) \geq 4dr+1$ for every $r > p_d/2$. Similarly, there exists an integer $k_d \geq 5p_d \log_2 d$ such that $f(r) \geq 8dk_d \cdot (1+2\lceil r/k_d\rceil)^d$ for every $r > k_d$. 
Therefore it follows from  Lemmas~\ref{Gdpk_bdd_ad} and \ref{Gdpk_growth} that the class $\mathcal{G}^{k_d,p_d}_{d}$ has growth at most $f$ and asymptotic dimension at least $d$. 

Hence the class $\bigcup_{d \geq 2}{\mathcal G}^{k_d,p_d}$ has growth at most $f$ but has asymptotic dimension at least $d$ for every $d \geq 2$.
\end{proof}

We note that the intriguing assumption of Theorem~\ref{th:superpoly:nope}, requiring that $f(r)\ge 3r+1$ for every $r  \in {\mathbb N}$, turns out to be necessary. If $f(r_0)\leq 3r_0$ for some $r_0 \in {\mathbb N}$, then graphs with growth at most $f$ do not contain the $(r_0-1)$-subdivision of $K_{1,3}$ as a subgraph and hence as a minor.
So the class of graphs with growth at most $f$ has asymptotic dimension at most 1 by Theorem~\ref{tw_ad_intro_1}. 
This proves the third part of Theorem~\ref{thm:polygrowth_intro}.

\section{Conclusion and open problems}\label{sec:ccl}

For some function $f$, we say that a class of graphs $\mathcal{G}$ has \emph{expansion} at most $f$ if any graph obtained by contracting pairwise disjoint  connected subgraphs of radius at most $r$ in a subgraph of a graph of $\mathcal{G}$ has average degree at most $f(r)$ (see~\cite{NO12} for more details on this notion). 
Note that every proper minor-closed family has constant expansion.
In Section~\ref{sec:geom} we have proved that classes of graphs of polynomial growth have bounded asymptotic dimension. Note that if a class has bounded (polynomial, and superpolynomial, respectively) growth, then it has bounded (polynomial, and superpolynomial) expansion.

\begin{qn}\label{qn:polyexp}
Is it true that every class of graphs of polynomial expansion has bounded asymptotic dimension? 
\end{qn}

Observe that polynomial expansion would again be best possible here, as we have constructed classes of graphs of (barely) superpolynomial growth (and therefore expansion) with unbounded asymptotic dimension.
It should be noted that there are important connections between polynomial expansion and the existence of strongly sublinear separators~\cite{DN16}. 
On the other hand, Hume~\cite{Hum17} proved that classes of graphs of bounded growth and bounded asymptotic dimension have sublinear separators. 

A class of graphs $\cC$ is \emph{monotone} if any subgraph of a graph from $\cC$ lies in $\cC$. 
David Wood asked a weaker form of the converse of Question \ref{qn:polyexp} in private communication after preliminary forms of this paper \cite{BBEGPS,Liu20} appeared in arXiv.
He asked whether every monotone class of graphs of bounded asymptotic dimension has bounded expansion. 
We remark that the monotonicity is required in Wood's question. 
The class consisting of all complete graphs has asymptotic dimension 0 (as all the graphs in this class have diameter at most 2) but does not have bounded expansion. 
This also shows that monotonicity cannot be replaced by the weaker property that the class is hereditary (i.e.\ closed under taking induced subgraphs).

It is known that monotone classes admitting strongly sublinear separators are exactly the monotone classes with polynomial expansions~\cite{DN16}.
Hence one can ask a stronger version of Wood's question as follows.

\begin{qn}\label{qn:convpolyexp}
Is it true that every monotone class of graphs of bounded asymptotic dimension has polynomial expansion?
\end{qn}

Note that positive answers of both Questions~\ref{qn:polyexp} and \ref{qn:convpolyexp} would give a characterization of monotone graph classes with bounded asymptotic dimension in terms of polynomial expansion.
However, during the review process of an earlier version of this paper, some of us (unpublished) were able to find a negative answer to Question \ref{qn:convpolyexp}: for every function $f$, there exists a monotone class of graphs with asymptotic dimension 1 and with expansion greater than $f$.
Note that the control functions in our negative answer for Question \ref{qn:convpolyexp} depend on $f$, so they do not dispute Wood's question for bounded expansion.

\medskip

Lemmas~\ref{lem:thetaminor} shows that for any $p,q$, the class of graphs with no $q$-fat $K_{2,p}$-minors has asymptotic dimension at most 1. On the other hand, Theorem~\ref{thm:minorch} shows that for any $t$, the class of $K_t$-minor free graphs has asymptotic dimension at most 2. A natural question is whether this can be extended to $q$-fat minors, as follows.
\begin{qn}\label{qn:qfat}
Is it true that there is a constant $d$ such that for any integer $q$ and graph $H$, the class of graphs with no $q$-fat $H$-minor has asymptotic dimension at most $d$? 
\end{qn}

\begin{acknowledgement}
This paper is a combination of parts of two manuscripts \cite{BBEGPS,Liu20} in the arXiv repository, where \cite{Liu20} has appeared as an extended abstract in a conference, and a number of results in the unpublished manuscript~\cite{BBEGPS} are strengthened to Assouad-Nagata dimension in this paper.
The authors would like to thank Arnaud de Mesmay, Alfredo Hubard and Emil Saucan for the discussions on the discretization of Riemannian surfaces, and Ga\"el Meigniez for suggesting the simple proof of Lemma~\ref{lem:rietogr} on mathoverflow. The authors would also like to thank David Hume and David Wood for interesting discussions.
The authors thank anonymous referees for their efficient and careful reading and suggestions; in particular, one of them allowed to improve the bound for layered treewidth in Theorem \ref{thm:nofptw_intro} from 12 to 1.
\end{acknowledgement}

\appendix

\section{From finite  metric spaces to infinite metric spaces}

In this section, we use a compactness argument to show that the asymptotic dimension of an infinite weighted graph is bounded from above by the asymptotic dimension of the class of its finite weighted induced subgraphs.
We will use the following result of Gottschalk~\cite{Got51}. A \emph{choice function} for a family of sets $(X_\alpha)_{\alpha\in I}$ is a function $c$ with $c(\alpha)\in X_\alpha$ for any $\alpha\in I$.

\begin{thm}[\cite{Got51}]\label{thm:got}
Let $I$ be a set.
Let  $(X_\alpha)_{\alpha\in I}$ be a family of finite sets,  let $\mathcal{A}$ be the class of all finite subsets of $I$,  and  for  each $A\in \mathcal{A}$, let $c_A$ be a  choice function of $(X_a)_{a\in A}$. 
Then there exists a choice function $c$ of $(X_\alpha)_{\alpha\in I}$ such that for every $A\in \mathcal{A}$ there exists $B\in \mathcal{A}$ with $B\supset A$ such that $ c(\alpha)=c_B(\alpha)$ for any $\alpha \in A$.
\end{thm}

We are now ready to prove the following  theorem by a compactness argument.

\begin{thm} \label{compact_graphs}
Let $(G,\phi)$ be a finite or infinite weighted graph.
Let $\F=\{(G[A],\phi|_{ E(G[A])}): A \subseteq V(G), |A|<\infty\}$.
Let $n \in {\mathbb N}$.
Let $f$ be an $n$-dimensional control function of $\F$.
Let $g: {\mathbb R}^+ \rightarrow {\mathbb R}^+$ be the function such that $g(x)=f(x+1)$ for every real $x>0$.
Then $g$ is an $n$-dimensional control function of $(G,\phi)$.
In particular, $\ad((G,\phi)) \leq \ad(\F)$.
\end{thm}

\begin{proof}
We denote the metric $d_{(G,\phi)}$ by $d$, and for every $A \subseteq V(G)$, we denote the metric $d_{(G[A],\phi|_{E(G[A])})}$ by $d_A$. 
Let $I=V(G)$.
For every $\alpha\in I$, let $X_\alpha=[n+1]$.
Fix a real $r>0$.

Since $f$ is an $n$-dimensional control function of $\F$, for every finite subset $A$ of $I=V(G)$, there exist $(r+1)$-disjoint $f(r+1)$-bounded (with respect to $d_A$) collections $\U_{A,1},\U_{A,2},\dots,\U_{A,n+1}$ of subsets of $A$ such that $\bigcup_{i=1}^{n+1}\bigcup_{U \in \U_{A,i}}U \supseteq A$.
So for every finite subset $A$ of $I$, there exists an $(n+1)$-colouring $c_A$ of $G[A]$ such that no monochromatic $(r+1)$-component contains two vertices $u,v\in A$ with $d_A(u,v)>f(r+1)$.
Note that $c_A$ is a choice function of $(X_a)_{a \in A}$.

Let $c$ be the choice function for $(X_\alpha)_{\alpha\in I}$ given by applying Theorem~\ref{thm:got}.
Since $X_\alpha=[n+1]$ for each $\alpha\in I = V(G)$, $c$ is an $(n+1)$-colouring of $G$.

Suppose for the sake of contradiction that there exist $u,v \in V(G)$ with $d(u,v)>f(r+1)$ and integer $j$ with $1 \leq j \leq n+1$ such that some  $c$-monochromatic $r$-component of colour $j$ contains both $u$ and $v$ (with respect to $d$).
So there exist $\ell \in {\mathbb N}$ and $p_0=u,p_1,p_2,\dots,p_\ell=v$ in $V(G)$ such that $d(p_i,p_{i+1}) \leq r$ for every $0 \leq i \leq \ell-1$, and $c(p_i)=j$ for every $0 \leq i \leq \ell$.
Note that for every $0 \leq i \leq \ell-1$, since $d(p_i,p_{i+1})$ is the infimum of the length of all paths in $G$ between $p_i$ and $p_{i+1}$, there exists a path $P_i$ in $G$ between $p_i$ and $p_{i+1}$ with length in $(G,\phi)$ at most $d(p_i,p_{i+1})+1 \leq r+1$.
Similarly, there exists a path $P$ in $G$ between $u$ and $v$ with length in $(G,\phi)$ at most $d(u,v)+1$.
Let $A = V(P) \cup \bigcup_{i=0}^{\ell-1}V(P_i)$. 
Note that $P$ and each $P_i$ is finite, so $A$ is finite. 
By Theorem~\ref{thm:got}, there
exists a finite subset $B\supset A$ of $V(G)$ such that $c_B(a)=c(a)$ for every $a \in A$.
So $c_B(p_i)=j$ for every $0 \leq i \leq \ell$.
Note that $d_B(p_i,p_{i+1}) \leq d_A(p_i,p_{i+1}) \leq r+1$ for every $0 \leq i \leq \ell-1$.
So $u=p_0$ and $v=p_\ell$ are contained in the same $c_B$-monochromatic $(r+1)$-component with respect to $d_{B}$.
In addition, $d_{B}(u,v) \geq d(u,v)>f(r+1)$, which contradicts the definition of $c_B$. 

Therefore, $c$ is an $(n+1)$-colouring of $(G,\phi)$ such that every  monochromatic $r$-component is $f(r+1)$-bounded with respect to $d$.
Hence $g(x):=f(x+1)$ is an $n$-dimensional control function of $(G,\phi)$.
In particular, if $n=\ad(\F)$, then $\ad((G,\phi)) \leq n$.
\end{proof}

\section{Equivalence between asymptotic dimension and weak diameter colouring} \label{sec:appendix_dia_asdim}

We shall prove Proposition~\ref{obs:weakdiameter} in this section.
The following is a restatement.

\begin{proposition} 
Let $\F$ be a class of graphs.
Let $m>0$ be an integer.
Then $\ad(\F) \leq m-1$ if and only if there exists a function $f: {\mathbb N} \rightarrow {\mathbb N}$ such that for every $G \in \F$ and $\ell \in {\mathbb N}$, $G^\ell$ is $m$-colourable with weak diameter in $G^\ell$ at most $f(\ell)$.
\end{proposition}

\begin{proof}
First observe that for any integer $\ell \geq 1$, $G \subseteq G^\ell$, so a set of vertices with weak diameter in $G$ at most $f(\ell)$ has weak diameter in $G^{\ell}$ at most $f(\ell)$, and conversely a set of vertices with weak diameter in $G^\ell$ at most $f(\ell)$ has weak diameter in $G$ at most $\ell \cdot f(\ell)$.
  
By definition, $\ad(\F) \leq m-1$ implies that there is a function $f: {\mathbb R}^+ \rightarrow {\mathbb R}^+$ such that for any $\ell \in {\mathbb R}^+$ (and hence in particular for any $\ell \in \N$) any graph $G\in\mathcal{F}$ has a cover that is a union of $m$ $\ell$-disjoint families of $f(\ell)$-bounded sets. 
This cover can be reduced to a partition of the vertex-set of $G$ (while maintaining the property that the $m$ families are $\ell$-disjoint and all their elements are $f(\ell)$-bounded). 
We can now consider each of the $m$ families as a distinct colour class in a colouring $c$ of $G$ (and $G^\ell$). 
Since each family is $\ell$-disjoint, each $c$-monochromatic
component in $G^\ell$ is included in a single element of the partition, and is therefore $\lfloor f(\ell) \rfloor$-bounded (since the distances in $G^\ell$ are integral).
This implies that $G^\ell$ is $m$-colourable with weak diameter in $G^\ell$ at most $\lfloor f(\ell) \rfloor$.

Conversely, for every integer $\ell \ge 1$ and $G \in \F$, any $m$-colouring of $G^\ell$ with weak diameter in $G^\ell$ at most $f(\ell)$ clearly defines a cover of $V(G)$ by $m$ $\ell$-disjoint families of $\ell \cdot f(\ell)$-bounded sets in $G$, where the members of the families are the monochromatic components.
Observe that, for any real $\ell > 0$, a set of $\ell$-disjoint families of $G$ is actually $\lfloor \ell \rfloor$-disjoint, and in particular if $\ell < 1$ then the trivial partition of $V(G)$ into singletons is $\ell$-disjoint.
So the domain of $f$ can be extended to ${\mathbb R}^+$, implying that the function $g(x):=x \cdot f(x)$ is an $m$-dimensional control function of $\F$.
\end{proof}

\end{document}